\documentclass[12pt,final]{amsart}
\usepackage{xspace,amssymb,amsfonts,epsfig,euscript,eufrak,enumerate,amsmath,nicefrac}
\usepackage{tikz,etex,bbm,xspace,afterpage}
 \usepackage{epstopdf,ifpdf,todonotes}
\usepackage{pxfonts,multicol}
\usepackage{stackengine}
\def\stacktype{L}
\stackMath
\usepackage{xr,xr-hyper}
\usepackage[notref,notcite]{showkeys}
\usepackage[margin=3cm,footskip=25pt,headheight=20pt]{geometry}

\newcommand{\doubletilde}[1]{
  \tilde{{\tilde{#1}}}}
\usepackage{fancyhdr,mathrsfs,hyperref}
\usepackage[polutonikogreek,english]{babel}

  \usepackage{pdfsync}

\begin{document}
\newtheoremstyle{all}
  {11pt}
  {11pt}
  {\slshape}
  {}
  {\bfseries}
  {}
  {.5em}
  {}

\externaldocument[m-]{merged}[http://arxiv.org/pdf/1309.3796.pdf]

\theoremstyle{all}
\newtheorem{theorem}{Theorem}[section]
\newtheorem*{proposition*}{Proposition}
\newtheorem{proposition}[theorem]{Proposition}
\newtheorem{corollary}[theorem]{Corollary}
\newtheorem{lemma}[theorem]{Lemma}
\newtheorem{definition}[theorem]{Definition}
\newtheorem{ques}[theorem]{Question}
\newtheorem{conj}[theorem]{Conjecture}

\theoremstyle{remark}
\newtheorem{remark}[theorem]{Remark}
\newtheorem{example}[theorem]{Example}

\newcommand{\nc}{\newcommand}
\newcommand{\renc}{\renewcommand}
\newcounter{subeqn}
\renewcommand{\thesubeqn}{\theequation\alph{subeqn}}
\newcommand{\subeqn}{%
  \refstepcounter{subeqn}
  \tag{\thesubeqn}
}\makeatletter
\@addtoreset{subeqn}{equation}
\newcommand{\newseq}{%
  \refstepcounter{equation}
}
  \nc{\kac}{\kappa^C}
\nc{\alg}{T}
\nc{\Lco}{L_{\la}}
\nc{\qD}{q^{\nicefrac 1D}}
\nc{\ocL}{M_{\la}}
\nc{\excise}[1]{}
\nc{\Dbe}{D^{\uparrow}}
\nc{\Dfg}{D^{\mathsf{fg}}}

\nc{\op}{\operatorname{op}}

\nc{\tr}{\operatorname{tr}}
\newcommand{\Mirkovic}{Mirkovi\'c\xspace}
\nc{\tla}{\mathsf{t}_\la}
\nc{\llrr}{\langle\la,\rho\rangle}
\nc{\lllr}{\langle\la,\la\rangle}
\nc{\K}{\mathbbm{k}}
\nc{\Stosic}{Sto{\v{s}}i{\'c}\xspace}
\nc{\cd}{\mathcal{D}}
\nc{\cT}{\mathcal{T}}
\nc{\vd}{\mathbb{D}}
\nc{\R}{\mathbb{R}}
\renc{\wr}{\operatorname{wr}}
  \nc{\Lam}[3]{\La^{#1}_{#2,#3}}
  \nc{\Lab}[2]{\La^{#1}_{#2}}
  \nc{\Lamvwy}{\Lam\Bv\Bw\By}
  \nc{\Labwv}{\Lab\Bw\Bv}
  \nc{\nak}[3]{\mathcal{N}(#1,#2,#3)}
  \nc{\hw}{highest weight\xspace}
  \nc{\al}{\alpha}
\renewcommand{\theequation}{}
\renc{\theequation}{\arabic{section}.\arabic{equation}}
  \nc{\be}{\beta}
  \nc{\bM}{\mathbf{m}}
  \nc{\bkh}{\backslash}
  \nc{\Bi}{\mathbf{i}}
  \nc{\Bj}{\mathbf{j}}
 \nc{\Bk}{\mathbf{k}}

\nc{\bd}{\mathbf{d}}
\nc{\D}{\mathcal{D}}
\nc{\mmod}{\operatorname{-mod}}  
\newcommand{\red}{\mathfrak{r}}

\nc{\RAA}{R^\A_A}
  \nc{\Bv}{\mathbf{v}}
  \nc{\Bw}{\mathbf{w}}
\nc{\Id}{\operatorname{Id}}
  \nc{\By}{\mathbf{y}}
\nc{\eE}{\EuScript{E}}
  \nc{\Bz}{\mathbf{z}}
  \nc{\coker}{\mathrm{coker}\,}
  \nc{\C}{\mathbb{C}}
  \nc{\ch}{\mathrm{ch}}
  \nc{\de}{\delta}
  \nc{\ep}{\epsilon}
  \nc{\Rep}[2]{\mathsf{Rep}_{#1}^{#2}}
  \nc{\Ev}[2]{E_{#1}^{#2}}
  \nc{\fr}[1]{\mathfrak{#1}}
  \nc{\fp}{\fr p}
  \nc{\fq}{\fr q}
  \nc{\fl}{\fr l}
  \nc{\fgl}{\fr{gl}}
\nc{\rad}{\operatorname{rad}}
\nc{\ind}{\operatorname{ind}}
  \nc{\GL}{\mathrm{GL}}
\newcommand{\arxiv}[1]{\href{http://arxiv.org/abs/#1}{\tt arXiv:\nolinkurl{#1}}}
  \nc{\Hom}{\mathrm{Hom}}
  \nc{\im}{\mathrm{im}\,}
  \nc{\La}{\Lambda}
  \nc{\la}{\lambda}
  \nc{\mult}{b^{\mu}_{\la_0}\!}
  \nc{\mc}[1]{\mathcal{#1}}
  \nc{\om}{\omega}
\nc{\gl}{\mathfrak{gl}}
  \nc{\cF}{\mathcal{F}}
 \nc{\cC}{\mathcal{C}}
  \nc{\Mor}{\mathsf{Mor}}
  \nc{\HOM}{\operatorname{HOM}}
  \nc{\Ob}{\mathsf{Ob}}
  \nc{\Vect}{\mathsf{Vect}}
\nc{\gVect}{\mathsf{gVect}}
  \nc{\modu}{\mathsf{-mod}}
\nc{\pmodu}{\mathsf{-pmod}}
  \nc{\qvw}[1]{\La(#1 \Bv,\Bw)}
  \nc{\van}[1]{\nu_{#1}}
  \nc{\Rperp}{R^\vee(X_0)^{\perp}}
  \nc{\si}{\sigma}
  \nc{\croot}[1]{\al^\vee_{#1}}
\nc{\di}{\mathbf{d}}
  \nc{\SL}[1]{\mathrm{SL}_{#1}}
  \nc{\Th}{\theta}
  \nc{\vp}{\varphi}
  \nc{\wt}{\mathrm{wt}}
\nc{\te}{\tilde{e}}
\nc{\tf}{\tilde{f}}
\nc{\hwo}{\mathbb{V}}
\nc{\soc}{\operatorname{soc}}
\nc{\cosoc}{\operatorname{cosoc}}
 \nc{\Q}{\mathbb{Q}}

  \nc{\Z}{\mathbb{Z}}
  \nc{\Znn}{\Z_{\geq 0}}
  \nc{\ver}{\EuScript{V}}
  \nc{\Res}[2]{\operatorname{Res}^{#1}_{#2}}
  \nc{\edge}{\EuScript{E}}
  \nc{\Spec}{\mathrm{Spec}}
  \nc{\tie}{\EuScript{T}}
  \nc{\ml}[1]{\mathbb{D}^{#1}}
  \nc{\fQ}{\mathfrak{Q}}
        \nc{\fg}{\mathfrak{g}}
  \nc{\Uq}{U_q(\fg)}
        \nc{\bom}{\boldsymbol{\omega}}
\nc{\bla}{{\underline{\boldsymbol{\la}}}}
\nc{\bmu}{{\underline{\boldsymbol{\mu}}}}
\nc{\bal}{{\boldsymbol{\al}}}
\nc{\bet}{{\boldsymbol{\eta}}}
\nc{\rola}{X}
\nc{\wela}{Y}
\nc{\fM}{\mathfrak{M}}
\nc{\fX}{\mathfrak{X}}
\nc{\fH}{\mathfrak{H}}
\nc{\fE}{\mathfrak{E}}
\nc{\fF}{\mathfrak{F}}
\nc{\fI}{\mathfrak{I}}
\nc{\qui}[2]{\fM_{#1}^{#2}}
\nc{\cL}{\mathcal{L}}
\nc{\ca}[2]{\fQ_{#1}^{#2}}
\nc{\cat}{\mathcal{V}}
\nc{\cata}{\mathfrak{V}}
\nc{\catf}{\mathscr{V}}
\nc{\hl}{\mathcal{X}}
\nc{\pil}{{\boldsymbol{\pi}}^L}
\nc{\pir}{{\boldsymbol{\pi}}^R}
\nc{\cO}{\mathcal{O}}
\nc{\Ko}{\text{\Denarius}}
\nc{\Ei}{\fE_i}
\nc{\Fi}{\fF_i}
\nc{\fil}{\mathcal{H}}
\nc{\brr}[2]{\beta^R_{#1,#2}}
\nc{\brl}[2]{\beta^L_{#1,#2}}
\nc{\so}[2]{\EuScript{Q}^{#1}_{#2}}
\nc{\EW}{\mathbf{W}}
\nc{\rma}[2]{\mathbf{R}_{#1,#2}}
\nc{\Dif}{\EuScript{D}}
\nc{\MDif}{\EuScript{E}}
\renc{\mod}{\mathsf{mod}}
\nc{\modg}{\mathsf{mod}^g}
\nc{\fmod}{\mathsf{mod}^{fd}}
\nc{\id}{\operatorname{id}}
\nc{\DR}{\mathbf{DR}}
\nc{\End}{\operatorname{End}}
\nc{\Fun}{\operatorname{Fun}}
\nc{\Ext}{\operatorname{Ext}}
\nc{\tw}{\tau}
\nc{\A}{\EuScript{A}}
\nc{\Loc}{\mathsf{Loc}}
\nc{\eF}{\EuScript{F}}
\nc{\LAA}{\Loc^{\A}_{A}}
\nc{\perv}{\mathsf{Perv}}
\nc{\gfq}[2]{B_{#1}^{#2}}
\nc{\qgf}[1]{A_{#1}}
\nc{\qgr}{\qgf\rho}
\nc{\tqgf}{\tilde A}
\nc{\Tr}{\operatorname{Tr}}
\nc{\Tor}{\operatorname{Tor}}
\nc{\cQ}{\mathcal{Q}}
\nc{\st}[1]{\Delta(#1)}
\nc{\cst}[1]{\nabla(#1)}
\nc{\ei}{\mathbf{e}_i}
\nc{\Be}{\mathbf{e}}
\nc{\Hck}{\mathfrak{H}}
\renc{\P}{\mathbb{P}}
\nc{\bbB}{\mathbb{B}}

\nc{\cI}{\mathcal{I}}
\nc{\cG}{\mathcal{G}}
\nc{\cH}{\mathcal{H}}
\nc{\coe}{\mathfrak{K}}
\nc{\pr}{\operatorname{pr}}
\nc{\bra}{\mathfrak{B}}
\nc{\rcl}{\rho^\vee(\la)}
\nc{\tU}{\mathcal{U}}
\nc{\dU}{{\stackon[8pt]{\tU}{\cdot}}}
\nc{\dT}{{\stackon[8pt]{\cT}{\cdot}}}

\nc{\RHom}{\mathrm{RHom}}
\nc{\tcO}{\tilde{\cO}}
\nc{\Yon}{\mathscr{Y}}
\nc{\sI}{{\mathsf{I}}}

\setcounter{tocdepth}{1}

\excise{
\newenvironment{block}
\newenvironment{frame}
\newenvironment{tikzpicture}
\newenvironment{equation*}
}

\baselineskip=1.1\baselineskip

 \usetikzlibrary{decorations.pathreplacing,backgrounds,decorations.markings,shapes.geometric}
\tikzset{wei/.style={draw=red,double=red!40!white,double distance=1.5pt,thin}}
\tikzset{awei/.style={draw=blue,double=blue!40!white,double distance=1.5pt,thin}}
\tikzset{bdot/.style={fill,circle,color=blue,inner sep=3pt,outer
    sep=0}}
\tikzset{dir/.style={postaction={decorate,decoration={markings,
    mark=at position .8 with {\arrow[scale=1.3]{>}}}}}}
\tikzset{rdir/.style={postaction={decorate,decoration={markings,
    mark=at position .8 with {\arrow[scale=1.3]{<}}}}}}
\tikzset{edir/.style={postaction={decorate,decoration={markings,
    mark=at position .2 with {\arrow[scale=1.3]{<}}}}}}\begin{center}
\noindent {\large  \bf Canonical bases and higher representation
  theory}
\medskip

\noindent {\sc Ben Webster}\footnote{Supported by the NSF under Grant DMS-1151473 and  by the NSA under Grant H98230-10-1-0199.}\\  
Department of Mathematics\\ University of Virginia\\
Charlottesville, VA\\
\end{center}
\bigskip
{\small
\begin{quote}
\noindent {\em Abstract.}
This paper develops a general theory of canonical bases, and how they
arise naturally in the context of categorification.  As an application,
we show that Lusztig's canonical basis in the whole quantized
universal enveloping algebra is given by the classes of the
indecomposable 1-morphisms in a 2-category categorifying the universal
enveloping algebra, when the associated Lie algebra is finite type and
simply laced. We also introduce natural categories whose Grothendieck groups
correspond to the tensor products of lowest and highest weight integrable
representations. This generalizes past work of the author's in the
highest wight case. 
\end{quote}
}

\vspace{1cm}

\renc{\thetheorem}{\Alph{theorem}}

\tableofcontents

One of the consistent motivations for the construction of
categorifications has been the accompanying appearance of canonical
bases in the original object under consideration.  At its core, this
is a consequence of a very simple principle: the indecomposable
objects of any Krull-Schmidt category give a basis of its split
Grothendieck group.  Furthermore, any map between Grothendieck groups
which lifts to a functor must have positive integer coefficients in
this basis.  

While this positivity is an appealing consequence, on its own, it has
trouble making up for the difficulty of computing this basis in many
situations. For example, irreducible characters give a basis of class
functions on a finite group where multiplication has positive integral
structure coefficients, but finding irreducible characters is still
very hard in general.

This computational problem eases if this basis has the additional
property of being {\bf canonical}---we will make precise in Section
\ref{sec:pre-canon-struct} what this means.  Defining a canonical basis requires a choice of
{\bf pre-canonical structure} which consists of a bar-involution, a
sesquilinear pairing and a ``first approximation'' to our basis.  Many
readers will be familiar with how these elements induce a canonical basis from work of Kazhdan, Lusztig and others
\cite{KL79,Lusbook}, but we will give a general account making this
definition precise.  If we can prove that the basis coming from our categorification
is canonical (using categorical properties), then we can reduce its
construction to the Gram-Schmidt algorithm.

It's worth noting that previously, the term ``canonical basis'' had
not had a precise meaning in the past (to the author's knowledge), but
was applied to a variety of cases with properties in common.  We'll
give a formal definition of a {\bf canonical basis} (Definition \ref{def:canonical}) which recovers most
examples we're aware of; most
importantly, it will recover the canonical bases defined by Lusztig on
$\dot{U}$ and tensor products in the cases where they are defined in
\cite{Lusbook}.  To avoid confusion, in this paper, we'll refer to the
bases defined in \cite{Lusbook} as {\bf Lusztig's bases} and use the
term ``canonical'' to only mean bases satisfying Definition
\ref{def:canonical}. See Theorem \ref{thm:lusztig} for the precise
connection between these.

In the first two
sections, we'll develop the theory of {\bf humorous categories}, which are
categories well-suited to a connection with a canonical basis. These
categories always have an {\bf orthodox}\footnote{The word ``orthodox'' comes from the Greek
      \textgreek{`orj'os} ``correct'' + \textgreek{d'oxa}
      ``belief''; it is a basis we can believe in.} basis arising from
    their indecomposable objects.  This basis will be canonical when
    the category satisfies an additional
condition on positivity of gradings, which is typically called {\bf
  mixedness} by analogy with mixed structures in geometry.  In
particular, we'll show (Lemmata \ref{lem:full-essential} and
\ref{asymptotic}) how information about mixedness and canonical bases
can pass back and forth between categories and certain special
quotients, and give a useful general principle for extracting mixed
humorous categories from $t$-structures on dg-categories (Lemma
\ref{J-mixed}).  All of these Lemmata are key for understanding the
canonical basis of $\dot{U}$ in categorical terms.

The aim of the rest of the paper is to apply this theory to give an
account of the canonical bases arising in quantized universal
enveloping algebras and their representations.  Our primary tool will
be higher representation theory, as developed by Rouquier, Khovanov,
Lauda and others.  We'll give brief reminder about the
categorification $\dU$ of the universal enveloping algebra itself, and
then define a categorification $\hl^{\bla}$ of a tensor product of a
sequence of highest and lowest integrable representations.  However,
this sequence cannot be in an arbitrary order; in effect, the
categorification forces us to put lowest weight representations on the
left and highest weight on the right.  The significance of this
condition is not clear at the moment, but it matches the underlying
algebra and combinatorics of these representations.  This can be seen in
\cite{BaoWang}, where the existence of a canonical basis in precisely
these tensor products is proven.  These latter categories are generalizations of the
categorifications of highest weight representations defined by the
author in \cite{Webmerged}.

We build on very important results of Vasserot-Varagnolo \cite[4.5]{VV} to show:
\begin{theorem}[\mbox{Theorems \ref{tensor} \&
    \ref{Udot}}]\label{main}\hfill
    \renc{\labelenumi}{(\alph{enumi})}
  \begin{enumerate}
  \item If $\fg$ is finite dimensional and simply-laced (that is, of
    ADE type), then the canonical basis of the modified quantized
    universal enveloping algebra $\dot{U}$ coincides with the classes
    of indecomposable objects in the categorification $\dU$ defined in
    Section \ref{sec:2-category-cu}.
\item     If $\fg$ is an arbitrary Kac-Moody algebra with symmetric Cartan
    matrix, the canonical basis of a tensor product of highest weight
    integrable representations coincides with the classes of
    indecomposable objects in the categorification $\hl^{\bla}$
    defined in Section \ref{sec:tens-prod-algebr}.
  \end{enumerate}
\end{theorem}
Let us reiterate: here ``canonical basis'' refers to Definition
\ref{def:canonical}, but these bases agree with Lusztig's in all cases
where the latter are defined by Theorem \ref{thm:lusztig}.

The result for $\dU$ follows directly from the case of a tensor
product of two dual representations, by Lemma \ref{asymptotic}.  For an
algebra of ADE type, these dual representations are both highest weight
and our result implies that the canonical basis of this tensor product
arises categorically. 

As indicated earlier in the introduction, this sort of result is
particularly interesting because of its positivity consequences:
\begin{corollary}\label{cor:positive}
   If $\fg$ is finite dimensional and simply-laced, then the structure
   coefficients of multiplication of the canonical basis of $\dot{U}$ are all Laurent polynomials with
   positive integer coefficients.  
 The same holds for matrix coefficients in the canonical basis of its action
   on any tensor product
   of finite dimensional modules.
\end{corollary}

Now let us give some indication of why the hypotheses on Theorem
\ref{main} are necessary.  If the Cartan matrix is not symmetric, then Theorem
\ref{main} simply will
not hold: the categorifications we use exist, but their projectives
don't give a canonical basis, since we know that the analogue of
Corollary \ref{cor:positive} fails.

If the Cartan matrix is symmetric, but has infinite type, Theorem
\ref{main}(b) applies, but at the moment we know no proof of the
analogue of Theorem \ref{main}(a). It seems likely that that this
categorical interpretation of the canonical basis will hold both for
$\dot{U}$ and arbitrary tensor products of a group of lowest weight
representations with a group of highest weight representations.  Later
in this paper, we'll define precanonical structures on these spaces
but the techniques in the proof of Theorem \ref{main} do not suffice
to prove that a canonical basis exists in this case, let alone that
such a basis arises from a categorification. The proof of Theorem
\ref{main} uses that highest and lowest weight modules of $\dot{U}$
are the same in a very strong way, so it cannot proceed here.

 A proof of Theorem \ref{main}(a) in
general type will require very
different techniques.
One promising approach would be to apply Lemma
\ref{asymptotic} to the categorical
actions on quantizations of quiver varieties described in
\cite{Webcatq}. However, in order to do so, we must prove
that certain functors are full, and this fullness is equivalent to Kirwan surjectivity for quiver
varieties. This is a long-standing open problem, so until it finds a
solution, we cannot use this approach.  See \cite{JKK} for a more
extensive discussion of the Kirwan surjectivity problem for a general
hyperk\"ahler quotient.

Even if this approach is successful, it can only be applied to the
canonical basis of $1_\la \dot{U} 1_{\la'}$ where $\la$ and $\la'$
both lie in the positive or negative Tits cone $\fg$ (the union of the
Weyl group orbits of the dominant or antidominant Weyl chamber).  A
more promising approach in general would be to apply the techniques of
\cite{EWHodge} in this context, but it remains to be seen if this can
be successful.

\subsection*{Acknowledgements}

I thank George Lusztig; without his questioning, this paper
might never have happened. Ben Elias was also helpful far beyond the
call of duty in giving comments and suggesting revisions.  I also thank Jon
Brundan, Yiqiang Li, Aaron Lauda, Ivan Losev, Marco Mackaay, Catharina
Stroppel, \'Eric Vasserot and Weiqiang Wang for useful discussions.

\subsection*{Notation}

We let $\fg$ be a symmetrizable Kac-Moody algebra.  Consider the
weight lattice $\wela(\fg)$ and root lattice $\rola(\fg)$, and the
simple roots $\al_i$ and coroots $\al_i^\vee$.  We let $\wela^+$
denote the set of dominant weights, and $\wela^-$ the anti-dominant
weights.  Let
$c_{ij}=\al_i^{\vee}(\al_j)$ be the entries of the Cartan matrix.

We let $\langle
-,-\rangle$ denote the symmetrized inner product on $\wela(\fg)$,
fixed by the fact that the shortest root has length $\sqrt{2}$
and $$2\frac{\langle \al_i,\la\rangle}{\langle
  \al_i,\al_i\rangle}=\al_i^\vee(\la).$$ As usual, we let $2d_i
=\langle\al_i,\al_i\rangle$, and for $\la\in\wela(\fg)$, we
let $$\la^i=\al_i^\vee(\la)=\langle\al_i,\la\rangle/d_i.$$

Throughout the paper, we will use $\bla=(\la_1,\dots, \la_\ell)$ to
denote an ordered $\ell$-tuple of dominant or anti-dominant weights, and always use the
notation $\la=\sum_{i}\la_i$.  

We let $U_q(\fg)$ denote the deformed universal enveloping algebra of
$\fg$; that is, the 
associative $\C(q)$-algebra given by generators $E_i$, $F_i$ for
$\al_i$ a simple root
and $K_{\mu}$ for $\mu \in \wela(\fg)$, subject to the relations:
\begin{center}
\begin{enumerate}[i)]
 \item $K_0=1$, $K_{\mu}K_{\mu'}=K_{\mu+\mu'}$ for all $\mu,\mu' \in \wela(\fg)$,
 \item $K_{\mu}E_i = q^{\al_i^{\vee}(\mu)}E_iK_{\mu}$ for all $\mu \in
 \wela(\fg)$,
 \item $K_{\mu}F_i = q^{ -\al_i^{\vee}(\mu)}F_iK_{\mu}$ for all $\mu \in
 \wela(\fg)$,
 \item $E_iF_j - F_jE_i = \delta_{ij}
 \frac{\tilde{K}_i-\tilde{K}_{-i}}{q^{d_i}-q^{-d_i}}$, where
 $\tilde{K}_{\pm i}=K_{\pm d_i \al_i}$,
 \item For all $i\neq j$ $$\sum_{a+b=-c_{ij}+1}(-1)^{a} E_i^{(a)}E_jE_i^{(b)} = 0
 \qquad {\rm and} \qquad
 \sum_{a+b=-c_{ij} +1}(-1)^{a} F_i^{(a)}F_jF_i^{(b)} = 0 $$
\end{enumerate} \end{center}
where $E_i^{(a)}=E_i^a/[a]_q!$, and $[a]_q!=[a]_q[a-1]_q\cdots$ for $[a]_q=(q^a-q^{-a})/(q-q^{-1})$.

This is a Hopf algebra with coproduct on Chevalley generators given
by $$\Delta(E_i)=E_i\otimes 1
+\tilde K_i\otimes E_i\hspace{1cm}\Delta(F_i)=F_i\otimes \tilde K_{-i}
+1 \otimes F_i$$
and antipode on these generators defined by
$S(E_i)=-\tilde{K}_{-i}E_i,S(F_i)=-F_i\tilde{K}_{i}$.  We'll need to
also use the opposite coproduct
 $$\Delta^{\op}(E_i)=E_i\otimes \tilde K_i
+1\otimes E_i\hspace{1cm}\Delta(F_i)=F_i\otimes 1
+ \tilde K_{-i} \otimes F_i.$$
 Consider the Cartan involution 
\[\omega(E_i)=F_i\qquad \omega(F_i)=E_i\qquad \omega(K_i^{\pm
  1})=K_i^{\mp 1}\] on $\dot{U}$, and note that this involution
intertwines the usual and opposite coproducts $(\omega\otimes \omega)\circ
\Delta \circ \omega=\Delta^{\op}$.

We let $U_q^\Z(\fg)$ denote the Lusztig (divided powers) integral form
generated over $\Z[q,q^{-1}]$ by
$E_i^{(n)},F_i^{(n)},K_\mu$ for $n\geq 1$.  The integral form of the representation of
highest weight $\la$ if $\la$ is dominant or lowest weight $\la$ if
$\la$ is anti-dominant over this quantum group will be denoted by
$V_\la^\Z$.  It will be natural for us to use a slightly unusual
convention for our tensor products: when we tensor on the right with a
highest weight represention, we use the usual coproduct, but when
tensor with a lowest weight representation. More precisely, let $\boldsymbol{\epsilon}$ be the sign vector such that $\ep_k=-1$ if
$\la_k$ is dominant, and $\ep_k=1$ if it is anti-dominant, and let
$\otimes^{-1}$ be the tensor product with the usual coproduct, and
$\otimes^1$ be tensor product using $\Delta^{\op}$
\[V_{\bla}^\Z=\Big(\cdots \big((V_{\la_1}^\Z\otimes_{\Z[q,q^{-1}]}^{\ep_2} V_{\la_2}^\Z)\otimes_{\Z[q,q^{-1}]}^{\ep_3} V_{\la_3}^\Z \big) \otimes_{\Z[q,q^{-1}]}^{\ep_4}\cdots \Big)
\otimes_{\Z[q,q^{-1}]}^{\ep_\ell} V_{\la_\ell}^\Z.\] If we let
$i_1,\dots,i_p$ be the indices with $\ep_{i_k}=1$ and $j_1,\dots,j_q$
the indices with $\ep_{j_k}=-1$, then this is
isomorphic to the tensor product  
\[V_{\bla}^\Z\cong V_{\la_{i_p}}^\Z\otimes_{\Z[q,q^{-1}]}\cdots
\otimes_{\Z[q,q^{-1}]} V_{\la_{i_1}}^\Z \otimes_{\Z[q,q^{-1}]}
V_{\la_{j_1}}^\Z\otimes_{\Z[q,q^{-1}]}\cdots \otimes_{\Z[q,q^{-1}]}
V_{\la_{j_q}}^\Z \] using only the usual coproduct. 
We let $\bar V_\bla^\Z$ denote the reduction of $V_\bla^\Z$ at $q=1$.
Note that since $\Delta\equiv \Delta^{\op}\pmod {q-1}$, these ordering
issues are unimportant after this specialization.

\renc{\thetheorem}{\arabic{section}.\arabic{theorem}}

\section{Pre-canonical structures}
\label{sec:pre-canon-struct}

We let $V$ be a free $\Z[q,q^{-1}]$-module.
\begin{definition}
  A {\bf pre-canonical structure} on $V$ is a triple consisting of:
  \begin{itemize}
  \item a ``bar involution'' $\psi\colon V\to V$ which is
    $\Z[q,q^{-1}]$- anti-linear,
  \item a sesquilinear inner product $\langle -,-\rangle\colon V\times
    V\to \Z((q^{-1}))$, for which $\psi$ is flip-unitary:
    \[\langle u,v\rangle=\langle \psi(v),\psi(u)\rangle.\]
  \item a ``standard basis'' $a_c$ with partially ordered index set
    $(C,<)$ such that \[\psi(a_c)\in a_c+\sum_{c' <c}
    \Z[q,q^{-1}]\cdot a_{c'}.\]
  \end{itemize}
\end{definition}
Pre-canonical structures arise as shadows of categorical structures.
One of our tasks will be to describe the sort of categories that will
interest us.
Fix a field $\K$.  
\begin{definition}
  A {\bf humorous category} is an additive $\K$-linear Krull-Schmidt category $\mathcal{C}$ equipped with:
  \begin{itemize}
\item 
    an invertible grading shift functor $(1)$ such that the induced
    $\Z$-action on the set of indecomposables is free with finitely
    many orbits (we denote this set of such orbits by $C$), and 
\item a duality functor
    $M\mapsto M^\circledast$ which satisfies $(M(1))^\circledast\cong
    M^\circledast(-1)$, and has a unique fixed point $P_c\cong P_c^\circledast$ on each
    orbit $c\in C$.
 \end{itemize}
These must satisfy the conditions that:
\begin{itemize}
\item $\mathrm{Hom}_{\mathcal{C}}(M,N(i))=0$ for $i\ll 0$, and
  that $\dim \mathrm{Hom}_{\mathcal{C}}(M,N)<\infty$ for all objects
  $M$ and $N$.
\item The local ring $\Hom(P_c,P_c)$ has residue $\K$ for all $c$ (in
  general, we could have a division algebra over $\K$).  That is,
  $P_c$ is absolutely irreducible over $\K$; it remains indecomposable
  under any finite degree field extension of $\K$.
\end{itemize}

\end{definition} 

We can call a morphism $M\to N(i)$ a
morphism of degree $i$.  As in \cite{KLI}, we'll let $\HOM(M,N)\cong
\oplus_{i\in \Z} \Hom(M,N(i))$; we call elements of this space
{\bf degraded morphisms}. The reader might prefer to think of
the {\bf degraded category} $\mathcal{C}_\Z$ with the same objects as $\mathcal{C}$
and morphisms given by $\HOM(M,N)$ as graded $\K$-vector spaces, but
it is usually more convenient to work with $\mathcal{C}$. 

\begin{lemma}\label{KS-algebra}
  A category (with additional data) is humorous if and only if it is equivalent to the
category of finitely generated projective graded left modules over a graded
$\K$-algebra $A$ with morphisms given by homogeneous maps of degree 0,
with:
\begin{itemize}
\item Each graded piece of $A$ is finite dimensional, and the grades
  that appear in $A$ are bounded below.
\item Every indecomposable graded $A$-module is absolutely indecomposable.
\item There is an anti-automorphism $\gamma\colon A\to A$ which preserves at
  least one primitive idempotent in each isomorphism class.  
\end{itemize}
In this case, the humorous structure is given by:
\begin{itemize}
\item  The grading shift $M(i)$ is the same module as $M$ with all gradings decreased
by $i$.
\item 
The duality $\circledast$ is given by
$M^\circledast=\operatorname{HOM}_A(M,A)$ with the anti-automorphism $\gamma$ used
to switch this right
module to a left module.
\end{itemize}
\end{lemma}
\begin{proof}
  Most of this proof is simply applying definitions, so we will only
  give a sketch.  First, let $\mathcal{C}$ be a humorous category;
  this contains a self-dual object $O$ in
  which every indecomposable module appears as a summand with
  multiplicity 1.  The algebra
  $A$ is the opposite algebra of graded endomorphisms of $O$, that is $A^{op}\cong
  \HOM_{\mathcal{C} }(O,O)$, with the isomorphism given by $M\mapsto
 \HOM_{\mathcal{C} }(O,M)$.  Commutation with grading
 shift is clear.  Of course, \[\HOM_{\mathcal{C}
 }(O,M^\circledast)\cong \HOM_{\mathcal{C} }(M,O)\cong
 \HOM_A(\HOM_{\mathcal{C} }(O,M),A),\] so this shows the commutation
 with duality.  The anti-involution $\gamma$ is induced by any
 isomorphism $O\cong O^\circledast$; this must preserve each primitive
 idempotent since there is only one in each isomorphism class.

Now, assume that $A$ is a $\K$-algebra satisfying our conditions.  Its
category of representations is Krull-Schmidt and has finitely many
indecomposables (up to shift) since $A_0$ is finite
dimensional.  Since every projective module is a summand of $A^{\oplus
  n}$, we have that the finiteness of Hom spaces between projectives
follows from the finiteness of the grades of $A$.  

Finally, if we define $M^\circledast=\operatorname{HOM}_A(M,A)$ with
this right module turned into a left module using $\gamma$, then this
obviously commutes with grading shift, and sends
$(Ae)^\circledast\cong A\gamma(e)$.  Thus, we can define $P_c$ to be
the image of any $\gamma$-invariant primitive idempotent in the
corresponding isomorphism class.
\end{proof}

The reader might rightly wonder about the full category of
finite-dimensional graded modules; this appears as the category of representations of $\mathcal{C}$.
\begin{definition}
  Let $\operatorname{Rep}(\mathcal{C})$ be the category of additive functors
  from $\mathcal{C}$ to the category of finite dimensional $\K$-vector
  spaces.  One can easily see that if $\mathcal{C}$ is the projective
  modules over $A$, then $ \operatorname{Rep}(\mathcal{C})$ is the
  category of finite dimensional modules over $A$.  
\end{definition}
 There is a duality $\star$ on the abelian category
$\operatorname{Rep}(\mathcal{C})$ defined by the property that 
\begin{equation}
 \Hom(P^\circledast,M)=
\Hom(P,M^\star)^*.\label{star}
\end{equation}
If $ \operatorname{Rep}(\mathcal{C})$ is the
  category of finite dimensional modules over $A$, then this is simply
  vector space duality, with $\gamma$ used to switch left and right
  modules as before.

The data of a pre-canonical structure often ``decategorify'' structures on this category: that is, we
can take their shadow on the
Grothendieck group $K^0(\mathcal{C})$.   For example:
\begin{itemize}
\item the grading shift functor induces an invertible endomorphism of
  $K^0(\mathcal{C})$ which we call $q$, making $K^0(\mathcal{C})$ a
  $\Z[q,q^{-1}]$ module.  That is $q[M]=[M(1)]$
\item the decategorification of the duality
  functor gives an anti-linear involution $\psi$, 
\item the graded Euler pairing
  \[\langle [M],[N]\rangle=\sum_{i\in \Z} q^{-i}\dim\Hom(M,N(i))\]
gives a sesquilinear pairing, since 

\begin{multline*}
  \langle q^a[M],q^b[N]\rangle=\sum_{i\in \Z}
  q^{-i}\dim\Hom(M(a),N(i+b))=\sum_{i\in \Z}
  q^{-i}\dim\Hom(M(a),N(i+b))\\=\sum_{j\in \Z}
  q^{-j-a+b}\dim\Hom(M,N(j))=q^{b-a}\langle [M],[N]\rangle.
\end{multline*}
\end{itemize}

\begin{definition}
We call a pre-canonical structure {\bf balanced} if  $\langle
  \psi(a_c),a_d\rangle=\delta_{c,d}$ and {\bf almost balanced} if $\langle
  \psi(a_c),a_d\rangle\in \delta_{c,d}+q^{-1}\Z[q^{-1}]$.
\end{definition}

\begin{lemma}\label{category-pre-canonical}
  Assume that $\mathcal{C}$ is a graded humorous category with
  a partial order $(C,<)$, and
  that either:
  \begin{enumerate}
  \item we have a collection of self-dual objects $M_c\cong M_c^\circledast$ in
    $\mathcal{C}$ such that $M_c\cong
    P_c\oplus (\oplus_{c'<c}P_{c'}^{\oplus m_{c'}})$ and take
    $a_c=[M_c]$, or
  \item the category of $\operatorname{Rep}(\mathcal{C})$ is highest
    weight \footnote{To fix conventions, the indecomposable projective
      module $P_c$ has a filtration by the modules $\Delta_{c'}$ with
      $c'\leq c$; this is the opposite of the most common convention
      for highest weight categories.} for the order $(C,<)$, with standard modules  $\Delta_c$ (graded so that
    $\Delta_c$ is a quotient of $P_c$), and we take
    $a_c=[\Delta_c]=[M^0_c]-[M^1_c]+\cdots $ where $\cdots \to M^1_c\to
    M^0_c\to \Delta_c$ is a projective resolution.
  \end{enumerate}
Then, the basis $\{a_c\}$, the involution $\psi$ and the graded Euler
pairing defines a pre-canonical structure; in the second case, the
pre-canonical structure is balanced.
\end{lemma}
\begin{proof}
  The conditions that we need to check are:
  \begin{itemize}
  \item that $\psi$ is flip-unitary, which follows from the
    isomorphism \[\operatorname{Hom}(P,Q(i))=\operatorname{Hom}(Q^\circledast,P^\circledast(i)).\]
\item we need also confirm that  \[\psi(a_c)\in a_c+\sum_{c' <c}
    \Z[q,q^{-1}]\cdot a_{c'}.\]  This is clear in the first case,
    since $\psi(a_c)=a_c$.  In the second, since
    $[\Delta_c]=[P_c]+\sum_{c'<c} m_{c'}[P_{c'}]$, we have that 
\[\psi([\Delta_c])-[\Delta_c]=\sum_{c'<c} (\bar{m}_{c'}-m_{c'})[P_{c'}].\]
Since $[P_{c'}]$ is in the span of $[\Delta_{d}]$ for $d\leq c'$, 
this completes the proof.
  \end{itemize}
Finally, we must prove the balanced structure in the highest weight
case.  This involves looking at
$\Ext^\bullet(\Delta_c^\circledast,\Delta_{c'})$.  Thus, we have that
\[\Ext^\bullet(\Delta_c^\circledast,\Delta_{c'})\cong
\Ext^\bullet(\Delta_c,\Delta_{c'}^\star)^*\cong
\begin{cases}
  \K & c=c'\\
  0 & c\neq c'
\end{cases}\qedhere\]
\end{proof}

As suggested in the introduction, in some cases, vector spaces with a
pre-canonical structure will have a special basis called a canonical basis.

\begin{definition}\label{def:canonical}
  We call a basis $\{b_c\}$ of $V$ {\bf canonical} if:
  \begin{enumerate}
    \renc{\labelenumi}{(\Roman{enumi})}
  \item each vector $b_c$ in the basis is invariant under
    $\psi$.\label{bar}
  \item each vector $b_c$ in the basis is in the set $a_c+\sum_{c'<c}
    \Z[q,q^{-1}]\cdot a_{c'}$. \label{upper}
  \item the vectors $b_c$ are {\bf almost orthonormal} in the sense
    that
    \begin{equation*}
\langle b_c,b_{c'}\rangle \in
    \delta_{c,c'}+q^{-1}\Z[[q^{-1}]].
  \end{equation*}
\end{enumerate}
\end{definition}
Throughout the paper, we'll use {\it (I),\,(II),\,(III)} to refer to the
statements above.
\excise{the vectors $b_c$ are {\bf almost orthonormal} with positive
    product in the sense
    that
    \begin{equation}
\langle b_c,b_{c'}\rangle \in
    \delta_{c,c'}+q^{-1}\Z_{\geq 0}[[q^{-1}]].\label{eq:4}
  \end{equation}
\label{positive}
  \end{enumerate}
We call a basis {\bf canonish} if it satisfies {\it I., II.} and  
\begin{enumerate}
\item [{\it III'.}] \begin{remark}
  Note that according to this definition, Lusztig's canonical basis
  for $\dot{U}$ or an irreducible representation of $\dot{U}$ as
  defined in \cite{Lusbook} for non-symmetric finite types is {\it
    not} canonical but just canonish in the sense above, since the
  coefficients in \eqref{eq:4} can be negative.
\end{remark}}
The reader might have expected the last two items to instead read:
\begin{itemize}
\item [{\it (II')}] the transition matrix from $b_c$ to $a_c$ is
  the
  identity modulo $q^{-1}$; that is $b_c$ is in the set $a_c+\sum_{c'<c}
    q^{-1}\Z[q^{-1}]\cdot a_{c'}$.
\end{itemize}
In many important cases, this is equivalent, but the first definition
will prove more flexible.

\begin{lemma}\label{two-prime}
  If the standard basis $a_c$ is almost balanced, then conditions {\it (II)+(III)}
  are equivalent to condition {\it (II')}
\end{lemma}
\begin{proof}
  Assume that {\it (II)+(III)} hold.  Then if $b_c=a_c+\sum_{d<c}
  m_da_d$ and assume that $c'$ is minimal amongst elements such that
  $m_{c'}\notin q^{-1}\Z[[q^{-1}]]$.  Let $b_{c'}=a_{c'}+\sum_{d<{c'}}
  n_da_d$.  By {\it
    III.}, we have that $\langle b_c,b_{c'}\rangle\in 
  1+q^{-1}\Z[[q^{-1}]]$, so we have that 
\[\langle b_c,b_{c'}\rangle= \langle \psi(b_c),b_{c'}\rangle=\left\langle \psi(a_c)+\sum_{d<c}
  \bar{m}_d\psi(a_d),a_{c'}+\sum_{d<{c'}}
  n_da_d\right\rangle\in 
  1+q^{-1}\Z[[q^{-1}]].\] By the minimality of $c'$,
  we have that $m_dn_{d'}\langle \psi(a_d),a_{d'}\rangle\in
  q^{-1}\Z[[q^{-1}]]$ for all $d<c'$, so $m_c\langle
  \psi(a_c),a_{c}\rangle\in 1+
  q^{-1}\Z[[q^{-1}]]$.  This is only possible if $m_{c'}\in 1+q^{-1}\Z[[q^{-1}]]$ as well.
  This is a contradiction, so  {\it (II')} holds.

   Assume that {\it (II')} holds, so $b_c=a_c+\sum_{d<c}
  m_da_d$ and $b_{c'}=a_{c'}+\sum_{d<{c'}}
  n_da_d$ for every $c,c'$.  As calculated above
\[\langle b_c,b_{c'}\rangle= \langle \psi(b_c),b_{c'}\rangle=\left\langle \psi(a_c)+\sum_{d<c}
  \bar{m}_d\psi(a_d),a_{c'}+\sum_{d<{c'}}
  n_da_d\right\rangle.\] If $c\neq c'$, we have that each term
  $m_dn_{d'}\langle \psi(a_d),a_{d'}\rangle\in
  q^{-1}\Z[[q^{-1}]]$, so the same is true of $\langle
  b_c,b_{c'}\rangle$.  On the other hand, when the basis vectors
  coincide, \[\langle b_c,b_{c}\rangle\equiv \langle
  \psi(a_c),a_{c}\rangle \equiv 1\pmod {q^{-1}}.\qedhere\]
\end{proof}
In fact, the dependence on the standard basis is quite weak.  
\begin{theorem}\label{canonical-sign}
  Assume $V$ is a finitely generated $\Z[q,q^{-1}]$-module, equipped
  with a pre-canonical structure, and a canonical basis.  If $v\in V$
  is any vector in $V$ such that $\bar v=v$ and $\langle v,v\rangle\in
  1+q^{-1}\Z[q^{-1}]$, then either $v$ or $-v$ is a canonical basis
  vector.  In particular, the canonical bases of any two pre-canonical
  structures with the same bar-involution and form coincide up to
  sign.  
\end{theorem}
\begin{proof}
  We can write $v=\sum f_c(q)b_c$, with $\bar{f}_c=f_c\in
  \Z[q,q^{-1}]$.  Thus, $\sum_c f_c^2\in 1+q^{-1}\Z[q^{-1}]$.  This is only possible if
  $f_c=\pm 1$ for some $c$ and 0 otherwise. 
\end{proof}

Canonical bases have the distinct advantage of being computable using
a Gram-Schmidt algorithm.  Embedded in this algorithm is 
a well-trodden argument showing the uniqueness of this basis:
\begin{proposition}\label{basis-unique}
  A pre-canonical structure has at most one canonical basis.  In fact,
  if bases $\{a_c\}$ and $\{a_c'\}$ define pre-canonical structures
  with the same bar involution and bilinear form, and
  $a_c=a_c'+\sum_{d<c} p_da_d'$, then any canonical basis for
  $\{a_c\}$ coincides with any for $\{a_c'\}$.
\end{proposition}
\begin{proof}
  Assume $\{b_c\}$ and $\{b_c'\}$ are canonical bases for  $\{a_c\}$
  and $\{a_c'\}$ which are not
  identical.  Assume that $c$ is minimal with $b_c\neq b_c'$. 
By assumption, $b_c-b_c'=\sum_{d<c} m_db_d$ for some $m_d\in
\Z((q^{-1}))$, since the span of $a_d$ for $d< c$, the span of $a_d'$
for $d< c$, and  the span of
$b_d$ for $d<c$ all coincide.  By $\psi$-invariance of
canonical bases, $m_d$ must be a palindromic Laurent
polynomial in $q$.  On the other hand, by almost orthogonality, we
have $\langle b_c-b_c',b_d\rangle\in q^{-1}\Z[[q^{-1}]]$, so we must
also have that $m_d\in  q^{-1}\Z[[q^{-1}]]$.  This is a
contradiction, so $\{b_c\}$ must be unique.
\end{proof}
 However, showing existence is generally quite difficult, unless the
 pre-canonical structure comes from a categorification.
In this case, we have an easy restatement of the canonical property which is
more ``categorical'' in nature.  
\begin{definition}\label{mixed}
  Following Beilinson, Ginzburg and Soergel \cite{BGS96} and Achar and
  Stroppel \cite{AchS}, we define a humorous category to be
  {\bf mixed} if there is a weight function $\operatorname{wt}$ from
  indecomposable objects to $\Z$ satisfying $\operatorname{wt}(M^\circledast)=-\operatorname{wt}(M)$ such that $\Hom(M,N)=0$ whenever
  $\operatorname{wt}(N)< \operatorname{wt}(M)$ or when $M\ncong N$ and
  $\operatorname{wt}(N)= \operatorname{wt}(M)$, and $\Hom(M,M)\cong
  \K$ when $M$ is indecomposable. An abelian category is
  mixed with all simples absolutely irreducible in the sense of the earlier references if its category of
  projectives is split mixed in this sense. 
\end{definition}
In terms of an algebra $A$ satisfying the hypotheses of Lemma \ref{KS-algebra}, the category of projectives over $A$ is
mixed if and only if $A$ is positively graded.  

\begin{definition}
  We let the {\bf orthodox basis} of $K^0(\mathcal{C})$ be that
  defined by the classes $[P_c]$ of self-dual indecomposable modules
  in $\mathcal{C}$.
\end{definition}
Note that while a canonical basis only depends on the
pre-canonical structure, orthodox bases only exist in Grothendieck
groups and explicitly depend on the category.  We'll see examples
later of categories with canonically isomorphic Grothendieck groups
that give different orthodox bases.

\begin{lemma}\label{mixed-canonical}
  If $\mathcal{C}$ is a category satisfying the hypotheses of Lemma
  \ref{category-pre-canonical} (including one of the conditions (1) or
  (2)), then the orthodox basis is canonical for the pre-canonical
  structure if and only if the category $\mathcal{C}$ is mixed.
\end{lemma}
\begin{proof}
  In the cases (1) and (2), the basis $[P_c]$ satisfies the conditions {\it I.} and {\it
    II.} of a canonical basis; thus we need only check that almost
  orthogonality is equivalent to mixedness.  Since the bilinear form
  is the graded Euler form
  in this case, almost orthogonality is exactly the statement that
  $\Hom(P_c,P_{c'}(i))_0=0$ for $i\leq 0$ unless $i=0$ and $c=c'$, in
  which case $\Hom(P_c,P_c)\cong
  \K$, which is precisely the same as mixedness if $\wt(P_c)=0$.  The
  compatibility between weight and duality shows that this is the only
  possible weight function for $\mathcal{C}$ so this completes the proof.
\end{proof}
In fact, when a canonical and orthodox basis coincide, there are
stronger positivity properties than just the canonical property
implies.  In particular:
\begin{corollary}
  If a basis $\{b_c\}$ is simultaneously orthodox for some humorous
  category and canonical for the induced pre-canonical structure, then 
    \begin{equation}
\langle b_c,b_{c'}\rangle \in
    \delta_{c,c'}+q^{-1}\Z_{\geq 0}[[q^{-1}]].\label{eq:4}
  \end{equation}
\end{corollary}

We'll make an essentially trivial observation about the compatibility
of mixed structures which will still prove quite useful for us.
Assume that $\mathcal{C}, \mathcal{C}'$ are humorous categories, and that 
 there is a full and essentially surjective functor
  $a\colon \mathcal{C}\to \mathcal{C}'$ which commutes with grading
  shifts and duality.  
\begin{lemma}\label{lem:full-essential}
  If $\mathcal{C}$ is mixed, then so is $\mathcal{C}'$.  In this
  case, each orthodox basis vector $K^0(\mathcal{C}' )$ is the image
  of a unique orthodox basis vector in $K^0(\mathcal{C})$ under the
  induced map $[a]\colon K^0(\mathcal{C})\to K^0(\mathcal{C}')$; every
  other orthodox basis vector in $K^0(\mathcal{C})$ is sent to 0.
  Put differently, the image of the canonical basis of
  $K^0(\mathcal{C}')$ (plus a suitable number of zeros) is the image of
  the orthodox basis of $K^0(\mathcal{C})$.
\end{lemma}
\begin{proof}
  Since the functor $a$ is full, it sends indecomposable objects to
  indecomposable objects, and by essential
  surjectivity, each indecomposable $M$ is the image of an object $N$.  Every
  idempotent in $\End(N)$ is sent to 0 or 1 in $\End(M)$; by the
  finite dimensionality of degree 0 endomorphisms, there exists an
  idempotent $e$ whose image in $\End(M)$ is 1 and which cannot be written as the
  sum of two commuting idempotents.  The image $eN$ must be
  indecomposable, and we have $a(eN)=M$; thus we may as well assume
  $N$ is indecomposable.  
Thus, we can define a weight
  function on $\tU$ (resp.  $\hl^\bla$) by $\wt(M)=\wt(N)$.  If $N'$
  is another indecomposable object such that $M=a(N')$, by fullness, we
  must have that $\Hom(N,N')\neq 0$ and $\Hom(N',N)\neq 0$ since the
  identity of $M$ must be the image of some morphism.  By mixedness of
  $\mathcal{C}$, we have that
  $N'\cong N$.   

Similarly, it follows that $\mathcal{C}'$ is mixed for this
weight; if $\wt(M) >\wt (M')$ or $M\ncong M'$ with $\wt(M)=\wt(M')$,
then we indeed have $\Hom(M,M')=0$ by fullness since these objects are the image
of indecomposables with the same vanishing.

We have already shown above that the class of each indecomposable
  object in $\mathcal{C}'$ is the image of one from $\mathcal{C}$, and
  that no non-isomorphic pair of objects can be sent to the same
  class.  This shows the desired statement on canonical bases.
\end{proof}

Of course, the converse of this Lemma is false, but we can make an
`'if and only if'' statement if we strengthen the conclusion:
\begin{lemma}\label{asymptotic}
  If we have a sequence of functors $a_i\colon \mathcal{C}\to
  \mathcal{C}_i$ as in Lemma \ref{lem:full-essential} such that for
  every object $M$ in $\mathcal{C}$, the natural map
  $\End(M)\to \End(a_i(M))$ is an isomorphism for some $a_i$, then $\mathcal{C}$
  is  mixed if and only if all $\mathcal{C}_i$ are.  A vector in
  $K^0(\mathcal{C})$ lies in the orthodox basis if and only if its
  image under $[a_i]$ lies in the orthodox basis for $\mathcal{C}_i$
  for every $i$.
\end{lemma}
\begin{proof}
  The ``only if'' direction is Lemma \ref{lem:full-essential}.  If 
  $\mathcal{C}$ is not mixed, then either there is  a map between modules
  with the wrong weights, or some indecomposable module $M$ such that
  $\End(M)$ is a local ring with non-trivial Jacobson radical.  In
  either case, we can choose $a_i$ so that it does not kill this bad
  morphism, contradicting the assumption that $\mathcal{C}_i$ is 
  mixed.  Similarly, if a class in $K^0(\mathcal{C})$ is not an
  orthodox vector, then it must be a linear combination of multiples
  of orthodox vectors, and we can choose $i$ so that $[a_i]$ doesn't kill
  any of these classes.  Since the image of our class is orthodox
  under $[a_i]$, this gives a contradiction.
\end{proof}

This sort of mixedness appears naturally in geometry and category
theory as follows:
let $\EuScript{J}$ be a compactly generated $\K$-linear dg-category  such that $\dim \mathrm{Ext}_{\mathcal{C}}^i(M,N)<\infty$ for all objects $M$ and
$N$. Assume that $\EuScript{J}$ is endowed with a $t$-structure $(\EuScript{J}^{\geq 0},\EuScript{J}^{\leq 0})$ and a
duality functor $\circledast$ which are compatible in the sense that
$\EuScript{J}^{\geq 0}$ and $\EuScript{J}^{\leq 0}$ are both invariant
under duality.  Assume further that every simple object in the heart
$\EuScript{H}=\EuScript{J}^{\geq 0}\cap \EuScript{J}^{\leq 0}$ is
self-dual and absolutely irreducible.
\begin{definition}
  Let $\mathcal{J}$ be additive subcategory of $\EuScript{J}$
  generated by shifts of simple objects of $\EuScript{H}$. We think of
  this as a category in the usual sense by taking the morphisms
  between two objects to be the zeroth cohomology of the morphism
  complex in the dg-category $\EuScript{J}$.
\end{definition}
\begin{lemma}\label{J-mixed}
  The category $\mathcal{J}$ is a  mixed humorous category, with $(1)$ given by homological shift in $\EuScript{J}$ and
  duality induced by $\circledast$.
\end{lemma}
\begin{proof}
  That $\mathcal{J}$ is additive over $\K$ is automatic from the
  definition.  The Krull-Schmidt property  follows from the assumption
  that $\dim \mathrm{Ext}_{\mathcal{C}}^0(M,M)<\infty$ for every $M$.
  Since we have a $t$-structure, any module in $\EuScript{H}$ has no
  negative self-Exts.  In particular, no shift of a module in
  $\EuScript{H}$ can lie in $\EuScript{H}$, so the action of grading
  shift on indecomposable objects is free.  The finiteness of the
  number of orbits follows from that fact that $\EuScript{H}$ has
  finitely many simples.  This follows
  because of the compact generation of the category by an object $M$
  with  $\dim \mathrm{Ext}_{\mathcal{C}}^i(M,M)<\infty$.  Every orbit
  has a fixed point given by representative in $\EuScript{H}$.

  Any duality on a compactly generated dg-category must satisfy that
  $(M[1])^\circledast\cong M^\circledast[-1]$, so that property is
  automatic.  Finally, the  mixedness follows from the fact
  that two simple objects $M,N\in \EuScript{H}$ have trivial negative
  Exts and \[\mathrm{Ext}_{\mathcal{C}}^0(M,N)\cong
  \begin{cases}
    \K & M\cong N\\
    0 & M\ncong N
  \end{cases}
\] by absolute irreducibility.
\end{proof}
The example of most interest to us is when $\EuScript{J}$ is a full
subcategory of the category of constructible complexes of sheaves on
an algebraic variety or Artin stack generated by IC sheaves for trivial local
systems.  In this case, the $t$-structure we will want to take is the
perverse one and the duality $\circledast$ will be Verdier duality.

While we know of nowhere in the literature where most of the
definitions above are
made in this generality, there are many examples.  In each case, we
will leave the details of the pre-canonical structure to the references:
\begin{itemize}
\item Kazhdan and Lusztig showed that the Hecke algebra of a Weyl group
  has a canonical basis \cite{KL79}, now usually called the
  {\bf Kazhdan-Lusztig basis}.
\item Lusztig showed that the simple integrable representations of quantized
  universal enveloping algebras of Kac-Moody algebras as well as a
  small modification $\dot{U}$ of the algebras themselves have canonical bases \cite{Lusbook}.
  These also appeared in the work of Kashiwara as {\bf global crystal
    bases}.
\item In finite type, the tensor product of simple representations
  also carries a natural canonical basis \cite[\S 27.3]{Lusbook}; in the special case of a
  tensor product of highest and lowest weight representations, this
  works for infinite type Kac-Moody algebras as well \cite{LusTen}. 
\item Lascoux, Leclerc and Thibon \cite{LLT} show that a level 1 Fock space representation
  of $\mathfrak{\widehat{sl}}_n$ carries a
  canonical basis. This was extended to higher level twisted Fock
  spaces by Uglov \cite{Uglov}. Brundan and Kleshchev \cite{BKgd} showed that tensor products
  of level 1 Fock spaces also have a canonical basis arising as a
  ``limit'' of Uglov's.
\end{itemize}

All of the bases and pre-canonical structures listed above have close
ties to humorous categories as in Lemma \ref{category-pre-canonical}:

\begin{itemize}
\item The Kazhdan-Lusztig basis arises from the categorification of
  the Hecke algebra by $B\times B$-equivariant mixed sheaves on $G$,
  the associated algebraic group \cite{SpIH}; alternatively, there is
  an equivalent approach using indecomposable
  Soergel bimodules \cite{Soe92}.  Recently, Elias and Williamson
  established that the category of Soergel bimodules is mixed for an
  arbitrary reflection group over $\R$ \cite{EWHodge}.
\item For $\mathfrak{sl}_n$, the canonical basis of a tensor product
  of fundamental representations 
  corresponds to the projective (or tilting, depending on conventions)
  objects in a parabolic category $\cO$, equipped with its Koszul
  grading \cite[Th. 6]{Subasis}.
\item For $\mathfrak{\widehat{sl}}_n$, the canonical basis on a level 1
  Fock space arises from a graded version of the $q$-Schur algebra
  (for $q$ an $n$th root of unity) \cite{Arikiq}. For higher level Fock spaces 
  this canonical basis to comes from category $\mathcal{O}$ for Cherednik
  of the complex reflection group $G(r,1,\ell)$ by recent work of
  several authors \cite{RSVV,LoVV,WebRou}.
\item For $\mathfrak{sl}_2$, the indecomposable
  objects of $\tU$ match Lusztig's canonical basis by work of Lauda
  \cite[9.12]{Laucq}, and similarly for $\mathfrak{sl}_3$ by \Stosic
  \cite{Stos}.
\end{itemize}
Our aim in this paper is to complete the connection of categorifications
to each of the canonical bases listed above, covering the examples of
$\dot{U}$ itself and tensor products of highest weight modules.

\section{Dual canonical bases}
\label{sec:dual-canonical-bases}

\subsection{Duality of pre-canonical structures} 
For any $\Z[q,q^{-1}]$-module $V$, let $\bar V$ be the same underlying
abelian group with the action of $\Z[q,q^{-1}]$ twisted by the
bar-involution.  

Throughout Section \ref{sec:dual-canonical-bases}, we'll assume that $V$ is a free
$\Z[q,q^{-1}]$-module with a pre-canonical structure $\{\langle
-,-\rangle,\psi, \{a_c\}\}$, with $\langle -,-\rangle$ valued in
$\Z[q,q^{-1}]$ (rather than $\Z((q^{-1}))$).  In this case, we can
consider the dual space $V^*$ of $\Z[q,q^{-1}]$-linear functionals that kill all but
finitely many $a_c$ and the anti-linear evaluation map $\epsilon$
sending $v\mapsto \langle v,-\rangle$. We'll furthermore assume that this
map is an isomorphism.  Versions of the results in this section are
possible in more general situations, but we have aimed for the setup that will give us the
cleanest statement of results.

\begin{lemma}\label{finite-hom}
  If $V$ is the Grothendieck group of a humorous category
  $\mathcal{C}$, then $\langle -,-\rangle$ is valued in
$\Z[q,q^{-1}]$ if and only if for every pair of objects $M,N$ in $\mathcal{C}$, we have that
  $\Hom_{\mathcal{C}}(M,N(i))=0$ for all but finitely many $i\in \Z$.\qed
\end{lemma}

We can naturally identify $V^*$ with the Grothendieck group of
$\operatorname{Rep}(\mathcal{C})$ using the pairing
$\{[P],[R]\}=\sum_iq^{-i}\dim R(P(i))$.    With this identification,
we can think of the map $\epsilon\colon \bar V\to V^*$ as categorifying the Yoneda embedding sending
$P\mapsto \Hom(P,-)$. 
\begin{lemma}\label{proj-dim}
If the category  $\operatorname{Rep}(\mathcal{C})$ has
  finite projective dimension,
then $\epsilon\colon\bar V\to V^*$ is an isomorphism.
\end{lemma}
\begin{proof}
  The space $V^*$ is spanned by the dual basis to the classes of
  indecomposables in $V$. These are the classes of the simple
  representations $L_c$: those killing all but one indecomposable object,
  and the remaining indecomposable to $\K$.  Thus, we need only prove
  that these objects are in the image of $\epsilon$.  Since
  $\operatorname{Rep}(\mathcal{C})$ has finite global dimension, we
  have a finite projective resolution $L_c\leftarrow P_0\leftarrow
  P_1\leftarrow\cdots$.  Thus, $[L_c]=\sum_{i=0}^\infty (-1)^i[P_i]$.  
\end{proof}
Throughout this section, we'll let $\mathcal{C}$ be a humorous category
satisfying the hypotheses of Lemmata \ref{finite-hom} \& \ref{proj-dim}.

In this
case, the same underlying abelian group of $V$ has a second natural
pre-canonical structure.  Let $\psi^*\colon V\to V$ be the involution
defined by
\begin{equation}
\langle \psi u,v\rangle=\overline{\langle
  u,\psi^*v\rangle}.\label{dual-bar}
\end{equation}
We call this the {\bf dual bar-involution}.
If we identify the Grothendieck groups of $\mathcal{C}$ and
$\operatorname{Rep}(\mathcal{C})$ as discussed above, comparing
equations \eqref{star} and \eqref{dual-bar} shows that 
$\psi^*$ categorifies the duality $\star$.  

Let $\{a^*_c\}$ denote the right dual basis of $\{a_c\}$, equipped with the opposite
partial order.  That is, it is the unique basis satisfying
$\langle a_c,a^*_{c'}\rangle=\delta_{c,c'}$. In certain infinite
dimensional situations, the Gram-Schmidt construction of this basis
will not converge;  it is enough to assume that $V$ is a sum of
finite-dimensional orthogonal spaces with bases given by sets of
$a_c$'s.  This will hold for a tensor product of highest weight
modules (where we break into weight spaces), but not for $U_q(\fg)$.

Note that in this notation, a pre-canonical structure is balanced if $\psi^*(a_c)=a_c^*$.

\begin{proposition}\label{dual-canonical}
  The triple $\{\overline{\langle -,-\rangle}, \psi^*, \{a_c^*\}\}$
  is a pre-canonical structure on $\bar V$ if we endow $C$ with the
  opposite partial order.  If the primal structure $\{\langle -,-\rangle,
  \psi, \{a_c\}\}$ has a canonical basis $\{b_c\}$, then the right dual basis
  $\{b^*_c\}$ is canonical for the dual pre-canonical structure.  
\end{proposition}
Thus, $\{b_c^*\}$ doubly merits the title {\bf dual canonical basis}: it
is both the dual basis to a canonical one, and canonical for the dual
pre-canonical structure.
\begin{proof}
Since $\psi$ is flip-unitary, its conjugate is flip-unitary as well:
\[\langle \psi^*u,\psi^*v\rangle=\overline{\langle\psi\psi^*u,v\rangle}
=\overline{\langle \psi v, \psi^*u\rangle}=\langle v,u\rangle.\]
Furthermore, the upper-triangularity of $\psi$ on the basis $\{a_c\}$
exactly translates into lower-triangularity for the transpose.  Thus,
for the reversed order, we have that $\psi^*$ is upper-triangular.

Obviously, the dual basis to any $\psi$-invariant basis will consist
of $\psi^*$-invariant elements.  Similarly, if a basis is triangular
with respect to $\{a_c\}$, its dual basis will be obtained from $\{a_c^*\}$ by
the transposed basis change matrix, and thus also be triangular.

Note that by duality, we have that $b_c=\sum_{c'}\langle
b_{c'},b_{c}\rangle b_{c'}^*$; thus, we have 
\[\delta_{cc''}=\langle
b_{c},b_{c''}^*\rangle=\sum_{c'}\langle
b_{c'},b_{c}\rangle\overline{ \langle
b_{c'}^*,b_{c''}^*\rangle}.\] That is, the matrix of $\langle
-,-\rangle$ for the basis $\{b_c\}$  is inverse to that for
$\overline{ \langle -,-\rangle}$ in $\{b_c^*\}$.

Since $\langle
b_{c'},b_{c}\rangle \in \delta_{c'c}+q^{-1}\Z[[q^{-1}]]$, it must
also hold that $\overline{ \langle
b_{c'}^*,b_{c''}^*\rangle}\in \delta_{c'c''}+q^{-1}\Z[[q^{-1}]]$.  
\end{proof}

 In this case we have a second way
of thinking about the dual canonical basis:
\begin{corollary}
Assume $\mathcal{C}$ is a humorous category such that its orthodox
basis defines a canonical basis for the induced
pre-canonical structure.
  The dual canonical basis of this pre-canonical structure is defined
  by the classes of the simple modules in
  $\operatorname{Rep}(\mathcal{C})$.  
\end{corollary}
The significance of using the barred Euler form  in the dual
pre-canonical structure is that the Ext space between two simple
modules in a mixed category is {\it negatively} graded as a simple
calculation with free resolutions shows, so this reversal is necessary
to have any hope of canonicity holding.

When the pre-canonical structure is balanced, we have yet a third
precanonical structure, given by $(\langle
-,-\rangle,\psi^*,\{a_c\})$.  This has  the bar-involution of the dual
action and the form of the primal; we'll call this the {\bf Ringel
  dual} pre-canonical structure (in order to distinguish, one could
call the ``dual'' above the {\bf Koszul dual}), since on the
categorical level, it matches with Ringel duality.

\subsection{Balanced positivity}
\label{sec:balanced-positivity}

One particularly common phenomenon is that if $a_c$ is carefully
chosen, the basis $\{b_c\}$ will have good positivity properties.
\begin{definition}
   We say a pre-canonical structure with canonical basis 
  $b_c=\sum_{c'}m_{cc'}(q)a_{c'}$ for $m_{cc'}(q)\in q^{-1}\Z_{\geq
    0}[[q^{-1}]]$ is {\bf
    balanced positive} if it is balanced and \[a_c=\sum_{c'}n_{cc'}(-q)b_{c'}\] for
  $n_{cc'}(q)\in q^{-1}\Z_{\geq 0}[[q^{-1}]]$.
\end{definition} 
Such bases arise naturally in representation theory through
categorifications.  Assume that $\mathcal{C}$ is a  mixed humorous category such that $\operatorname{Rep}(\mathcal{C})$ is highest weight
and standard Koszul (see \cite{ADL} for definitions). 
As before we let $C$ be the set of self-dual indecomposable modules
in $\mathcal{C}$  and let $\Delta(c)$ be the standard modules with cosocle
concentrated in weight 0.  

\begin{theorem}\label{Grothendieck-balanced}
  With $\mathcal{C}$ as discussed above, the Grothendieck group $K^0(\mathcal{C})$ has a balanced positive
  pre-canonical structure such that
  \[ \psi=[\circledast]\qquad \psi^* =[\star]\qquad \langle [M] ,[N]
  \rangle =\sum_{i,j}(-1)^iq^{-j}\dim\Ext^i_A(M,N(j))\qquad a_c=[\Delta(c)].\]
The canonical basis is given by the classes of the $\circledast$-self-dual indecomposable
projective modules and the dual canonical basis by the classes of the
 $\star$-self-dual simple modules.  The canonical basis of the Ringel dual pre-canonical
structure are given by the classes of $\star$-self-dual indecomposable tilting
modules.  
\end{theorem}
This set up seems quite specialized (and indeed it is), but it has
made several appearances in the literature.  Categories that satisfy
the hypotheses of the theorem include:
\begin{itemize}
\item the category $\cO$ of a semi-simple Lie algebra by \cite[3.8]{ADL}, so the Kazhdan-Lusztig
  basis of the Hecke algebra is balanced positive.
\item the truncated parabolic category $\cO$ of Shan and Vasserot
  satisfies these conditions by \cite[4.3]{SVV}, so Uglov's canonical
  basis of twisted Fock spaces are balanced positive.  As a special
  case, the same holds for tensor products of wedge powers of the
  natural representation of $\mathfrak{sl}_n$.
\item the hypertoric category $\cO$ defined by the author jointly with
  Braden, Licata and Proudfoot \cite{GDKD,BLPWtorico}.
\end{itemize}

\begin{proof}
That we have a pre-canonical structure follows from Lemma
\ref{category-pre-canonical}.
The compatibility of these two involutions is precisely the
decategorification of \eqref{star}.
By Lemma \ref{mixed-canonical}, the basis $b_c=[P_c]$ is canonical, so
by Proposition \ref{dual-canonical}, the right dual basis
$b_c^*=[L_c]$ is the classes of the simple modules.  The classes of
the indecomposable tiltings are invariant under $\psi^*$, and almost
orthogonality follows from the positivity of the grading on the Ringel
dual, and the upper-triangularity with respect to classes of the
standards is a standard property of tiltings.

Finally, we wish to show balanced positivity.  The positivity of
$[P_c]$ in terms of $[\Delta_c]$ follows from the fact that $P_c$
has a standard filtration; the polynomials $m_{**}(q)$ are just the
graded multiplicities of this filtration.  On the other hand, the
(twisted) positivity of $[\Delta_c]$ follows from the fact that
$[\Delta_c] $ has a {\it linear} resolution by projectives (this is
the definition of standard Koszulity).  Thus $n_{**}(q)$ is the
graded multiplicities of this resolution, where $q$ could either
measure grading shift {\it or} homological shift, which coincide by
the linearity of the resolution.
\end{proof}

This theorem shows that balanced positivity is a natural condition
from the categorical perspective.  Now, we give a combinatorial
consequence of this definition.
\begin{theorem}
  If $\{\langle -,-\rangle,\psi, \{a_c\}\}$ is a balanced positive
  pre-canonical structure, then its dual structure will be balanced positive
for the variable $p=-q^{-1}$.
\end{theorem}
\begin{proof}
The essential point is that the polynomials $m_{**}$ and $n_{**}$ will
switch roles. 
  The proof of this fact is a combinatorial version of BGG
  reciprocity. Note that 
\[\langle b_{c},a_{c'}^*\rangle=\overline{m_{cc'}(q)}=m_{cc'}(-p)\qquad \langle a_{c},b_{c'}^*\rangle=\overline{n_{cc'}(-q)}=n_{cc'}(p)\]
 By the definition of dual bases
 \begin{align*}
   a_c^*&=\sum_{c'} \langle b_{c'},a_c^*\rangle b_{c'}^*=\sum_{c'}
   m_{c'c}(-p) b_{c'}^*\\
 b_c^*&=\sum \langle a_{c'},b_c^*\rangle
   a_{c'}^*=\sum_{c'}
   n_{c'c}(p) a_{c'}^*
 \end{align*}
This exactly shows balanced positivity.
\end{proof}

\begin{remark}
  When we are considering a basis which comes a category
  $\mathcal{C}=A\mmod$ for some standard Koszul algebra $A$ 
  satisfying the the hypotheses of Theorem \ref{Grothendieck-balanced}, this fact has a categorical proof.
  As usual, the algebra $A$ has a Koszul dual $A^!\cong
  \Ext^\bullet_{A}(A_0,A_0)$, which satisfies the same conditions.
  The Koszul duality functor $K$ defined in \cite[2.12]{BGS96} induces an isomorphism
of graded Grothendieck groups  $K^0(A\mmod)\cong K^0(A^!\mmod)$
sending $q$ to $-q^{-1}$ which sends the
  dual pre-canonical structure to the primal pre-canonical structure
  of Theorem \ref{Grothendieck-balanced}; thus if one has the
  desired positivity, the other does as well.
\end{remark}

\section{The 2-category $\mathcal{U}$}
\label{sec:2-category-cu}

In this paper, our notation builds on that of Khovanov and Lauda, who
give a graphical version of the 2-quantum group, which we denote $\tU$
(leaving $\fg$ understood).  These constructions could also be
rephrased in terms of Rouquier's description and we have striven to
make the paper readable following either \cite{KLIII} or
\cite{Rou2KM}; however, it is most sensible for us to use the
2-category defined by Cautis and Lauda \cite{CaLa}, which is a
variation on both of these.  See the introduction of \cite{CaLa} for
more detail on the connections between these different approaches.   

The object of interest for this subsection is a strict 2-category; as
described, for example, in \cite{Laucq}, one natural yoga for
discussing strict 2-categories is planar diagrammatics.  The
2-category $\tU$ is thus most clearly described in this language.  

\begin{definition}
  A {\bf blank KL diagram} is a collection of finitely many oriented
  curves in $\R\times
  [0,1]$ which has no triple points,
  decorated with finitely many dots.  Every strand is labeled with an
  element of $\Gamma$, and any open end must meet one of the lines
  $y=0$ or $y=1$ at a distinct point from all other ends.

  A {\bf KL diagram} is a blank KL diagram together with a labeling of
  regions between strands (the components of its complement) with weights following the rule \[  \tikz[baseline,very thick]{
\draw[postaction={decorate,decoration={markings,
    mark=at position .5 with {\arrow[scale=1.3]{<}}}}] (0,-.5) -- node[below,at start]{$i$}  (0,.5);
\node at (-1,0) {$\mu$};
\node at (1,.05) {$\mu-\al_i$};.
}\]
\end{definition}
We identify two KL diagrams if they are isotopic via an isotopy which
does not cancel any critical points of the height function or move
critical points through crossings or dots.  We will
deal with isotopies that do have these features later.  In the interest of simplifying diagrams, we'll often write a dot with
a number beside it to indicate a group of that number of dots.

We call the lines $y=0,1$ the {\bf  bottom} and {\bf top} of the
diagram.  Reading across the bottom and top from left to right, we
obtain a sequence of elements of $\Gamma$, which we wish to record in
order from left to right.  Since orientations are quite important, we
let $\pm \Gamma$ denote $\Gamma\times \{\pm 1\}$, and associate $i$ to
a strand labeled with $i$ which is oriented upward and $-i$ to one oriented downward.  For example, we have a blank KL diagram  \[
  m=
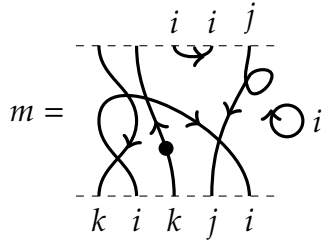
\begin{tikzpicture}[baseline,very thick]
\draw[thin,dashed] (1.33,1) -- (-1.33,1); 
\draw[thin,dashed] (1.33,-1) -- (-1.33,-1); 
\draw [postaction={decorate,decoration={markings,
    mark=at position .7 with {\arrow[scale=1.3]{>}}}}] (-.5,-1) to[out=90,in=-90] node[below,at start]{$i$} (-1,0)
  to[out=90,in=90] node[below,at end]{$i$} (1,-1);
  \draw [postaction={decorate,decoration={markings,
    mark=at position .4 with {\arrow[scale=1.3]{<}}}}] (.5,-1) to[out=90,in=-120] node[below,at start]{$j$} (1,.6)
  to[out=60,in=90] (1.3,.6) to [out=-90,in=-60] (1,.4)to[out=120,in=-90]  node[above,at end]{$j$} (1,1);
  \draw[postaction={decorate,decoration={markings,
    mark=at position .5 with {\arrow[scale=1.3]{<}}}}]  (.5,1) to[out=-90,in=-90] node[above,at start]{$i$} node[above,at end]{$i$}(0,1);
  \draw[postaction={decorate,decoration={markings,
    mark=at position .4 with {\arrow[scale=1.3]{<}}}}] (-1, -1) to[out=90,in=-90]node[below,at start]{$k$} (-.5,0) to[out=90,in=-90] (-1,1);
  \draw[postaction={decorate,decoration={markings,
    mark=at position .5 with {\arrow[scale=1.3]{>}}}}] (0,-1)
to[out=90,in=-90] node[below,at start]{$k$} node[pos=.3,circle,fill,inner sep=2pt]{}
  (-.5,1);
\draw[postaction={decorate,decoration={markings,
    mark=at position .5 with {\arrow[scale=1.3]{<}}}}] (1.5,0) circle (6pt); \node at (1.9,0){$i$};
\end{tikzpicture}
\]
with top given by $(-k,k,-i,i-j)$ and bottom given by
$(-k,i,k,-j,-i)$.  

We also wish to record the labeling on regions; since fixing the label
on one region determines all the others, we'll typically only record
$\EuScript{L}$, the weight of the region at far left and
$\EuScript{R}$, the weight at far right.  
In addition, we will
typically not draw the weights on all regions in the interest of
simplifying pictures. We call the pair of a sequence $\Bi\in (\pm
\Gamma)^n$ and the weight $\EuScript{L}$ a {\bf KL pair}; let
$\EuScript{R}:=\EuScript{L}+\sum_{j=1}^n\al_{i_j}$ where we let
$\al_{-i}=-\al_i$.
\begin{definition}
  Given KL diagrams $a$ and $b$, their {\bf (vertical) composition} $ab$ is
  given by stacking $a$ on top of $b$ and attempting to join the
  bottom of $a$ and top of $b$. If the sequences
  from the bottom of $a$ and top of $b$ don't match or
  $\EuScript{L}_a\neq \EuScript{L}_b$, then the
  composition is not defined and by convention is 0, which is not a
  KL diagram, just a formal symbol.

  The {\bf horizontal composition}  $a\circ b$ of KL diagrams is the
  diagram which 
pastes together the strips where $a$ and $b$ live with $a$ to the {\it
  right} of $b$.  The only compatibility we
require is that $  \EuScript{L}_a=\EuScript{R}
_b$, so that the regions of the new diagram can be labeled
consistently. If $  \EuScript{L}_a\neq\EuScript{R}
_b$, the horizontal composition is 0 as well.
\end{definition}

Implicit in this definition is a rule for {\bf horizontal composition}
of KL pairs in $\pm \Gamma$, which is the reverse of concatenation
$(i_1,\dots, i_m)\circ (j_1,\dots, j_n)=(j_1,\dots,
j_n,i_1,\dots,i_m)$, and gives 0 unless $  \EuScript{L}_{\Bi}=\EuScript{R}
_{\Bj}$.

We should warn the reader, this convention
requires us to read our diagrams differently from the conventions of
\cite{Laucq,KLIII,CaLa}; in our diagrammatic calculus, 1-morphisms point
from the left to the right, not from the right to the left as
indicated in \cite[\S 4]{Laucq}.   The
practical implication will be that our relations are the reflection
through a vertical line of Cautis and Lauda's.  

We can define a {\bf degree} function on KL diagrams.  The degrees are
given on elementary diagrams by \[
  \deg\tikz[baseline,very thick,scale=1.5]{\draw[->] (.2,.3) --
    (-.2,-.1) node[at end,below, scale=.8]{$i$}; \draw[<-] (.2,-.1) --
    (-.2,.3) node[at start,below,scale=.8]{$j$};}
  =\langle\al_i,\al_j\rangle \qquad \deg\tikz[baseline,very
  thick,->,scale=1.5]{\draw (0,.3) -- (0,-.1) node[at
    end,below,scale=.8]{$i$} node[midway,circle,fill=black,inner
    sep=2pt]{};}=-\langle\al_i,\al_i\rangle \qquad
  \deg\tikz[baseline,very thick,scale=1.5]{\draw[<-] (.2,.3) --
    (-.2,-.1) node[at end,below,scale=.8]{$i$}; \draw[->] (.2,-.1) --
    (-.2,.3) node[at start,below,scale=.8]{$j$};}
  =\langle\al_i,\al_j\rangle \qquad \deg\tikz[baseline,very
  thick,<-,scale=1.5]{\draw (0,.3) -- (0,-.1) node[at
    end,below,scale=.8]{$i$} node[midway,circle,fill=black,inner
    sep=2pt]{};}=-\langle\al_i,\al_i\rangle\]
  \[
  \deg\tikz[baseline,very thick,scale=1.5]{\draw[->] (.2,.1)
    to[out=-120,in=-60] node[at end,above left,scale=.8]{$i$} (-.2,.1)
    ;\node[scale=.8] at (0,.3){$\la$};} =\langle\la,\al_i\rangle-d_i
  \qquad \deg\tikz[baseline,very thick,scale=1.5]{\draw[<-] (.2,.1)
    to[out=-120,in=-60] node[at end,above left,scale=.8]{$i$}
    (-.2,.1);\node[scale=.8] at (0,.3){$\la$};} =-\langle
  \la,\al_i\rangle-d_i
  \]
  \[
  \deg\tikz[baseline,very thick,scale=1.5]{\draw[<-] (.2,.1)
    to[out=120,in=60] node[at end,below left,scale=.8]{$i$} (-.2,.1)
    ;\node[scale=.8] at (0,-.1){$\la$};} =\langle
  \la,\al_i\rangle-d_i \qquad \deg\tikz[baseline,very
  thick,scale=1.5]{\draw[->] (.2,.1) to[out=120,in=60] node[at
    end,below left,scale=.8]{$i$} (-.2,.1);\node[scale=.8] at
    (0,-.1){$\la$};} =-\langle\la,\al_i\rangle-d_i.
  \]
For a general diagram, we sum together the degrees of the elementary
diagrams it is constructed from.

Now, we'll wish to assemble these into a linear category over a ring
$\K$; we'll be most interested in the case where $\K$ is a field, but
it will be convenient for us to let $\K$ be a commutative complete local ring as well.
Once and for all, fix a matrix of
polynomials $Q_{ij}(u,v)=\sum_{k,m}Q_{ij}^{(k,m)}u^kv^m$  valued in
$\K$ and indexed by $i\neq j\in
\Gamma$; by convention
$Q_{ii}=0$. We assume each polynomial is homogeneous
of degree $-\langle\al_i,\al_j\rangle= -2d_ic_{ij}=-2d_jc_{ji}$ when
$u$ is given degree $2d_i$ and $v$ degree $2d_j$.   We will always
assume that the leading order of $Q_{ij}$ in $u$ is $-c_{ij}$, and
that $Q_{ij}(u,v)=Q_{ji}(v,u)$.  We let $t_{ij}=Q_{ij}^{(-c_{ij},0)}=Q_{ij}(1,0)$; by
convention $t_{ii}=1$. In \cite{CaLa}, the coefficients of this
polynomial are denoted
\[Q_{ij}(u,v)=t_{ij} u^{-c_{ij}}+t_{ji} v^{-c_{ji}}+\sum_{pd_i+ qd_j=d_ic_{ij}} s^{pq}_{ij}u^pv^q.\]
Khovanov and Lauda's original category uses the
choice $Q_{ij}=u^{-c_{ij}}+v^{-c_{ji}}$.

\begin{definition}
Let $\tU$ be the 2-category whose
\begin{itemize}
\item objects are the weights $\wela(\fg)$,
\item 1-morphisms $\mu\to \nu$ are grading shifts of KL pairs with $\EuScript{L}=\mu$
  and $\EuScript{R}=\nu$,
\end{itemize}
and whose degraded 2-morphisms are a quotient of the formal span over $\K$ of
KL diagrams.  Before giving these, we note that each such diagram has
a degree, which we adjust by the grading shift of its source and
target, to arrive at the degree of the 2-morphism; by our conventions,
these will be ``honest'' 2-morphisms only if they have degree 0. 
 As before, we'll denote the
  space of 2-morphisms of degree 0 between two 1-morphisms $u,v$ by
  $\Hom(u,v)$, and let $\HOM(u,v):=\oplus_i\Hom(u,v(i))$ denote
  the space of degraded morphisms.

 The relations we
impose on degraded 2-morphisms are:
\begin{itemize}
\item the cups and caps are the units and counits of a biadjunction.
  The dot morphism is cyclic.  The cyclicity for crossings can be
  derived from the pitchfork relation:
\newseq
  \begin{equation*}\subeqn\label{pitch1}
   \tikz[baseline,very thick]{\draw[dir] (-.5,.5) to [out=-90,in=-90] node[above, at start]{$i$} node[above, at end]{$i$} (.5,.5); \draw[edir]
     (.5,-.5) to[out=90,in=-90] node[below, at start]{$j$} node[above, at end]{$j$} (0,.5);}= \tikz[baseline,very thick]{\draw[dir] (-.5,.5) to [out=-90,in=-90] node[above, at start]{$i$} node[above, at end]{$i$} (.5,.5); \draw[edir]
  (-.5,-.5) to[out=90,in=-90] node[below, at start]{$j$} node[above, at
   end]{$j$}(0,.5) ;}    \qquad  \tikz[baseline,very thick]{\draw[dir] (-.5,-.5) to [out=90,in=90] node[below, at start]{$i$} node[below, at end]{$i$} (.5,-.5); \draw[dir]
     (0,-.5) to[out=90,in=-90] node[below, at start]{$j$} node[above, at end]{$j$} (-.5,.5);}= \tikz[baseline,very thick]{\draw[dir] (-.5,-.5) to [out=90,in=90] node[below, at start]{$i$} node[below, at end]{$i$} (.5,-.5); \draw[dir]
    (0,-.5)  to[out=90,in=-90] node[below, at start]{$j$} node[above, at end]{$j$} (.5,.5);}
  \end{equation*}
\begin{equation*}\subeqn\label{pitch2}
 \tikz[baseline,very thick]{\draw[dir] (-.5,.5) to [out=-90,in=-90] node[above, at start]{$i$} node[above, at end]{$i$} (.5,.5); \draw[dir]
      (0,.5) to[out=-90,in=90] node[above, at start]{$j$} node[below, at end]{$j$}(.5,-.5);}=t_{ij} \tikz[baseline,very thick]{\draw[dir] (-.5,.5) to [out=-90,in=-90] node[above, at start]{$i$} node[above, at end]{$i$} (.5,.5); \draw[dir]
      (0,.5) to[out=-90,in=90] node[above, at start]{$j$} node[below, at end]{$j$}(-.5,-.5);}\qquad  \tikz[baseline,very thick]{\draw[dir] (-.5,-.5) to [out=90,in=90] node[below, at start]{$i$} node[below, at end]{$i$} (.5,-.5); \draw[edir]
      (-.5,.5) to[out=-90,in=90] node[above, at start]{$j$} node[below, at end]{$j$}(0,-.5);}=t_{ij} \tikz[baseline,very thick]{\draw[dir] (-.5,-.5) to [out=90,in=90] node[below, at start]{$i$} node[below, at end]{$i$} (.5,-.5); \draw[edir]
      (.5,.5) to[out=-90,in=90] node[above, at start]{$j$} node[below, at end]{$j$}(0,-.5);}.
  \end{equation*}
The mirror images of these relations through a vertical axis also hold.
\item Recall that a {\bf bubble} is a morphism given by a closed
  circle, endowed with some number of dots.  Any bubble of negative degree is zero,
  any bubble of degree 0 is equal to 1.  We must add formal symbols
  called ``fake bubbles'' which are bubbles labelled with a negative
  number of dots (these are explained in \cite[\S 3.1.1]{KLIII});
  given these, we have the inversion formula for bubbles:
  \begin{equation}\label{inv}
\begin{tikzpicture}[baseline]
\node at (-1,0) {$\displaystyle \sum_{k=\la^i-1}^{j+\la^i+1}$};
\draw[postaction={decorate,decoration={markings,
    mark=at position .5 with {\arrow[scale=1.3]{<}}}},very thick] (.5,0) circle (15pt);
\node [fill,circle,inner sep=2.5pt,label=right:{$k$},right=11pt] at (.5,0) {};
\node[scale=1.5] at (1.375,-.6){$\la$};
\draw[postaction={decorate,decoration={markings,
    mark=at position .5 with {\arrow[scale=1.3]{>}}}},very thick] (2.25,0) circle (15pt);
\node [fill,circle,inner sep=2.5pt,label=right:{$j-k$},right=11pt] at (2.25,0) {};
\node at (5.7,0) {$
  =\begin{cases}
    1 & j=-2\\
    0 & j>-2
  \end{cases}
$};
\end{tikzpicture}
\end{equation}

\item 2 relations connecting the crossing with cups and caps, shown in (\ref{lollipop1}-\ref{switch-2}).
\newseq
 \begin{equation*}\subeqn\label{lollipop1}
\begin{tikzpicture} [scale=1.3, baseline=35pt] 
\node[scale=1.5] at (-.7,1){$\la$};
\draw[postaction={decorate,decoration={markings,
    mark=at position .5 with {\arrow[scale=1.3]{>}}}},very thick] (0,0) to[out=90,in=-90]  (1,1) to[out=90,in=0]  (.5,1.5) to[out=180,in=90]  (0,1) to[out=-90,in=90]  (1,0);  
\node at (1.5,1.15) {$= \,-$}; 
\node at (2.2,1) {$\displaystyle\sum_{a+b=-1}$};
\draw[postaction={decorate,decoration={markings,
    mark=at position .5 with {\arrow[scale=1.3]{>}}}},very thick]
(3,0) to[out=90,in=180]  (3.5,.5) to[out=0,in=90]  node
[pos=.7,fill=black,circle,label={[label distance=3.5pt]right:{$a$}},inner sep=2.5pt]{} (4,0); 

\draw[postaction={decorate,decoration={markings,
    mark=at position .5 with {\arrow[scale=1.3]{>}}}},very thick] (3.5,1.3) circle (10pt);
\node [fill,circle,inner sep=2.5pt,label=right:{$b$},right=10.5pt] at (3.5,1.3) {};
\node[scale=1.5] at (4.7,1){$\la$};
\end{tikzpicture}
\end{equation*}
\begin{equation*}\subeqn \label{eq:1}
\begin{tikzpicture} [scale=1.3, baseline=35pt] 
\node[scale=1.5] at (-.7,1){$\la$};
\draw[postaction={decorate,decoration={markings,
    mark=at position .5 with {\arrow[scale=1.3]{<}}}},very thick] (0,0) to[out=90,in=-90]  (1,1) to[out=90,in=0]  (.5,1.5) to[out=180,in=90]  (0,1) to[out=-90,in=90]  (1,0);  
\node at (1.5,1.15) {$=$}; 
\node at (2.2,1) {$\displaystyle\sum_{a+b=-1}$};
\draw[postaction={decorate,decoration={markings,
    mark=at position .5 with {\arrow[scale=1.3]{<}}}},very thick]
(3,0) to[out=90,in=180]  (3.5,.5) to[out=0,in=90]  node
[pos=.7,fill=black,circle,label={[label distance=3.5pt]right:{$a$}},inner sep=2.5pt]{} (4,0); 

\draw[postaction={decorate,decoration={markings,
    mark=at position .5 with {\arrow[scale=1.3]{<}}}},very thick] (3.5,1.3) circle (10pt);
\node [fill,circle,inner sep=2.5pt,label=right:{$b$},right=10.5pt] at (3.5,1.3) {};
\node[scale=1.5] at (4.7,1){$\la$};
\end{tikzpicture}
\end{equation*}
\begin{equation*}\subeqn\label{switch-1}
\begin{tikzpicture}[baseline,scale=1.3]
\node at (0,0){
\begin{tikzpicture} [scale=1.3]
\node[scale=1.5] at (-.7,1){$\la$};
\draw[postaction={decorate,decoration={markings,
    mark=at position .5 with {\arrow[scale=1.3]{<}}}},very thick] (0,0) to[out=90,in=-90] (1,1) to[out=90,in=-90] (0,2);  
\draw[postaction={decorate,decoration={markings,
    mark=at position .5 with {\arrow[scale=1.3]{>}}}},very thick] (1,0) to[out=90,in=-90] (0,1) to[out=90,in=-90] (1,2);
\end{tikzpicture}
};
\node at (1.5,0) {$=$};
\node at (5.4,0){
\begin{tikzpicture} [scale=1.3, baseline=35pt]

\node[scale=1.5] at (.3,1){$\la$};

\node at (.7,1) {$-$};
\draw[postaction={decorate,decoration={markings,
    mark=at position .5 with {\arrow[scale=1.3]{<}}}},very thick] (1,0) to[out=90,in=-90]  (1,2);  
\draw[postaction={decorate,decoration={markings,
    mark=at position .5 with {\arrow[scale=1.3]{>}}}},very thick] (1.7,0) to[out=90,in=-90]  (1.7,2);

\node at (2.5,1.15) {$+$}; 
\node at (3,1) {$\displaystyle\sum_{a+b+c=-2}$};
\draw[postaction={decorate,decoration={markings,
    mark=at position .5 with {\arrow[scale=1.3]{<}}}},very thick] (4,0) to[out=90,in=-180] (4.5,.5) to[out=0,in=90] node [pos=.6, fill,circle,inner sep=2.5pt,label=above:{$a$}] {} (5,0);  
\draw[postaction={decorate,decoration={markings,
    mark=at position .5 with {\arrow[scale=1.3]{>}}}},very thick] (4,2) to[out=-90,in=-180] (4.5,1.5) to[out=0,in=-90] node [pos=.6, fill,circle,inner sep=2.5pt,label=below:{$c$}] {} (5,2);
\draw[postaction={decorate,decoration={markings,
    mark=at position .5 with {\arrow[scale=1.3]{<}}}},very thick] (5.5,1) circle (10pt);
\node [fill,circle,inner sep=2.5pt,label=right:{$b$},right=10.5pt] at (5.5,1) {};
\node[scale=1.5] at (6.5,1){$\la$};
\end{tikzpicture}
};
\end{tikzpicture}
\end{equation*}

\begin{equation*}\subeqn\label{switch-2}
\begin{tikzpicture}[scale=.9,baseline]
\node at (-3,0){
\scalebox{.95}{\begin{tikzpicture} [scale=1.3]
\node at (0,0){\begin{tikzpicture} [scale=1.3]
\node[scale=1.5] at (-.7,1){$\la$};
\draw[postaction={decorate,decoration={markings,
    mark=at position .5 with {\arrow[scale=1.3]{>}}}},very thick] (0,0) to[out=90,in=-90] (1,1) to[out=90,in=-90] (0,2);  
\draw[postaction={decorate,decoration={markings,
    mark=at position .5 with {\arrow[scale=1.3]{<}}}},very thick] (1,0) to[out=90,in=-90] (0,1) to[out=90,in=-90] (1,2);
\end{tikzpicture}};

\node at (1.5,0) {$=$};
\node at (5.4,0){
  {
\begin{tikzpicture} [scale=1.3, baseline=35pt]

\node[scale=1.5] at (.3,1){$\la$};

\node at (.7,1) {$-$};
\draw[postaction={decorate,decoration={markings,
    mark=at position .5 with {\arrow[scale=1.3]{>}}}},very thick] (1,0) to[out=90,in=-90]  (1,2);  
\draw[postaction={decorate,decoration={markings,
    mark=at position .5 with {\arrow[scale=1.3]{<}}}},very thick] (1.7,0) to[out=90,in=-90]  (1.7,2);

\node at (2.5,1.15) {$+$}; 
\node at (3,1) {$\displaystyle\sum_{a+b+c=-2}$};
\draw[postaction={decorate,decoration={markings,
    mark=at position .5 with {\arrow[scale=1.3]{>}}}},very thick] (4,0) to[out=90,in=-180] (4.5,.5) to[out=0,in=90] node [pos=.6, fill,circle,inner sep=2.5pt,label=above:{$a$}] {} (5,0);  
\draw[postaction={decorate,decoration={markings,
    mark=at position .5 with {\arrow[scale=1.3]{<}}}},very thick] (4,2) to[out=-90,in=-180] (4.5,1.5) to[out=0,in=-90] node [pos=.6, fill,circle,inner sep=2.5pt,label=below:{$c$}] {} (5,2);
\draw[postaction={decorate,decoration={markings,
    mark=at position .5 with {\arrow[scale=1.3]{>}}}},very thick] (5.5,1) circle (10pt);
\node [fill,circle,inner sep=2.5pt,label=right:{$b$},right=10.5pt] at (5.5,1) {};
\node[scale=1.5] at (6.5,1){$\la$};
\end{tikzpicture}
}
};
\end{tikzpicture}}};
\end{tikzpicture}
\end{equation*}

\item Oppositely oriented crossings of differently colored strands
  simply cancel with a scalar.\newseq
\begin{equation*}\subeqn\label{opp-cancel1}
    \begin{tikzpicture}[baseline]
      \node at (0,0){ \begin{tikzpicture} [scale=1.3] \node[scale=1.5]
        at (-.7,1){$\la$};
        \draw[postaction={decorate,decoration={markings, mark=at
            position .5 with {\arrow[scale=1.3]{<}}}},very thick]
        (0,0) to[out=90,in=-90] node[at start,below]{$i$} (1,1)
        to[out=90,in=-90] (0,2) ;
        \draw[postaction={decorate,decoration={markings, mark=at
            position .5 with {\arrow[scale=1.3]{>}}}},very thick]
        (1,0) to[out=90,in=-90] node[at start,below]{$j$} (0,1)
        to[out=90,in=-90] (1,2);
      \end{tikzpicture}};

    \node at (1.7,0) {$=$}; \node[scale=1.1] at (2.3,0) {$t_{ij}$};
    \node at (3.9,0){
      \begin{tikzpicture} [scale=1.3,baseline=35pt]

        \node[scale=1.5] at (2.4,1){$\la$};

        \draw[postaction={decorate,decoration={markings, mark=at
            position .5 with {\arrow[scale=1.3]{<}}}},very thick]
        (1,0) to[out=90,in=-90] node[at start,below]{$i$} (1,2);
        \draw[postaction={decorate,decoration={markings, mark=at
            position .5 with {\arrow[scale=1.3]{>}}}},very thick]
        (1.7,0) to[out=90,in=-90] node[at start,below]{$j$} (1.7,2);
      \end{tikzpicture}
    };
  \end{tikzpicture}
\end{equation*} 
\begin{equation*}\subeqn\label{opp-cancel2}
\begin{tikzpicture}[baseline]
\node at (0,0){\begin{tikzpicture} [scale=1.3]
\node[scale=1.5] at (-.7,1){$\la$};
\draw[postaction={decorate,decoration={markings,
    mark=at position .5 with {\arrow[scale=1.3]{>}}}},very thick] (0,0) to[out=90,in=-90] node[at start,below]{$i$} (1,1) to[out=90,in=-90] (0,2) ;  
\draw[postaction={decorate,decoration={markings,
    mark=at position .5 with {\arrow[scale=1.3]{<}}}},very thick] (1,0) to[out=90,in=-90] node[at start,below]{$j$} (0,1) to[out=90,in=-90] (1,2);
\end{tikzpicture}};

\node at (1.7,0) {$=$};
\node[scale=1.1] at (2.3,0) {$t_{ji}$};
\node at (3.9,0){
\begin{tikzpicture} [scale=1.3,baseline=35pt]

\node[scale=1.5] at (2.4,1){$\la$};

\draw[postaction={decorate,decoration={markings,
    mark=at position .5 with {\arrow[scale=1.3]{>}}}},very thick] (1,0) to[out=90,in=-90]  node[at start,below]{$i$ }(1,2) ;  
\draw[postaction={decorate,decoration={markings,
    mark=at position .5 with {\arrow[scale=1.3]{<}}}},very thick] (1.7,0) to[out=90,in=-90]  node[at start,below]{$j$} (1.7,2);
\end{tikzpicture}
};

\end{tikzpicture}
\end{equation*}

\item the endomorphisms of words only using only $-\Gamma$ (or by duality only $+\Gamma$) satisfy the relations of the {\bf quiver Hecke algebra} $R$.\newseq

\begin{equation*}\subeqn\label{first-QH}
    \begin{tikzpicture}[scale=.9,baseline]
      \draw[very thick,postaction={decorate,decoration={markings,
    mark=at position .2 with {\arrow[scale=1.3]{<}}}}](-4,0) +(-1,-1) -- +(1,1) node[below,at start]
      {$i$}; \draw[very thick,postaction={decorate,decoration={markings,
    mark=at position .2 with {\arrow[scale=1.3]{<}}}}](-4,0) +(1,-1) -- +(-1,1) node[below,at
      start] {$j$}; \fill (-4.5,.5) circle (3pt);
      \node at (-2,0){=}; \draw[very thick,postaction={decorate,decoration={markings,
    mark=at position .8 with {\arrow[scale=1.3]{<}}}}](0,0) +(-1,-1) -- +(1,1)
      node[below,at start] {$i$}; \draw[very thick,postaction={decorate,decoration={markings,
    mark=at position .8 with {\arrow[scale=1.3]{<}}}}](0,0) +(1,-1) --
      +(-1,1) node[below,at start] {$j$}; \fill (.5,-.5) circle (3pt);
      \node at (4,0){unless $i=j$};
    \end{tikzpicture}
  \end{equation*}
\begin{equation*}\subeqn
    \begin{tikzpicture}[scale=.9,baseline]
      \draw[very thick,postaction={decorate,decoration={markings,
    mark=at position .2 with {\arrow[scale=1.3]{<}}}}](-4,0) +(-1,-1) -- +(1,1) node[below,at start]
      {$i$}; \draw[very thick,postaction={decorate,decoration={markings,
    mark=at position .2 with {\arrow[scale=1.3]{<}}}}](-4,0) +(1,-1) -- +(-1,1) node[below,at
      start] {$j$}; \fill (-3.5,.5) circle (3pt);
      \node at (-2,0){=}; \draw[very thick,postaction={decorate,decoration={markings,
    mark=at position .8 with {\arrow[scale=1.3]{<}}}}](0,0) +(-1,-1) -- +(1,1)
      node[below,at start] {$i$}; \draw[very thick,postaction={decorate,decoration={markings,
    mark=at position .8 with {\arrow[scale=1.3]{<}}}}](0,0) +(1,-1) --
      +(-1,1) node[below,at start] {$j$}; \fill (-.5,-.5) circle (3pt);
      \node at (4,0){unless $i=j$};
    \end{tikzpicture}
  \end{equation*}
\begin{equation*}\subeqn\label{nilHecke-1}
    \begin{tikzpicture}[scale=.9,baseline]
      \draw[very thick,postaction={decorate,decoration={markings,
    mark=at position .2 with {\arrow[scale=1.3]{<}}}}](-4,0) +(-1,-1) -- +(1,1) node[below,at start]
      {$i$}; \draw[very thick,postaction={decorate,decoration={markings,
    mark=at position .2 with {\arrow[scale=1.3]{<}}}}](-4,0) +(1,-1) -- +(-1,1) node[below,at
      start] {$i$}; \fill (-4.5,.5) circle (3pt);
      \node at (-2,0){=}; \draw[very thick,postaction={decorate,decoration={markings,
    mark=at position .8 with {\arrow[scale=1.3]{<}}}}](0,0) +(-1,-1) -- +(1,1)
      node[below,at start] {$i$}; \draw[very thick,postaction={decorate,decoration={markings,
    mark=at position .8 with {\arrow[scale=1.3]{<}}}}](0,0) +(1,-1) --
      +(-1,1) node[below,at start] {$i$}; \fill (.5,-.5) circle (3pt);
      \node at (2,0){$+$}; \draw[very thick,postaction={decorate,decoration={markings,
    mark=at position .5 with {\arrow[scale=1.3]{<}}}}](4,0) +(-1,-1) -- +(-1,1)
      node[below,at start] {$i$}; \draw[very thick,postaction={decorate,decoration={markings,
    mark=at position .5 with {\arrow[scale=1.3]{<}}}}](4,0) +(0,-1) --
      +(0,1) node[below,at start] {$i$};
    \end{tikzpicture}
  \end{equation*}
 \begin{equation*}\subeqn\label{nilHecke-2}
    \begin{tikzpicture}[scale=.9,baseline]
      \draw[very thick,postaction={decorate,decoration={markings,
    mark=at position .8 with {\arrow[scale=1.3]{<}}}}](-4,0) +(-1,-1) -- +(1,1) node[below,at start]
      {$i$}; \draw[very thick,postaction={decorate,decoration={markings,
    mark=at position .8 with {\arrow[scale=1.3]{<}}}}](-4,0) +(1,-1) -- +(-1,1) node[below,at
      start] {$i$}; \fill (-4.5,-.5) circle (3pt);
      \node at (-2,0){=}; \draw[very thick,postaction={decorate,decoration={markings,
    mark=at position .2 with {\arrow[scale=1.3]{<}}}}](0,0) +(-1,-1) -- +(1,1)
      node[below,at start] {$i$}; \draw[very thick,postaction={decorate,decoration={markings,
    mark=at position .2 with {\arrow[scale=1.3]{<}}}}](0,0) +(1,-1) --
      +(-1,1) node[below,at start] {$i$}; \fill (.5,.5) circle (3pt);
      \node at (2,0){$+$}; \draw[very thick,postaction={decorate,decoration={markings,
    mark=at position .5 with {\arrow[scale=1.3]{<}}}}](4,0) +(-1,-1) -- +(-1,1)
      node[below,at start] {$i$}; \draw[very thick,postaction={decorate,decoration={markings,
    mark=at position .5 with {\arrow[scale=1.3]{<}}}}](4,0) +(0,-1) --
      +(0,1) node[below,at start] {$i$};
    \end{tikzpicture}
  \end{equation*}
  \begin{equation*}\subeqn\label{black-bigon}
    \begin{tikzpicture}[very thick,scale=.9,baseline]
      \draw[postaction={decorate,decoration={markings,
    mark=at position .5 with {\arrow[scale=1.3]{<}}}}] (-2.8,-1) .. controls (-1.2,0) ..  (-2.8,1)
      node[below,at start]{$i$}; \draw[postaction={decorate,decoration={markings,
    mark=at position .5 with {\arrow[scale=1.3]{<}}}}] (-1.2,-1) .. controls
      (-2.8,0) ..  (-1.2,1) node[below,at start]{$i$}; \node at (-.5,0)
      {=}; \node at (0.4,0) {$0$};
\node at (1.5,.05) {and};
    \end{tikzpicture}
\hspace{.4cm}
    \begin{tikzpicture}[very thick,scale=.9,baseline]

      \draw[postaction={decorate,decoration={markings,
    mark=at position .5 with {\arrow[scale=1.3]{<}}}}] (-2.8,-1) .. controls (-1.2,0) ..  (-2.8,1)
      node[below,at start]{$i$}; \draw[postaction={decorate,decoration={markings,
    mark=at position .5 with {\arrow[scale=1.3]{<}}}}] (-1.2,-1) .. controls
      (-2.8,0) ..  (-1.2,1) node[below,at start]{$j$}; \node at (-.5,0)
      {=}; 
\draw (1.8,0) +(0,-1) -- +(0,1) node[below,at start]{$j$};
      \draw (1,0) +(0,-1) -- +(0,1) node[below,at start]{$i$}; 
\node[inner xsep=10pt,fill=white,draw,inner ysep=8pt] at (1.4,0) {$Q_{ij}(y_1,y_2)$};
    \end{tikzpicture}
  \end{equation*}
 \begin{equation*}\subeqn\label{triple-dumb}
    \begin{tikzpicture}[very thick,scale=.9,baseline]
      \draw[postaction={decorate,decoration={markings,
    mark=at position .2 with {\arrow[scale=1.3]{<}}}}] (-3,0) +(1,-1) -- +(-1,1) node[below,at start]{$k$}; \draw[postaction={decorate,decoration={markings,
    mark=at position .8 with {\arrow[scale=1.3]{<}}}}]
      (-3,0) +(-1,-1) -- +(1,1) node[below,at start]{$i$}; \draw[postaction={decorate,decoration={markings,
    mark=at position .5 with {\arrow[scale=1.3]{<}}}}]
      (-3,-1) .. controls (-4,0) ..  (-3,1) node[below,at
      start]{$j$}; \node at (-1,0) {=}; \draw[postaction={decorate,decoration={markings,
    mark=at position .8 with {\arrow[scale=1.3]{<}}}}] (1,0) +(1,-1) -- +(-1,1)
      node[below,at start]{$k$}; \draw[postaction={decorate,decoration={markings,
    mark=at position .2 with {\arrow[scale=1.3]{<}}}}] (1,0) +(-1,-1) -- +(1,1)
      node[below,at start]{$i$}; \draw[postaction={decorate,decoration={markings,
    mark=at position .5 with {\arrow[scale=1.3]{<}}}}] (1,-1) .. controls
      (2,0) ..  (1,1) node[below,at start]{$j$}; \node at (5,0)
      {unless $i=k\neq j$};
    \end{tikzpicture}
  \end{equation*}
\begin{equation*}\subeqn\label{triple-smart}
    \begin{tikzpicture}[very thick,scale=.9,baseline]
      \draw[postaction={decorate,decoration={markings,
    mark=at position .2 with {\arrow[scale=1.3]{<}}}}] (-3,0) +(1,-1) -- +(-1,1) node[below,at start]{$i$}; \draw[postaction={decorate,decoration={markings,
    mark=at position .8 with {\arrow[scale=1.3]{<}}}}]
      (-3,0) +(-1,-1) -- +(1,1) node[below,at start]{$i$}; \draw[postaction={decorate,decoration={markings,
    mark=at position .5 with {\arrow[scale=1.3]{<}}}}]
      (-3,-1) .. controls (-4,0) ..  (-3,1) node[below,at
      start]{$j$}; \node at (-1,0) {=}; \draw[postaction={decorate,decoration={markings,
    mark=at position .8 with {\arrow[scale=1.3]{<}}}}] (1,0) +(1,-1) -- +(-1,1)
      node[below,at start]{$i$}; \draw[postaction={decorate,decoration={markings,
    mark=at position .2 with {\arrow[scale=1.3]{<}}}}] (1,0) +(-1,-1) -- +(1,1)
      node[below,at start]{$i$}; \draw[postaction={decorate,decoration={markings,
    mark=at position .5 with {\arrow[scale=1.3]{<}}}}] (1,-1) .. controls
      (2,0) ..  (1,1) node[below,at start]{$j$}; \node at (2.8,0)
      {$+$};        \draw (6.2,0)
      +(1,-1) -- +(1,1) node[below,at start]{$i$}; \draw (6.2,0)
      +(-1,-1) -- +(-1,1) node[below,at start]{$i$}; \draw (6.2,0)
      +(0,-1) -- +(0,1) node[below,at start]{$j$}; 
\node[inner ysep=8pt,inner xsep=5pt,fill=white,draw,scale=.8] at (6.2,0){$\displaystyle \frac{Q_{ij}(y_3,y_2)-Q_{ij}(y_1,y_2)}{y_3-y_1}$};
    \end{tikzpicture}
  \end{equation*}

\end{itemize}
  
\end{definition}

As in \cite{KLIII}, we let $\dU$ denote the strict 2-category where every
Hom-category is replaced by its idempotent completion. We note that since
every object in $\tU$ has a finite-dimensional degree 0 part of its endomorphism algebra,
every Hom-category in $\dU$ satisfies the Krull-Schmidt property.

This 2-category is a categorification of the universal enveloping
algebra in the sense that:
\begin{theorem}[\mbox{\cite[4.10]{Webunfurl}}]\label{categorify-U}
  The Grothendieck group of $\dU$ is isomorphic to $\dot{U}$ and its
  graded 
  Euler form is given by Lusztig's inner product $(-,-)$ on  $\dot{U}$.
\end{theorem}
This theorem was first conjectured by Khovanov and
Lauda \cite{KLIII} and proven by them in the special case of
$\mathfrak{sl}_n$.  While not explicitly stated in their paper, this
also follows easily from \cite[8.1]{CaLa} which was proved
independently of the work above, relying on the paper of Kang and
Kashiwara \cite{KK} in its stead.

We recall from \cite[\S 3.3.2]{KLIII} that we have an involution \[\tilde
\psi\colon \HOM(\eE_{i_1}\cdots \eE_{i_m}\la,\eE_{j_1}\cdots
\eE_{j_n}\la)\to \HOM(\eE_{j_1}\cdots
\eE_{j_n}\la,\eE_{i_1}\cdots \eE_{i_m}\la)\]
reflecting the diagrams of two morphisms through a horizontal line and
reversing orientation.
This extends to a 2-functor $\dU\to\dU$ which is covariant on
1-morphisms and contravariant on 2-morphisms, sending $\eE_i(k)\mapsto
\eE_i(-k), \eF_i(k)\mapsto
\eF_i(-k)$.

\begin{proposition}[\mbox{Khovanov-Lauda \cite[3.28]{KLIII}}]
  The 2-functor $\tilde{\psi}$ categorifies the bar involution of
  $\dot{U}_q(\fg)$ (denoted by $\psi$ in \cite{KLIII}).
\end{proposition}

This inner product and involution are part of the pre-canonical
structure used by Lusztig to define the canonical basis of $\dot{U}$;
the role of the standard basis can be played by a number of different
bases of $\dot{U}$.  We will use one defined using string parametrizations of crystal
elements. This is perhaps less elegant on the
level of the quantum groups than the PBW basis defined via the braid
action used by Lusztig in \cite{LusCB} (in particular, it is not
balanced as defined in \ref{sec:pre-canon-struct}), but is easier to handle in
the categorification.


\section{The 2-category $\mathcal{T}$}
\label{sec:2-category-ct}

In the next three sections, we will present a construction of a
categorification of tensor products of highest and lowest weight
representations.  Almost all of the results which appear have
equivalents in the author's earlier paper \cite{Webmerged}, and in most
cases, the nature of the proofs is quite similar.  First, we present
an auxiliary category which generalizes that presented in
\cite[\S \ref{m-sec:double-tens-prod}]{Webmerged}.

\begin{definition}
  A {\bf blank tricolore diagram}\footnote{When drawn on a blackboard,
    this diagram involves red, white and blue colors.  Citizens of
    Australia, Cambodia, Chile, the Cook Islands, Costa Rica, Croatia,
    Cuba, the Czech Republic, the Dominican Republic, Faroe Islands, France,
    Haiti, Iceland, North Korea, Laos, Liberia, Luxembourg, the
    Netherlands, Norway, Panama, Paraguay, Russia, Samoa, Serbia, Sint
    Maarten, Slovakia, Slovenia, Taiwan, Thailand, the United Kingdom
    and the United States are all free to regard this patriotically
    according to their preferences.} is a collection of
  finitely many oriented curves in
  $\R\times [0,1]$. Each curve is either
  \begin{itemize}
  \item colored red and labeled with a dominant weight of $\fg$, or
  \item colored blue and labeled with an anti-dominant weight of
    $\fg$, or 
  \item colored black and labeled with $i\in \Gamma$ and decorated with finitely many dots.
  \end{itemize}
  The red strands are constrained to be oriented downwards and the
  blue strands to be oriented upwards; we will generally not draw the
  orientation on these strands.  Furthermore, red and blue strands are
  forbidden to intersect with any other red or blue strand.  The
black strands are allowed to close into circles, self-intersect,
intersect red and blue strands, etc.

Blank tricolore diagrams divide their complement in $\R^2\times [0,1]$ into finitely
many connected components, and we define a {\bf tricolore
  diagram} to be a blank tricolore
  diagram together with a labeling of these regions
by weights consistent with the rules 
\begin{equation*}
  \tikz[baseline,very thick]{
\draw[wei,postaction={decorate,decoration={markings,
    mark=at position .5 with {\arrow[scale=1.3]{<}}}}] (0,-.5) -- node[below,at start]{$\la$}  (0,.5);
\node at (-1,0) {$\mu$};
\node at (1,.05) {$\mu+\la$};
}\qquad \qquad   \tikz[baseline,very thick]{
\draw[awei,postaction={decorate,decoration={markings,
    mark=at position .7 with {\arrow[scale=1.3]{>}}}}] (0,-.5) -- node[below,at start]{$-\la$}  (0,.5);
\node at (-1,0) {$\mu$};
\node at (1,.05) {$\mu-\la$};
}\qquad \qquad 
  \tikz[baseline,very thick]{
\draw[postaction={decorate,decoration={markings,
    mark=at position .5 with {\arrow[scale=1.3]{<}}}}] (0,-.5) -- node[below,at start]{$\la$}  (0,.5);
\node at (-1,0) {$\mu$};
\node at (1,.05) {$\mu-\al_i$};
}
\end{equation*}
Since this labeling is fixed as soon as one region is labeled, we will
typically not draw in the weights in all regions in the interest of
simplifying pictures.
\end{definition}  

For example,
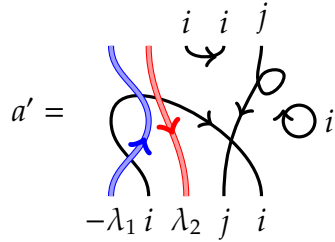
\[  a'=
\begin{tikzpicture}[baseline,very thick]
\draw [postaction={decorate,decoration={markings,
    mark=at position .7 with {\arrow[scale=1.3]{>}}}}] (-.5,-1) to[out=90,in=-90] node[below,at start]{$i$} (-1,0)
  to[out=90,in=90] node[below,at end]{$i$} (1,-1);
\draw [postaction={decorate,decoration={markings,
    mark=at position .4 with {\arrow[scale=1.3]{<}}}}] (.5,-1) to[out=90,in=-120] node[below,at start]{$j$} (1,.6)
  to[out=60,in=90] (1.3,.6) to [out=-90,in=-60] (1,.4)to[out=120,in=-90]  node[above,at end]{$j$} (1,1); 
  \draw[postaction={decorate,decoration={markings,
    mark=at position .5 with {\arrow[scale=1.3]{<}}}}]  (.5,1) to[out=-90,in=-90] node[above,at start]{$i$} node[above,at end]{$i$}(0,1);
 \draw[awei,postaction={decorate,decoration={markings,
    mark=at position .4 with {\arrow[scale=1]{>}}}}] (-1, -1) to[out=90,in=-90]node[below,at start]{$-\la_1$} (-.5,0) to[out=90,in=-90] (-1,1);
\draw[wei,postaction={decorate,decoration={markings,
    mark=at position .5 with {\arrow[scale=1]{<}}}}] (0,-1) to[out=90,in=-90]node[below,at start]{$\la_2$} (-.5,1);
\draw[postaction={decorate,decoration={markings,
    mark=at position .5 with {\arrow[scale=1.3]{<}}}}] (1.5,0) circle (6pt);  \node at (1.9,0){$i$};
\end{tikzpicture}
\]
is a blank tricolore diagram. Both the notion of {\bf KL diagrams} and {\bf double Stendhal
diagrams} from \cite{Webmerged} are special cases of
tricolore diagrams: a KLD is a tricolore diagram with no red
and blue strands and a DSD is a tricolore diagram with no blue
strands.

As usual, we will want to record the horizontal slices at $y=0$ and
$y=1$, the {\bf bottom} and {\bf top} of the diagram.  This will be
encoded as a {\bf tricolore quadruple}, consisting of 
\begin{itemize}
\item the sequence $\Bi\in (\pm \Gamma)^n$ of simple roots and their
  negatives on black strands, read from the left;
\item a sequence $\bla\in (\wela^{\pm})^\ell$ of dominant or
  anti-dominant weights on red and blue strands, read from the left;
\item the weakly increasing function $\kappa\colon [1,\ell]\to [0,n]$
  such that $\kappa(m)$ is the number of black strands left of $m$th red or blue strand (both counted
  from the left). By
convention, we write $\kappa(i)=0$, if the $i$th red or blue strand is left of all
black strands.  
\end{itemize}
We'll often condense these 3 items together into a single sequence
whose entries are both elements of $\pm\Gamma$ and $\wela^\pm$; we'll
typically denote such sequences with upper-case sans-serif letters,
such as $\sI$.
\begin{itemize}
\item the weights $\EuScript{L}$ and $\EuScript{R}$ at the far left and
  right of the diagram.  These are related by
\[\EuScript{L}+\sum_{k=1}^\ell\la_k+\sum_{m=1}^n\al_{i_m}=\EuScript{R}.\]
\end{itemize}

Tricolore diagrams are endowed with horizontal and vertical
composition operations, just like KL and DS diagrams; similarly
tricolore quadruples are endowed with a horizontal composition.  As in
\cite{Webmerged}, we maintain the dyslexic convention that the horizontal
composition $a\circ b$ places $a$ to the {\it right} of $b$ and we
read diagrams from bottom to top. 

\begin{definition}
  We let  $\doubletilde{\cT}$ be the strict 2-category where
 \begin{itemize}
  \item objects are weights in $X(\fg)$,
\item 1-morphisms $\mu\to \nu$ are tricolore quadruples with
  $\EuScript{L}=\mu,\EuScript{R}=\nu$ and composition is given by horizontal
  composition as above.
\item degraded 2-morphisms $h\to h'$ between tricolore quadruples are $\K$-linear combinations
  of tricolore diagrams  with $h$ as bottom
  and $h'$ as top, and vertical and horizontal composition of
  2-morphisms is defined above.  As with $\tU$, we should only take
  elements of 
  degree 0 as ``honest'' 2-morphisms.
\end{itemize}
\end{definition}

We can grade the 2-morphism spaces of this 2-category by endowing each
tricolore diagram with a degree.  KL diagrams are graded by the
degrees given in Section \ref{sec:2-category-cu}. The  red/black or
blue/black crossings have the following degrees, which are invariant under
reflection through a vertical line: \[
  \deg\tikz[baseline,very thick,scale=1.5]{\draw[->] (.2,.3) --
    (-.2,-.1) node[at end,below, scale=.8]{$i$}; \draw[wei] (.2,-.1) --
    (-.2,.3) node[at start,below,scale=.8]{$\la$};}
  =\langle\al_i,\la\rangle \qquad \deg\tikz[baseline,very thick,scale=1.5]{\draw[<-] (.2,.3) --
    (-.2,-.1) node[at end,below, scale=.8]{$i$}; \draw[wei] (.2,-.1) --
    (-.2,.3) node[at start,below,scale=.8]{$\la$};}
  =0 \qquad 
  \deg\tikz[baseline,very thick,scale=1.5]{\draw[->] (.2,.3) --
    (-.2,-.1) node[at end,below,scale=.8]{$i$}; \draw[awei] (.2,-.1) --
    (-.2,.3) node[at start,below,scale=.8]{$-\la$};}
  =0 \qquad 
  \deg\tikz[baseline,very thick,scale=1.5]{\draw[<-] (.2,.3) --
    (-.2,-.1) node[at end,below,scale=.8]{$i$}; \draw[awei] (.2,-.1) --
    (-.2,.3) node[at start,below,scale=.8]{$-\la$};}
  =\langle\al_i,\la\rangle \]
  \begin{definition}
    Let $\cT$ be\footnote{The author recognizes that the same symbol
      in used in \cite{Webmerged} to denote the subcategory of this one
      where no blue strands are allowed.  From context, we do not
      think that confusion is likely.} the quotient of $\doubletilde{\cT}$ by the following relations on 2-morphisms:
    \begin{itemize}
    \item All the relations of $\tU$ given in
      (\ref{pitch1}--\ref{triple-smart}) hold on black strands.
      \item \newseq 
Oppositely oriented crossings of differently labelled strands
  simply cancel, shown in (\ref{color-opp-cancel}).  This includes
  crossings of red/blue strands with black ones.
        \begin{equation}
\begin{tikzpicture} [scale=.9,baseline]

\node at (0,0){\begin{tikzpicture} [scale=1.1]
\node[scale=1.5] at (-.7,1){$\la$};
\draw[postaction={decorate,decoration={markings,
    mark=at position .5 with {\arrow[scale=1.3]{<}}}},very thick] (0,0) to[out=90,in=-90] node[at start,below]{$i$} (1,1) to[out=90,in=-90] (0,2) ;  
\draw[awei] (1,0) to[out=90,in=-90] node[at start,below]{$-\mu$} (0,1) to[out=90,in=-90] (1,2);
\end{tikzpicture}};

\node at (1.7,0) {$=$};
\node at (3.4,0){\begin{tikzpicture} [scale=1.3]

\node[scale=1.5] at (.3,1){$\la$};

\draw[postaction={decorate,decoration={markings,
    mark=at position .5 with {\arrow[scale=1.3]{<}}}},very thick] (1,0)--  (1,2) node[at start,below]{$i$};  
\draw[awei] (1.7,0) --  (1.7,2) node[at start,below]{$-\mu$};
\end{tikzpicture}};
\node at (7,0){\begin{tikzpicture} [scale=1.3]
\node[scale=1.5] at (-.7,1){$\la$};
\draw[postaction={decorate,decoration={markings,
    mark=at position .5 with {\arrow[scale=1.3]{>}}}},very thick] (0,0) to[out=90,in=-90] node[at start,below]{$i$} (1,1) to[out=90,in=-90] (0,2) ;  
\draw[wei] (1,0) to[out=90,in=-90] node[at start,below]{$\mu$} (0,1) to[out=90,in=-90] (1,2);
\end{tikzpicture}};

\node at (8.7,0) {$=$};
\node at (10.4,0){\begin{tikzpicture} [scale=1.3]

\node[scale=1.5] at (.3,1){$\la$};

\draw[postaction={decorate,decoration={markings,
    mark=at position .5 with {\arrow[scale=1.3]{>}}}},very thick] (1,0) -- (1,2) node[at start,below]{$i$};  
\draw[wei] (1.7,0) --  (1.7,2) node[at start,below]{$\mu$};
\end{tikzpicture}};

\end{tikzpicture}
\label{color-opp-cancel}
\end{equation}

\item  All black crossings and dots can pass through red or blue
  lines, with a correction term similar to Khovanov and Lauda's (for
  the relations of (\ref{fig:pass-through}-\ref{pitch}), we also include their mirror images through
  a vertical line), as shown below:\newseq
    \begin{equation*}\subeqn\label{blue-triple}
      \begin{tikzpicture}[very thick,baseline]
          \draw[postaction={decorate,decoration={markings,
    mark=at position .2 with {\arrow[scale=1.3]{>}}}}] (-3,0) +(1,-1) -- +(-1,1) node[at start,below]{$i$};
          \draw[postaction={decorate,decoration={markings,
    mark=at position .8 with {\arrow[scale=1.3]{>}}}}] (-3,0) +(-1,-1) -- +(1,1)node [at start,below]{$j$};
          \draw[awei](-3,-1) .. controls (-4,0) ..  (-3,1)
          node [at start, below]{$-\la$};
          \node at (-1,0) {=}; \draw[postaction={decorate,decoration={markings,
    mark=at position .8 with {\arrow[scale=1.3]{>}}}}] (1,0) +(1,-1) -- +(-1,1) node[at
          start,below]{$i$}; \draw[postaction={decorate,decoration={markings,
    mark=at position .2 with {\arrow[scale=1.3]{>}}}}] (1,0) +(-1,-1) -- +(1,1) node [at
          start,below]{$j$}; \draw[awei] (1,-1) .. controls
          (2,0) ..  (1,1)  node [at start, below]{$-\la$}; \node at (2.8,0) {$-$}; \draw[postaction={decorate,decoration={markings,
    mark=at position .2 with {\arrow[scale=1.3]{>}}}}] (6.5,0)
          +(1,-1) -- +(1,1) node[midway,circle,fill,inner
          sep=2.5pt,label=right:{$a$}]{} node[at start,below]{$i$};
          \draw[postaction={decorate,decoration={markings,
    mark=at position .2 with {\arrow[scale=1.3]{>}}}}] (6.5,0) +(-1,-1) -- +(-1,1)
          node[midway,circle,fill,inner sep=2.5pt,label=left:{$b$}]{}
          node [at start,below]{$j$}; \draw[awei] (6.5,0) +(0,-1) --
          +(0,1)  node [at start, below]{$-\la$}; \node at (3.8,-.2){$\displaystyle \sum_{a+b-1=\la^i}
            \delta_{i,j} $} ;
        \end{tikzpicture}
      \end{equation*}
   \begin{equation*}\subeqn\label{red-triple}
   \begin{tikzpicture}[very thick,baseline]
          \draw (-3,0)[postaction={decorate,decoration={markings,
    mark=at position .2 with {\arrow[scale=1.3]{<}}}}] +(1,-1) -- +(-1,1) node[at start,below]{$i$};
          \draw[postaction={decorate,decoration={markings,
    mark=at position .8 with {\arrow[scale=1.3]{<}}}}] (-3,0) +(-1,-1) -- +(1,1)node [at start,below]{$j$};
          \draw[wei] (-3,-1) .. controls (-4,0) ..  (-3,1)  node [at start, below]{$\la$};
          \node at (-1,0) {=}; \draw[postaction={decorate,decoration={markings,
    mark=at position .8 with {\arrow[scale=1.3]{<}}}}] (1,0) +(1,-1) -- +(-1,1) node[at
          start,below]{$i$}; \draw[postaction={decorate,decoration={markings,
    mark=at position .2 with {\arrow[scale=1.3]{<}}}}] (1,0) +(-1,-1) -- +(1,1) node [at
          start,below]{$j$}; \draw[wei] (1,-1) .. controls
          (2,0) ..  (1,1) node [at start, below]{$\la$}; \node at (2.8,0) {$+ $}; \draw[postaction={decorate,decoration={markings,
    mark=at position .8 with {\arrow[scale=1.3]{<}}}}] (6.5,0)
          +(1,-1) -- +(1,1) node[midway,circle,fill,inner
          sep=2.5pt,label=right:{$a$}]{} node[at start,below]{$i$};
          \draw[postaction={decorate,decoration={markings,
    mark=at position .8 with {\arrow[scale=1.3]{<}}}}] (6.5,0) +(-1,-1) -- +(-1,1)
          node[midway,circle,fill,inner sep=2.5pt,label=left:{$b$}]{}
          node [at start,below]{$j$}; \draw[wei] (6.5,0) +(0,-1) --
          +(0,1) node [at start, below]{$\la$}; \node at (3.8,-.2){$\displaystyle \sum_{a+b-1=\la^i}
            \delta_{i,j} $} ;
        \end{tikzpicture}
      \end{equation*}
\begin{equation*}\subeqn \label{fig:pass-through1}
   \begin{tikzpicture}[yscale=.9,baseline]
      \node at (3.6,1.5){    \begin{tikzpicture}[very thick]
          \draw[postaction={decorate,decoration={markings,
    mark=at position .2 with {\arrow[scale=1.3]{>}}}}] (-3,0) +(1,-1) -- +(-1,1);
          \draw[postaction={decorate,decoration={markings,
    mark=at position .8 with {\arrow[scale=1.3]{<}}}}] (-3,0) +(-1,-1) -- +(1,1);
          \draw[awei] (-3,0) +(0,-1) .. controls (-4,0) ..  +(0,1);
          \node at (-1,0) {=}; \draw[postaction={decorate,decoration={markings,
    mark=at position .8 with {\arrow[scale=1.3]{>}}}}] (1,0) +(1,-1) -- +(-1,1); \draw[postaction={decorate,decoration={markings,
    mark=at position .2 with {\arrow[scale=1.3]{<}}}}] (1,0) +(-1,-1) -- +(1,1); \draw[awei] (1,0) +(0,-1) .. controls
          (2,0) ..  +(0,1); 
        \end{tikzpicture}};
     \node at (-3.6,1.5){    \begin{tikzpicture}[very thick]
       \draw (-3,0)[postaction={decorate,decoration={markings,
    mark=at position .2 with {\arrow[scale=1.3]{<}}}}] +(1,-1) -- +(-1,1);
          \draw[postaction={decorate,decoration={markings,
    mark=at position .8 with {\arrow[scale=1.3]{>}}}}] (-3,0) +(-1,-1) -- +(1,1);
          \draw[wei] (-3,0) +(0,-1) .. controls (-4,0) ..  +(0,1);
          \node at (-1,0) {=}; \draw[postaction={decorate,decoration={markings,
    mark=at position .8 with {\arrow[scale=1.3]{<}}}}] (1,0) +(1,-1) -- +(-1,1); \draw[postaction={decorate,decoration={markings,
    mark=at position .2 with {\arrow[scale=1.3]{>}}}}] (1,0) +(-1,-1) -- +(1,1); \draw[wei] (1,0) +(0,-1) .. controls
          (2,0) ..  +(0,1); 
        \end{tikzpicture}};
  \node at (3.6,-1.5){    \begin{tikzpicture}[very thick]
          \draw[postaction={decorate,decoration={markings,
    mark=at position .2 with {\arrow[scale=1.3]{<}}}}] (-3,0) +(1,-1) -- +(-1,1);
          \draw[postaction={decorate,decoration={markings,
    mark=at position .8 with {\arrow[scale=1.3]{>}}}}] (-3,0) +(-1,-1) -- +(1,1);
          \draw[awei] (-3,0) +(0,-1) .. controls (-4,0) ..  +(0,1);
          \node at (-1,0) {=}; \draw[postaction={decorate,decoration={markings,
    mark=at position .8 with {\arrow[scale=1.3]{<}}}}] (1,0) +(1,-1) -- +(-1,1); \draw[postaction={decorate,decoration={markings,
    mark=at position .2 with {\arrow[scale=1.3]{>}}}}] (1,0) +(-1,-1) -- +(1,1); \draw[awei] (1,0) +(0,-1) .. controls
          (2,0) ..  +(0,1); 
        \end{tikzpicture}};
     \node at (-3.6,-1.5){    \begin{tikzpicture}[very thick]
       \draw (-3,0)[postaction={decorate,decoration={markings,
    mark=at position .2 with {\arrow[scale=1.3]{>}}}}] +(1,-1) -- +(-1,1);
          \draw[postaction={decorate,decoration={markings,
    mark=at position .8 with {\arrow[scale=1.3]{<}}}}] (-3,0) +(-1,-1) -- +(1,1);
          \draw[wei] (-3,0) +(0,-1) .. controls (-4,0) ..  +(0,1);
          \node at (-1,0) {=}; \draw[postaction={decorate,decoration={markings,
    mark=at position .8 with {\arrow[scale=1.3]{>}}}}] (1,0) +(1,-1) -- +(-1,1); \draw[postaction={decorate,decoration={markings,
    mark=at position .2 with {\arrow[scale=1.3]{<}}}}] (1,0) +(-1,-1) -- +(1,1); \draw[wei] (1,0) +(0,-1) .. controls
          (2,0) ..  +(0,1); 
        \end{tikzpicture}};
      \end{tikzpicture}
    \end{equation*}\begin{equation*}\subeqn \label{fig:pass-through}
   \begin{tikzpicture}[yscale=.9,baseline]
      \node at (3.6,1.5){    \begin{tikzpicture}[very thick]
          \draw[postaction={decorate,decoration={markings,
    mark=at position .2 with {\arrow[scale=1.3]{<}}}}] (-3,0) +(1,-1) -- +(-1,1);
          \draw[postaction={decorate,decoration={markings,
    mark=at position .8 with {\arrow[scale=1.3]{<}}}}] (-3,0) +(-1,-1) -- +(1,1);
          \draw[awei] (-3,0) +(0,-1) .. controls (-4,0) ..  +(0,1);
          \node at (-1,0) {=}; \draw[postaction={decorate,decoration={markings,
    mark=at position .8 with {\arrow[scale=1.3]{<}}}}] (1,0) +(1,-1) -- +(-1,1); \draw[postaction={decorate,decoration={markings,
    mark=at position .2 with {\arrow[scale=1.3]{<}}}}] (1,0) +(-1,-1) -- +(1,1); \draw[awei] (1,0) +(0,-1) .. controls
          (2,0) ..  +(0,1); 
        \end{tikzpicture}};
     \node at (-3.6,1.5){    \begin{tikzpicture}[very thick]
       \draw (-3,0)[postaction={decorate,decoration={markings,
    mark=at position .2 with {\arrow[scale=1.3]{>}}}}] +(1,-1) -- +(-1,1);
          \draw[postaction={decorate,decoration={markings,
    mark=at position .8 with {\arrow[scale=1.3]{>}}}}] (-3,0) +(-1,-1) -- +(1,1);
          \draw[wei] (-3,0) +(0,-1) .. controls (-4,0) ..  +(0,1);
          \node at (-1,0) {=}; \draw[postaction={decorate,decoration={markings,
    mark=at position .8 with {\arrow[scale=1.3]{>}}}}] (1,0) +(1,-1) -- +(-1,1); \draw[postaction={decorate,decoration={markings,
    mark=at position .2 with {\arrow[scale=1.3]{>}}}}] (1,0) +(-1,-1) -- +(1,1); \draw[wei] (1,0) +(0,-1) .. controls
          (2,0) ..  +(0,1); 
        \end{tikzpicture}};
    \node at (-3.6,-1.5){    \begin{tikzpicture}[very thick]
          \draw[postaction={decorate,decoration={markings,
    mark=at position .8 with {\arrow[scale=1.3]{<}}}}](-3,0) +(-1,-1) -- +(1,1); \draw[wei](-3,0) +(1,-1) --
          +(-1,1); \fill (-3.5,-.5) circle (3pt); \node at (-1,0) {=};
          \draw[postaction={decorate,decoration={markings,
    mark=at position .2 with {\arrow[scale=1.3]{<}}}}](1,0) +(-1,-1) -- +(1,1); \draw[wei](1,0) +(1,-1) --
          +(-1,1); \fill (1.5,0.5) circle (3pt);
        \end{tikzpicture}};
    \node at (3.6,-1.5){    \begin{tikzpicture}[very thick]
          \draw[postaction={decorate,decoration={markings,
    mark=at position .8 with {\arrow[scale=1.3]{<}}}}](-3,0) +(-1,-1) -- +(1,1); \draw[awei](-3,0) +(1,-1) --
          +(-1,1); \fill (-3.5,-.5) circle (3pt); \node at (-1,0) {=};
          \draw[postaction={decorate,decoration={markings,
    mark=at position .2 with {\arrow[scale=1.3]{<}}}}](1,0) +(-1,-1) -- +(1,1); \draw[awei](1,0) +(1,-1) --
          +(-1,1); \fill (1.5,0.5) circle (3pt);
        \end{tikzpicture}};
      \end{tikzpicture}
    \end{equation*}
 \begin{equation*}\subeqn
 \begin{tikzpicture}[very thick,baseline]\label{pitch}
      \draw[postaction={decorate,decoration={markings,
    mark=at position .9 with {\arrow[scale=1.3]{>}}}}] (-3,0)  +(1,-1)
to[out=90,in=0] +(0,0) to[out=180,in=90] +(-1,-1);
      \draw[wei] (-3,0)  +(0,-1) .. controls (-4,0) ..  +(0,1);
      \node at (-1,0) {=};
      \draw[postaction={decorate,decoration={markings,
    mark=at position .9 with {\arrow[scale=1.3]{>}}}}] (1,0)  +(1,-1)
to[out=90,in=0] +(0,0) to[out=180,in=90] +(-1,-1);
      \draw[wei] (1,0) +(0,-1) .. controls (2,0) ..  +(0,1);   
 \end{tikzpicture}\qquad    \begin{tikzpicture}[very thick,baseline]
      \draw[postaction={decorate,decoration={markings,
    mark=at position .9 with {\arrow[scale=1.3]{>}}}}] (-3,0)  +(1,-1)
to[out=90,in=0] +(0,0) to[out=180,in=90] +(-1,-1);
      \draw[awei] (-3,0)  +(0,-1) .. controls (-4,0) ..  +(0,1);
      \node at (-1,0) {=};
      \draw[postaction={decorate,decoration={markings,
    mark=at position .9 with {\arrow[scale=1.3]{>}}}}] (1,0)  +(1,-1)
to[out=90,in=0] +(0,0) to[out=180,in=90] +(-1,-1);
      \draw[awei] (1,0) +(0,-1) .. controls (2,0) ..  +(0,1);   
 \end{tikzpicture}
  \end{equation*}

\item The ``cost'' of a separating similarly oriented red/blue and
  black lines is adding $\la^i=\al_i^\vee(\la)$ dots to the black
  strand as shown in (\ref{cost}):
     \begin{equation}\label{cost}
      \begin{tikzpicture}[very thick,baseline=1.45cm]
        \draw[postaction={decorate,decoration={markings,
    mark=at position .51 with {\arrow[scale=1.3]{<}}}}]  (-2.8,-1) .. controls (-1.2,0) ..  (-2.8,1)
        node[below,at start]{$i$}; \draw[wei] (-1.2,0) +(0,-1)
        .. controls (-2.8,0) ..  +(0,1) node[below,at start]{$\la$};
        \node at (-.3,0) {=}; \draw[wei] (2.8,0) +(0,-1) -- +(0,1)
        node[below,at start]{$\la$}; \draw[postaction={decorate,decoration={markings,
    mark=at position .2 with {\arrow[scale=1.3]{<}}}}]  (1.2,0) +(0,-1) -- +(0,1)
        node[below,at start]{$i$}; \fill (1.2,0) circle (3pt)
        node[left=3pt]{$\la^i$}; \draw[wei] (-2.8,3) +(0,-1)
        .. controls (-1.2,3) ..  +(0,1) node[below,at start]{$\la$};
        \draw[postaction={decorate,decoration={markings,
    mark=at position .51 with {\arrow[scale=1.3]{<}}}}]  (-1.2,3) +(0,-1) .. controls (-2.8,3) ..  +(0,1)
        node[below,at start]{$i$}; \node at (-.3,3) {=}; \draw[postaction={decorate,decoration={markings,
    mark=at position .2 with {\arrow[scale=1.3]{<}}}}]  (2.8,3)
        +(0,-1) -- +(0,1) node[below,at start]{$i$}; \draw[wei]
        (1.2,3) +(0,-1) -- +(0,1) node[below,at start]{$\la$}; \fill
        (2.8,3) circle (3pt) node[right=3pt]{$\la^i$};
      \end{tikzpicture}\quad
      \begin{tikzpicture}[very thick,baseline=1.6cm]
        \draw[postaction={decorate,decoration={markings,
    mark=at position .5 with {\arrow[scale=1.3]{>}}}}]  (-2.8,-1) .. controls (-1.2,0) ..  (-2.8,1)
        node[below,at start]{$i$}; \draw[awei] (-1.2,0) +(0,-1)
        .. controls (-2.8,0) ..  +(0,1) node[below,at start]{$-\la$};
        \node at (-.3,0) {=}; \draw[awei] (2.8,0) +(0,-1) -- +(0,1)
        node[below,at start]{$-\la$}; \draw[postaction={decorate,decoration={markings,
    mark=at position .8 with {\arrow[scale=1.3]{>}}}}]  (1.2,0) +(0,-1) -- +(0,1)
        node[below,at start]{$i$}; \fill (1.2,0) circle (3pt)
        node[left=3pt]{$\la^i$}; \draw[awei] (-2.8,3) +(0,-1)
        .. controls (-1.2,3) ..  +(0,1) node[below,at start]{$-\la$};
        \draw[postaction={decorate,decoration={markings,
    mark=at position .5 with {\arrow[scale=1.3]{>}}}}]  (-1.2,3) +(0,-1) .. controls (-2.8,3) ..  +(0,1)
        node[below,at start]{$i$}; \node at (-.3,3) {=}; \draw[postaction={decorate,decoration={markings,
    mark=at position .8 with {\arrow[scale=1.3]{>}}}}]  (2.8,3)
        +(0,-1) -- +(0,1) node[below,at start]{$i$}; \draw[awei]
        (1.2,3) +(0,-1) -- +(0,1) node[below,at start]{$-\la$}; \fill
        (2.8,3) circle (3pt) node[right=3pt]{$\la^i$};
      \end{tikzpicture}
    \end{equation}
\end{itemize}
 \end{definition}

The category $\cT$ has a subcategory, which we will call the {\bf
  Polish\footnote{No slight to the residents of Austria, Bahrain,
    etc. intended.} subcategory} given by diagrams with no blue lines,
only red and black.  We have a complementary  {\bf
  Scottish\footnote{Similar apologies to Salvadore\~nos, Finns, etc.}
  subcategory} of diagrams with no red strands, only
blue and black.

As with $\tU$ and $\dU$, we let $\dT$ be the idempotent completion of
the Hom-categories of $\cT$.
We have an obvious 2-functor $\tU\to \cT$ (and thus $\dU\to \dT$),
thinking of a KL diagram as a tricolore diagram.
The 2-category $\cT$ thus carries an action of $\tU$ by horizontal
composition on the right and on the left.  

Unfortunately, as usual with a presentation by
generators and relations, it is far from obvious that 
the category
even has any non-zero objects.

The morphism space between any two sequences $(\bla,\Bi,\kappa)$ and
$(\bla,\Bi',\kappa')$ has an obvious spanning set $D$.  As usual,
there are actually many such spanning sets, depending on certain
choices.  The elements of $D$ are in canonical bijection with the
basis $B_{\Bi,\Bi'}$ of Khovanov and Lauda, defined in \cite[\S
2.2]{KLIII}.  Each basis vector is given by a KL diagram; we then choose a way to insert the
red and blue strands into this diagram which preserves the tricolore
conditions.  We can show that this set is a basis by using a
deformation of this category and reducing to the non-degeneracy for
cyclotomic quotients proven in \cite{Webmerged} and \cite{KK}.  This
is carried out in \cite{Webunfurl}\footnote{In the published version
  of this paper, a different proof was given.  This proof used a
  localization of $\tU$; proving that this localization has large
  enough morphism spaces to prove Theorem \ref{C-basis} referred to
  \cite{Webmerged}, but ultimately required some
  additional arguments and a modification of the definition.  These
  changes were carried out in \cite{Webunfurl}, thus we just refer to
  that paper for these results.  A corrigendum will be published,
  noting these changes in the published version.}:
\addtocounter{theorem}{1}
\begin{proposition}[\mbox{\cite[4.9]{Webunfurl}}]\label{C-basis}
  For all pairs of tricolore triples, the set $D$ is a basis for the morphism space
  $\HOM_{\cT} \big((\bla,\Bi,\kappa),(\bla,\Bi',\kappa')\big)$.
\end{proposition}

Just as on $\tU$, the 2-category $\cT$ has an autofunctor flipping diagrams which is
covariant on 1-morphisms and contravariant on 2-morphisms.  Abusing
notation, we
will also denote this $\tilde{\psi}$. 

\section{Tensor product algebras}
\label{sec:tens-prod-algebr}

The category $\cT$ is quite auxiliary from our
perspective.  The fundamental object of this paper is an
induced module category over this 2-category.  We wish to consider
representations of $\tU$; for our purposes, this means strict
2-functors $\tU\to \mathsf{Cat}$ to the strict 2-category of
categories, functors and natural transformations.  

Recall that the category $\tU$ has a ``trivial'' representation on
$\Vect_\K$.  Every 1-morphism corresponding to a KL pair
with $\Bi\neq \emptyset$ acts
by the zero functor, as does the identity 1-morphism of any
non-zero weight, while $\operatorname{id}_0\cdot V\cong V$ for all
vector spaces.
\begin{definition}
  We let $\hl$ denote the ``induction'' of this representation to $\cT$.
  That is, an object of $\hl$ is a sum of 1-morphisms of $\cT$ formally
  applied to objects of $\Vect_\K$.  In addition to the morphisms given
  by tensor products, we also add a natural isomorphism
\[tu\cdot V\cong t\cdot uV  \qquad t\in \Hom_{\cT}(\la,\mu), u\in
\Hom_\tU(\mu,\nu), V\in \Ob(\Vect_\K).\]
\end{definition}
Remember that our convention for switching between formulas and
diagrams is ``dyslexic;'' it switches left and right.  In essence,
thus $\hl$ is the quotient of all diagrams in $\cT$ (which we view as
objects in $\hl$ by tensoring with $\K$ itself) with a $\eF_i$ or
$\eE_i$ or a weight other than $0$ at the far {\em left}, since we can
move these over to act (trivially) on the vector space $\K$.  The
reader is free to imagine the object $\K$ as a horde of zombies at the
far left of the plane which hungrily eats any black strand or
non-zero weight it can reach, but which is unable to pass through red
or blue lines.

Obviously, the only tricolore quadruples will survive are those where
$\EuScript{L}=0$.  Thus, we define a {\bf tricolore triple} to be a
tricolore quadruple with the weight $\EuScript{L}$ left out and
understood to be 0.

The category $\hl$ still carries a $\tU$-action by horizontal
composition on the right; note that $\tU$ is unable to change the
labeling or ordering of the red and blue strands.

\begin{definition}
  Let $\hl^\bla$ denote the subcategory consisting of
  all 1-morphisms (now thought of as object of $\hl$) where the
  sequence of labels on red and blue lines is exactly $\bla$. This
  inherits an action of $\tU$ from $\hl$. Let
  $\hl^\bla_\mu$ be the subcategory of $\hl^\bla$ where the weight
  $\EuScript{R}$ is $\mu$.
\end{definition}

Recall that the author has already defined a categorification of the
tensor product of highest weight representations, based on certain
algebras $T^\bla$, defined in \cite[\S\ref{m-sec:defn}]{Webmerged}; these are, in fact, a
special case of the categorifications we discuss in this paper.

\begin{theorem}\label{red-tensor}
  If $\bla$ consists only of dominant weights, then $\hl^\bla\cong T^\bla\pmodu$.
\end{theorem}
\begin{proof}
If one sums over all triciolore triples, then the resulting object has
endomorphism algebra given by the algebra $DT^\bla$ defined in
\cite[\S\ref{m-defn-double-tensor}]{Webmerged}.  This is Morita
equivalent to $T^\bla$ by \cite[\S\ref{m-Morita}]{Webmerged}. 
\excise{ Obviously, if one takes the direct sum over all tricolore triples using negative simple
  roots, then one finds an object whose endomorphisms are $T^\bla$.
  Thus, we need only show that such tricolore triples generate our category
  (as an additive, idempotent complete category).  

  Using the relations (\ref{switch-1}) and \eqref{color-opp-cancel} to pass all $\eE_i$'s leftward (in the
  diagram; toward the zombies), we can always write a tricolore triple as an
  summand in a sum of tricolore triples with $\eE_i$'s further left or with
  fewer $\eE_i$'s; here it is crucial that there are no blue strands,
  since $\eE_i$'s cannot pass through these.  By induction, every
  tricolore triple is a summand of tricolore triples where all $\eE_i$'s come at the
  far left.  Thus, the only ones that give non-zero objects are those
  with no $\eE_i$'s at all.}
\end{proof}
For future results, we must have a precise notion of equivariant
morphisms between representations of $\tU$.  Let
$\aleph_1,\aleph_2\colon \tU\to \mathsf{Cat}$ be two strict
2-functors.
\begin{definition}
  A {\bf strongly equivariant} functor $\be$ is a collection of
  functors $\be(\la)\colon \aleph_1(\la) \to \aleph_2(\la)$ together
  with natural isomorphisms of functors $c_u\colon \be\circ\aleph_1(u)\cong
  \aleph_2(u)\circ \be$ for every 1-morphism $u\in \tU$ such that 
 \[c_v\circ (\id_{\be}\otimes\,  \aleph_1(\al))= (\aleph_2(\al) \otimes
 \id_{\be}) \circ c_u\] for every 2-morphism $\al\colon u\to v$ in
 $\tU$. (Here we use $\otimes$ for horizontal composition, and $\circ$
 for vertical composition of 2-morphisms). 
\end{definition}
As usual, we let $\hl^\la=\hl^{(\la)}$.

\begin{theorem}\label{highest-lowest}
  If $\fg$ is finite dimensional, we have a strongly $\tU$-equivariant equivalence $\hl^{\la}\cong \hl^{w_0\la}$. 
\end{theorem}
\begin{proof}
By symmetry, it suffices to assume $\la$ is dominant.  

Consider the $\tU$-module $\hl^{w_0\la}$; since the Grothendieck group
of this category is an irreducible module, the Jordan-H\"older
filtration introduced by Rouquier in \cite[\S 5]{Rou2KM} must have a
single step.  That is, by \cite[5.8]{Rou2KM} we have that the category $\hl^{w_0\la}$ is
strongly equivariantly equivalent to a base change category, which Rouquier denotes
$\mathcal{L}(\la)\otimes_{\mathrm{End}(\bar 1_\la)}\hl^{w_0\la}_\la$.  
By
\cite[\ref{m-universal}]{Webmerged}, we can unpack this a bit more
explicitly: the category $\hl^{w_0\la}$ is equivalent to the category
of projective modules over an algebra
$\check{R}^\la\otimes_{\check{R}^\la_\la}A$ where 
\begin{itemize}
\item $\check{R}^\la$ is a free deformation of $T^\la$ defined and
  shown to be free in
  \cite[\S 2.6]{Webmerged}; the deformation base $\check{R}^\la_\la$ can be identified with a
  polynomial ring freely generated by the fake bubble endomorphisms of
  $\id_\la$ in $\tU$.  
\item $A$ is an Artinian algebra such that the weight space category
  $\hl^{w_0\la}_\la$ is equivalent to $A\pmodu$; this inherits a map
  $\check{R}^\la_\la\to A$ from the action of the endomorphisms of $\id_\la$ in $\tU$.
\end{itemize}
Since $\hl^{w_0\la}$ has a unique simple module, $A$ is in fact
local.  
By the freeness of $\check{R}^\la$ over $\check{R}^\la_\la$, the base
change $\check{R}^\la\otimes_{\check{R}^\la_\la}A$ is free over $A$,
and thus so is any projective over
$\check{R}^\la\otimes_{\check{R}^\la_\la}A$.  Thus, the morphism space
between any two objects in $\hl^{w_0\la}$ must be a free $A$ module.  

Since the homomorphism space between the tricolore triples with
$\Bi=\emptyset$ is 1-dimensional, this is only possible if $A=\K$.
Thus, we have that $\hl^{w_0\la}$ is strongly equivariantly equivalent
to $\hl^{\la}$.  
\end{proof}

One very interesting special case is when there is one red and one
blue line; assume that $-\la$ and $\mu$ are dominant and consider
$\hl^{(\la,\mu)}$.  
\begin{proposition}\label{lm-algebra}
  Every object in $\hl^{(\la,\mu)}$ is a summand of $(\la,\mu,\Bi)$
  for some $\Bi$.  The morphism space $(\la,\mu,\Bi)\to (\la,\mu,\Bj)$
  is the quotient of the morphisms $\HOM_{\tU}(\Bi,\Bj)$ by the  relations
  \begin{equation}
\tikz[baseline]{\draw[postaction={decorate,decoration={markings,
    mark=at position .8 with {\arrow[scale=1.3]{>}}}},very thick]
(0,-.5) -- node [below,at start]{$j$} node
[pos=.3,circle,fill=black,inner sep=2pt,label=below left:$-\la^j$]{}
(0,.5); \node at (.5,0){$\cdots$};}=0\hspace{.8in} \tikz[baseline]{\draw[postaction={decorate,decoration={markings,
    mark=at position .8 with {\arrow[scale=1.3]{<}}}},very thick]
(0,-.5) -- node [below,at start]{$j$} node
[pos=.3,circle,fill=black,inner sep=2pt,label=below left:$\mu^j$]{}
(0,.5); \node at (.5,0){$\cdots$};}=0
\hspace{.8in} \tikz[baseline]{\draw[postaction={decorate,decoration={markings,
    mark=at position .8 with {\arrow[scale=1.3]{<}}}},very thick]
(0,0) to[out=90,in=0] 
(-.5,.5) to[out=180,in=90] node [above,at start]{$j$} (-1,0) to[out=-90,in=180] node
[pos=.3,circle,fill=black,inner sep=2pt,label=below left:$a$]{} (-.5,-.5)
to[out=0,in=-90] (0,0);    \node at (.5,0){$\cdots$};}=0 \quad a\geq \mu^j +\la^j
\label{lm-relations}
\end{equation}
\end{proposition}
\begin{proof}
  By definition, any object is a summand of some sequence $(\la,\Bi',\mu,\Bi'')$.
  Now let us prove the first statement by induction on the number of pairs in
  $\Bi'$ of entries where $i_k\in -\Gamma,i_{k'}\in \Gamma$ and
  $k<k'$.

If this number is zero, then we can move all elements in $-\Gamma$
left past the blue strand labeled $\la$ by the Scottish relation in \eqref{color-opp-cancel},
and similarly all elements in $\Gamma$ right past the strand labeled
$\mu$ by the Polish version of the same relation.  

On the other hand, if it is not 0, there is a pair where $k$ and $k'$
are consecutive.  We can apply the relation (\ref{switch-1}) or
(\ref{opp-cancel1}) to rewrite $e_{(\la,\Bi',\mu,\Bi'')}$ as factoring
through quadruples with a lower number of such pairs.  

The relations \eqref{lm-relations} follow immediately from
\eqref{color-opp-cancel} and (\ref{cost}).  The reduction to these
relations is essentially the same as the proof of
\cite[\ref{m-cyclotomic}]{Webmerged}. The kernel is spanned by diagrams with
a strand left of both the red and blue strands.  If there are multiple
such strands, at least one attached to the top or bottom, we can use
the relations much as in the proof of \cite[\ref{m-cyclotomic}]{Webmerged}
to remove these until we're left with a single strand left of this
point which thus is a
consequence of  one of the first 2 relations of \eqref{lm-relations}.  Otherwise, we can reduce to the case where a single bubble of
positive degree is left of all red and blus strands, which thus is a
consequence of the last relation of  \eqref{lm-relations}.
\end{proof}

\begin{proposition}\label{prop:hl-tensor}
   If $\fg$ is finite dimensional, we have  a strongly $\tU$-equivariant equivalence $\hl^{(\la,\mu)}\cong \hl^{(w_0\la,\mu)}$. 
\end{proposition}
\begin{proof}
We may as well assume that $-\la$ and $\mu$ are dominant, since all
other cases follow from this one by symmetry.

We've proven in Theorem \ref{highest-lowest} above that
$\hl^{w_0\la}_\la$ is equivalent to $\Vect_\K$. Let $(w_0\la,\Bi_\la)$ be the unique
  indecomposable object in the $\la$-weight space of $\hl^{w_0\la}$.

  The equivalence must send the tricolore triple $(\la, \mu)$ to
  $(w_0\la,\Bi_\la,\mu)$.
  Such a functor exists since $(w_0\la,\Bi_\la,\mu, i)$ is killed by
  the $-\la^i$th power of the dot the last $\eE_i$, and similarly,
  $(w_0\la,\Bi_\la,\mu, -i)$ is killed by the $\mu^i$th power of the
  dot on last $\eF_i$; this confirms \eqref{lm-relations}, so a
  functor exists by Proposition \ref{lm-algebra}.  Since the ungraded Euler forms of the 2
  categories coincide by Theorem \ref{th:shapovalov-1}, we need only
  prove that this functor is full.

Now, consider a morphism between  $(w_0\la,\Bi_\la,\mu, \Bj)$ and
$(w_0\la,\Bi_\la,\mu, \Bk)$, where $\Bj$ and $\Bk$ are arbitrary
KL pairs with $\EuScript{L}=\mu+\la$.  We wish to show that this is induced by a
2-morphism in $\tU$ from $\Bj\to\Bk$.

When we draw the diagram of such a morphism then a terminal in
$\Bi_\la$ at the bottom might connect to one in $\Bj$ or $\Bk$
(corresponding to either a $\eE_i$ or a $\eF_i$ respectively).   We
must show we can write a diagram of either type in terms of ones where
strands do not cross the second red strand.  A
diagram where we have a connect to $\Bk$ must have a crossing between  the strand
passing from $\Bi_\la$ to $\Bk$ and one passing
from $\Bj$ to the copy of $\Bi_\la$ at the top.  By the freedom we
have to choose the spanning set $D$, this crossing can
be assumed to occur left of the red line for $\mu$ by Theorem \ref{C-basis}. Thus
this diagram factors through a tricolore triple of the form
$(w_0\la,\Bi_\la,-\al_i,\dots)$ for some $\al_i$; but
$(w_0\la,\Bi_\la,-\al_i)\cong 0$ since the $\la-\la_i$-weight space of
$V_{w_0\la}^\Z$ is trivial.  Thus, this diagram is 0, and by
induction, we can write our diagram with no strands from $\Bi_\la$
connecting to $\Bk$.  This argument is represented schematically in
the first picture of Figure \ref{fig:fullness}.

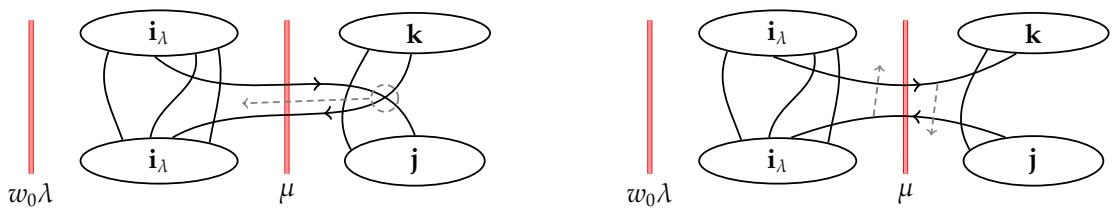
\begin{figure}
  \centering
 \scalebox{.85}{ \begin{tikzpicture}[thick]
    \node[ellipse,inner xsep=20pt,inner ysep=2pt,draw,black] (a) at
    (-2,1){$\Bi_\la$};
  \node[ellipse,inner xsep=20pt,inner ysep=2pt,draw,black] (b) at
  (-2,-1){$\Bi_\la$};
  \node[ellipse,inner xsep=20pt,inner ysep=2pt,draw,black] (c) at
    (2,1){$\Bk$};
  \node[ellipse,inner xsep=20pt,inner ysep=2pt,draw,black] (d) at
  (2,-1){$\Bj$};
\draw[wei](-4,-1.2) -- node[below,at start]{$w_0\la$} (-4,1.2);
\draw[wei](0,-1.2) -- node[below,at start]{$\mu$} (0,1.2);
\draw[postaction={decorate,decoration={markings,
    mark=at position .6 with {\arrow[scale=1.3]{<}}}}] (b.60) to[out=30,in=-100] (c.-100);
\draw[postaction={decorate,decoration={markings,
    mark=at position .6 with {\arrow[scale=1.3]{>}}}}] (a.-100) to[out=-40,in=110] (d.90);
\draw (b.20) to[out=80,in=-80] (a.-15);
\draw (b.110) to[out=80,in=-80] (a.-30);
\draw (b.150) to[out=110,in=-110] (a.-160);

\draw (d.170) to[out=120,in=-120] (c.-160);
\draw[densely dashed,gray] (1.53,-0.02) circle (6pt);
\node[left=-2pt] (e) at (1.53,-0.02){};
\draw[densely dashed,gray,->] (e) to (-.7,-.1);
  \end{tikzpicture}}\qquad \qquad 
\scalebox{.85}{\begin{tikzpicture}[thick]
    \node[ellipse,inner xsep=20pt,inner ysep=2pt,draw,black] (a) at
    (-2,1){$\Bi_\la$};
  \node[ellipse,inner xsep=20pt,inner ysep=2pt,draw,black] (b) at
  (-2,-1){$\Bi_\la$};
  \node[ellipse,inner xsep=20pt,inner ysep=2pt,draw,black] (c) at
    (2,1){$\Bk$};
  \node[ellipse,inner xsep=20pt,inner ysep=2pt,draw,black] (d) at
  (2,-1){$\Bj$};
\draw[wei](-4,-1.2) -- node[below,at start]{$w_0\la$} (-4,1.2);
\draw[wei](0,-1.2) -- node[below,at start]{$\mu$} (0,1.2);
\draw[postaction={decorate,decoration={markings,
    mark=at position .6 with {\arrow[scale=1.3]{<}}}}]  (b.60) to[out=20,in=160]  (d.140);
\draw[postaction={decorate,decoration={markings,
    mark=at position .6 with {\arrow[scale=1.3]{>}}}}] (a.-100) to[out=-20,in=-150] (c.-130);
\draw (b.20) to[out=80,in=-80] (a.-15);
\draw (b.110) to[out=80,in=-80] (a.-30);
\draw (b.150) to[out=110,in=-110] (a.-160);

\draw(d.170) to[out=120,in=-120] (c.-160);
\draw[densely dashed,gray,->]  (-.5,-.3) to (-.4,.5);
\draw[densely dashed,gray,->] (.5,.2) to (.4,-.6);
  \end{tikzpicture}}
\caption{The argument for fullness in Proposition
  \ref{prop:hl-tensor}}
\label{fig:fullness}
\end{figure}

Thus we can assume that there is a strand from $\Bi_\la$ at the
bottom connecting to $\Bj$.  In this case, there must be at
least one strand opposite it which arcs from the top copy of $\Bi_\la$
to $\Bk$.  We can push these strands together using relation
\eqref{switch-1} left of the strand labeled $\mu$.  There are two
terms: the 
correction terms have fewer strands pass from $\Bi_\la$ to $\Bj$, and
the resulting diagram with a bigon factors
through $(w_0\la,\Bi_\la,-\al_i,\dots)$, and we can thus use the
argument from above to see that this diagram is 0.  This argument is represented schematically in
the second picture of Figure \ref{fig:fullness}. This shows the
fullness of the functor and completes the proof.
\end{proof}

 As in
\cite[\ref{m-stringy-defn}]{Webmerged}, we fix an infinite list
$\mathbf{p}=\{p_1,p_2,\dots\}\in \Gamma$ of simple roots such
that each element of $\Gamma$ appears infinitely often.  For any element $v$ of a
highest weight crystal $\mathcal{B}^{\la}$, there are unique integers $\{a_1,\dots\}$ such
that $\cdots \te_{p_2}^{a_2}\te_{p_1}^{a_1}v=v_{high}$ and
$\te_k^{a_k+1}\cdots \te_{p_1}^{a_1}v=0$. The parametrization of the
elements of the crystal by this tuple is called the {\bf string
parametrization.}  We can associate this to a sequence with
multiplicities $(\dots, p_2^{(a_2)},p_1^{(a_1)})$. While this is {\it a priori}
infinite, $a_j=0$ for all but finitely many $j$, so deleting entries
with multiplicity 0, we obtain a finite sequence, which we'll call the
{\bf  string parametrization} of the corresponding crystal element.  By
convention, we'll let $|\mathbf{a}|=\sum a_i<\infty$.

Let $\boldsymbol{\epsilon}$ be the sign vector such that $\ep_k=-1$ if
$\la_k$ is dominant, and $\ep_k=1$ if it is anti-dominant.
For $\ell$-tuple of words
$(\mathbf{a}^{(1)},\dots,\mathbf{a}^{(\ell)})$, we let
$\sI(\mathbf{a}^{(1)},\dots,\mathbf{a}^{(\ell)})$ be the tricolore
triple such that the black block after the $j$th red or blue strand is
the sequence associated to the word $a^{(j)}$, with upward strands if
$\la_j$ is anti-dominant or downward strands if it is dominant.  More
formally, this is the tricolore triple with our chosen $\bla$,
\[\Bi=(\dots,\ep_1p_2^{(a_2^{(1)})},\ep_1p_1^{(a_1^{(1)})},\dots,\ep_2p_2^{(a_2^{(2)})},\ep_2p_1^{(a_1^{(2)})},\dots,\ep_\ell
p_2^{(a_2^{(\ell)})},\ep_\ell p_1^{(a_1^{(\ell)})}),\] and
$\kappa(j)=|\mathbf{a}^{(1)}|+\cdots+|\mathbf{a}^{(j-1)}|$.

\begin{definition}
  We define an ordering on compositions of length $\ell$ called {\bf
    reverse dominance order} by $\nu\geq \nu'$ if and only if
  $\sum_{k=j}^\ell \nu'_k\geq \sum_{k=j}^\ell \nu_k$ for all $j\in
  [1,\ell]$.  If $|\nu|=|\nu'|$, then this coincides with the usual
  dominance order.  

We'll order $\ell$-tuples of words by reverse dominance order on the
composition $|\mathbf{a}^\bullet|$ given by taking sums of each word,
with ties broken by lexicographic order on $\mathbf{a}^{(\ell)}$, then
lexicographic order on $\mathbf{a}^{(\ell-1)}$, etc.
\end{definition}

We define
a {\bf stringy tricolore triple} for the sequence $\mathbf{p}$ to be
$\sI(\mathbf{a}^{(1)},\dots,\mathbf{a}^{(\ell)})$ for
$\mathbf{a}^{(k)}$ the string parameterization for an element of the crystal of
highest or lowest weight $\la$.

Since we will use this fact many times, let us remind the
reader that in a graded category where the degree 0 part of the
endomorphisms of any object are finite dimensional (a condition
satisfied by $\tU,\cT$ and $\hl^\bla$), an object is indecomposable if
and only if its endomorphism algebra is graded local, i.e. has a
unique maximal homogeneous ideal.

\begin{lemma}
  Every indecomposable object of $\hl^\bla$ is isomorphic to a summand of
  $\sI(\mathbf{a}^{(1)},\dots,\mathbf{a}^{(\ell)})$ for some
  $\ell$-tuple $(\mathbf{a}^{(1)},\dots,\mathbf{a}^{(\ell)})$.
\end{lemma}
\begin{proof}
  Since every element of $\Gamma$ occurs infinitely often, every
  sequence in $\epsilon_p\Gamma$ occurs attached to some word $\mathbf{a}^{(p)}$.
Thus we need to show that every indecomposable is a summand of a
  object where in the black block right of a red strand is all
  downward strands and that after a blue strand is all upward strands.
  This is essentially the proof of Theorem \ref{red-tensor} with a
  slightly more delicate induction. We induct on the statement that
  every indecomposable is a summand of an tricolore triple where the
  $m$ rightmost black blocks have the desired form.  We can start with
  a tricolore triple where this is true of the $m-1$ rightmost black
  blocks and concentrate on $m$th from the right.  As in the proof of
  Theorem \ref{red-tensor}, the relations
  (\ref{switch-1}-\ref{switch-2}) and (\ref{color-opp-cancel}) allow
  us to push badly oriented strands further left until they are out of
  the $m$th black block.  Thus, we are done.
\end{proof}

\begin{lemma}\label{stringy-summand}\label{bla-unique}
The object  $\sI(\mathbf{a}^{(1)},\dots,\mathbf{a}^{(\ell)})$ has at
most one summand that is not a summand of the tricolore triple for a
greater word, and no such summand unless this triple is stringy.
\end{lemma}
\begin{proof}
  First, we claim that it is sufficient to show this for $\K$ a field.
  We let $\mathbbm{r}=\K/\mathfrak{m}$ be the residue field of $\K$,
  and assume that the theorem holds in this case.  We have a natural
  functor $\hl^\bla_\K\to \hl_\mathbbm{r}^\bla$ given by simply
  killing $\mathfrak{m}$.   By Hensel's lemma, this functor induces a
  bijection between indecomposable projectives and between summands of
  a given tricolore sequence.  This reduces us to the case where $\K$
  is a field.  

Now, fix an $\ell$-tuple of words $(\mathbf{a}^{(1)},\dots,\mathbf{a}^{(\ell)})$.
Let $I$ be the 2-sided ideal in $\End({\sI}(\mathbf{a}^\bullet))$ generated by elements factoring
through $\sI(\mathbf{b}^{(1)},\dots,\mathbf{b}^{(\ell)})$ such that $\mathbf{b}^\bullet>\mathbf{a}^\bullet$.  We wish to show by
induction that the quotient $\End({\sI}(\mathbf{a}^\bullet))/I$ is
graded local.  

We can assume that $\mathbf{a}^{(\ell)}\neq 0$, since if  $\mathbf{a}^{(\ell)}= 0$, we can
simply remove the rightmost red or blue strand without changing
$\End({\sI}(\mathbf{a}^\bullet))$ or $I$, and we reduce to the case
where we have $\ell-1$ red and blue strands.  Let $q$ be minimal such that
$a=a^{(\ell)}_q\neq 0$.  Thus, the rightmost part of the tricolore
triple   ${\sI}(\mathbf{a}^\bullet)$ is $a$ strands labeled $p_q$.  
Let $\mathbf{c}^\bullet$ be the sequence of words which coincides with
$\mathbf{a}^\bullet$, except that $\mathbf{c}^{(\ell)}_q=0$.  By
induction, we can assume that
$\End({\sI}(\mathbf{c}^\bullet))/I_{\mathbf{c}}$ is graded local.
Let $R_{a}$ be the ring of symmetric functions in $a$ variables, which
acts by natural endomorphisms on the functor $\eF_{i_q}^a$ as
symmetric polynomials in the dots; this is
also graded local.  Thus, we have a natural map
$\phi\colon \End({\sI}(\mathbf{c}^\bullet))/I_{\mathbf{c}}\otimes
R_{a}\to \End({\sI}(\mathbf{a}^\bullet))/I$ where the first term acts on
all but the rightmost $a$ strands, and the second on the last $a$.  If
we prove that $\phi$ is surjective, then we will know that
$ \End({\sI}(\mathbf{a}^\bullet))/I$ is necessarily graded local.

We
can divide the matchings with top and bottom
${\sI}(\mathbf{a}^\bullet)$ whose associated diagrams are not in the
image of $\phi$ fit into 2 categories:
\begin{enumerate}
\item those where one of the last $a$ strands is connected by a cap to the
  bottom.
\item those where there is no such cap, but one of these $a$ strands
  is connected to another node at the top.
\end{enumerate}
If we choose our basis vectors $D$ carefully, then we can assume that
the diagrams associated to these matchings must lie in $I$.  
In case (1), this because any diagram with a single cup connecting one
of the
last $a$ strands with the bottom of the diagram and no other crossings
has bottom smaller than the top in reverse dominance order.  For any
matching of type (1), we can choose the basis vector so that the
bottom portion of the diagram is of this form.  
In case (2), we can push every crossing between strands from the
rightmost $a$ at the top and the rightmost $a$ at the bottom to the
far right; after doing this, there will be a slice through the middle
of the diagram which corresponds to a tricolore triple for a word with $a^{(\ell)}_q>a$; this is higher
in our order.   Thus, $\phi$ is indeed surjective, and we find that we
have at most one new projective.

Now, we need only show that if ${\sI}(\mathbf{a}^\bullet)$ is not
stringy, then this quotient is trivial.  If $\mathbf{a}^{(k)}$ is not
a string parametrization, the identity of
${\sI}(\mathbf{a}^\bullet)$ can be rewritten as factoring through
triples where $\mathbf{a}^{(k)}$ is replaced by higher words in
lexicographic order, or where one of the strands is pulled left
through the $k$th red or blue strand.  Both these are higher in the
order, so this triple has no ``new'' summands.
\end{proof}

We call an object of a $\K$-linear category {\bf
  absolutely indecomposable} if it remains indecomposable after
extension by any local ring homomorphism $\K\to \K'$ between complete
local rings.  If $\K$ is a field, this just means that this object
remains indecomposable under field extensions.
\begin{corollary}\label{absolutely-indecomposable}
  Every indecomposable object of $\hl^\bla$ is absolutely indecomposable.
\end{corollary}
\begin{proof}
  Ring extension sends the object  ${\sI}(\mathbf{a}^\bullet)$ to the
  same object over the same ring.  Fix a local
  homomorphism $\K\to \K'$ and let $\mathbf{a}^\bullet$ be
  the greatest stringy sequence whose associated indecomposable $M$ splits
  after extension to $\K'$.  None of the summands of
  $M\otimes_{\K}\K'$ appears in a greater stringy sequence, so either
  ${\sI}(\mathbf{a}^\bullet)$ has two new indecomposable summands, or
  a new summand with multiplicity $>1$.  Both of these are impossible
  by Lemma \ref{bla-unique}, so we have arrived at a contradiction.
\end{proof}

This result can easily be extended to the 2-category $\tU$.  We call a
1-morphism $(\Bi,-\Bj)$ in $\dU$ {\bf stringy} if $\Bi,\Bj$ are both
positive string parametrizations for elements of the crystal
$B(-\infty)$.  We can endow these with a similar variant of
lexicographic order: if we fix the weight of the 1-morphism, then
$(\Bi,-\Bj)>(\Bi',-\Bj')$
\begin{itemize}
\item if $|\Bi|<|\Bi'|$, or 
\item if $|\Bi|=|\Bi'|$ and $\Bj>\Bj'$ in
  lexicographic order on the corresponding words or
\item if $\Bj=\Bj'$ and $\Bi>\Bi'$ in lexicographic order on the
  corresponding words.
\end{itemize}
The proof can be extended to show that:
\begin{proposition}\label{unique}
  The 1-morphism $(\Bi,-\Bj)$ in $\dU$ has at most one summand that is
  not a summand of any higher
  stringy sequence, and unless  $(\Bi,-\Bj)$ is itself stringy.  This summand is absolutely indecomposable and
  every indecomposable appears this way for a unique stringy sequence.
\end{proposition}

As in \cite[\S \ref{m-sec:decat}]{Webmerged}, we can
  define vectors $v_{\Bi}^\kappa=v_{\sI}$ in $V_\bla^\Z$ inductively by 
\begin{itemize}
\item if $\kappa(\ell)=n$, then
  $v_{\Bi}^\kappa=v_{\Bi}^{\kappa^-}\otimes v_\ell$ where $v_\ell$ is
  the highest (lowest) weight vector of $V_{\la_\ell}$ if $\la_\ell$
  is dominant (anti-dominant), and $\kappa^-$ is the
  restriction to $[1,\ell-1]$.
\item If $\kappa(\ell)\neq n$, so $v_{\Bi}^\kappa=E_{i_n}v^{\kappa}_{\Bi^-}$, where
  $\Bi^-=(i_1,\dots,i_{n-1})$, using the convention that $F_{i}=E_{-i}$. 
\end{itemize}

\begin{theorem}\label{th:shapovalov-1}
  The 
ungraded Grothendieck group 
$K^0(\hl^\bla)$ is isomorphic to $\bar V_\bla^\Z$ via the map sending
$[P^\kappa_\Bi]\to v_{\Bi}^\kappa$.  This map intertwines the Euler form with the factorwise Shapovalov
  form $\langle-,-\rangle_s$ on the tensor product.
\end{theorem}
\begin{remark}
Note that the comparable theorem for $\bla$ all highest weight
(\cite[\ref{m-Uq-action}]{Webmerged}) did not require setting $q=1$.
This is because we already knew the form on the quantum tensor product
we expected to match the Euler form with, whereas here we do not know
{\it a priori} that a suitable bilinear form exists on a tensor
product.  We'll establish later that this result holds with $q$
not specialized.
\end{remark}

\begin{proof}
Once we have proved the equality of dimensions \begin{equation}\label{Euler-match}
\sum \dim \HOM(\sI,\sI')=\langle
v_{\sI},v_{\sI'}\rangle_s,
\end{equation}
the proof is precisely the same as that of
\cite[\ref{m-Uq-action}]{Webmerged} with its associated lemmata, which we
leave as an exercise to the reader.

The proof of \eqref{Euler-match} is also quite similar, but requires
small changes.  We need only prove that \eqref{Euler-match} holds when the
tricolore triples are of the form $\sI(\mathbf{a}^{\bullet})$.  As in
\cite[\ref{m-Uq-action}]{Webmerged}, the induction is easier to swing
if we allow one extra strand which points in the ``wrong'' direction.   

We induct on the reverse dominance order on words.  If $\sI$ and
$\sI'$ both have $\mathbf{a}^{(\ell)}=0$, then \eqref{Euler-match}
will hold after we remove the
rightmost strand, which is equivalent the desired case of
\eqref{Euler-match}.  If one of $\sI$ and
$\sI'$ have $\mathbf{a}^{(\ell)}\neq 0$, then we can apply the
adjunction to pull a strand from one side to the other, and then slide
it left using the relations (\ref{switch-1}) and
(\ref{color-opp-cancel}).  
\excise{We claim that 
\begin{equation}\label{Euler-match}
\dim \HOM(P^{\kappa}_{\Bi},P^{\kappa'}_{\Bi'})=\langle
v_{\Bi'}^{\kappa'},v_{\Bi'}^{\kappa'}\rangle_s.
\end{equation}
We prove \eqref{Euler-match} by induction on $n$ and $\ell$.  Unless
$n=\kappa(\ell)=\kappa'(\ell)$, we can move a $\fF_i$ from one side to
become a $\fE_i$ on the other (up to shift).  This has the same effect
on the Shapovalov form, since $E_i$ and $F_i$ are biadjoint under
$\langle-,-\rangle_s$. The decompositions of $\fE_iP^{\kappa}_{\Bi}$ into $P^{\kappa''}_{\Bi''}$'s matches that of the vector since both are done using the commutation relations between $\fE_i$ and $\fF_i$ or $E_i$ and $F_i$, which we already know match.

If  $n=\kappa(\ell)=\kappa'(\ell)$, then the dimension of the
$\HOM$-space and the inner product are both unchanged by simply
removing the red line.  This shows the equality \eqref{Euler-match}.

Thus, if we are given any linear relation satisfied by
$[P^\kappa_\Bi]$'s, the corresponding linear combination of
$v_{\Bi}^\kappa$'s is in the kernel of this form, and thus 0 in
$V_\bla$.  Thus, $[P^{\kappa}_{\Bi}]\mapsto v_\Bi^\kappa$ defines a
surjective map.

Thus, we need only show that this map is injective. By Lemma
\ref{bla-unique}, the stringy sequences span $K^0(\hl^\bla)$; on
the other hand, the are clearly sent to a basis of $V^\Z_\bla$ since
they are upper triangular with any crystal basis.  Of course, any
linear map sending a spanning set to a basis is an isomorphism.}
\end{proof}
This in particular shows that the classes of stringy sequences are
linearly independent, so the ``new'' summand of each stringy sequence
must be non-zero.
\begin{corollary}\label{cor:exist}
    Each stringy sequence in $\hl^\bla$ or $\tU$ has exactly one indecomposable
    non-zero summand
  which is not isomorphic to any summand of a larger sequence.
\end{corollary}

The autofunctor $\tilde\psi$ obviously preserves violating morphisms,
and thus descends to an involution on $\hl^\bla$ which we 
denote $\tilde \psi^\bla$.  

This functor defines an involution on $V^\bla$ for  any $\bla$.  We will
denote this involution $\psi^\bla$. 

\begin{proposition}
  For each indecomposable projective $P$ in $\hl^\bla$ (resp. $\dU$), there is a unique grading
  shift $P(n)$ such that $\tilde\psi^\bla(P(n))\cong P(n)$ (resp.
  $\tilde\psi(P(n))\cong P(n)$)
\end{proposition}
\begin{proof}
  Such a shift is obviously unique, so we need only prove it exists.
  There is a unique $n$ such that $P(n)$ is a summand 
  of the corresponding stringy sequence.  Since the latter module is
  self-dual,  $\tilde\psi^\bla(P(n))$ is a summand of it, and by the
  uniqueness of Proposition \ref{unique}, we must have $\tilde\psi^\bla(P(n))\cong P(n)$.
\end{proof}

These results show that:
\begin{theorem}\label{KS-duality}
  The categories $\hl^\bla$ and $\tU$ are humorous categories
  with the obvious grading shift, and dualities given by $\tilde
  \psi^\bla$ and $\tilde \psi$. \end{theorem}

\section{Representation categories and standard modules}
\label{sec:repr-categ}

As in \cite{Webmerged}, it will be useful to deal with an abelian
category, not just an additive one.  In particular, (as far as the
author is aware) this is necessary to check that the Grothendieck
group of $\hl^\bla$ is the tensor product representation as a
representation of $U_q(\fq)$; we have thus far only checked that this
holds at $q=1$.

\begin{definition}
  We let $\cata^\bla:=\operatorname{Rep}(\hl^\bla)$.  Let $\Yon\colon\hl^\bla\to \cata^\bla$
  be the Yoneda embedding $\sI\mapsto \Hom(\sI,-)$.
\end{definition}
Note that we do {\it not} require an object in $\cata^\bla$ to be
finitely generated.
\begin{definition}
  Let $H^\bla$ be the subring of the opposite endomorphism ring
  $\End_{\hl^\bla}\left(\bigoplus_{\sI}\sI\right)^{op}$ which kills
  all but finitely many summands.
\end{definition}
Note that this is a non-unital ring; by a $H^\bla$-module $M$, we
always mean one which is the direct sum of the images of the
idempotents $e_{\Bi,\bla,\kappa}$.  Since $H^\bla$ is locally unital
for the system of idempotents $e_{\Bi,\bla,\kappa}$, this is the
natural generalization of the condition that the identity of a ring
must act by the identity on a module. 

We can interpret an object in
$\cata^\bla$ as a module over $H^\bla$ using the obvious functor
$\hl^\bla\to H^\bla\pmodu$ given by the morphism space $X\mapsto
\HOM_{\hl}(\bigoplus_{\sI}\sI,-)$.
Of course, $\cata^\bla$ is an abelian category, and since $\Yon(P)$ is
projective for any $P\in \Ob(\hl^\bla)$, $\cata^\bla$ has
enough projectives.  However, it is not clear that if $M$ is finitely
generated, and $P$ is a finitely generated projective with a
surjection $P\to M$, then the kernel of this map is finitely
generated.  

We can define an action of
$\tU$ on $\cata^\bla$ by exact functors using the biadjunction between $\eE_i$ and
$\eF_i$ as a definition, i.e.
\begin{align*}
  \eF_i\cdot M(\sI)&:= M\big(\eE_i\sI(\langle-\al_i,\mu\rangle+1)\big)\\
  \eE_i\cdot M(\sI)&:= M\big(\eF_i\sI(\langle-\al_i,\mu\rangle+1)\big).
\end{align*}
Note the switch of $\eE_i$ and $\eF_i$ above; this is what is required
so the the Yoneda embedding intertwines the categorification functors,
since $\eE_i$ and $\eF_i$ are biadjoint on $\hl^\bla$.

\begin{theorem}
  The Grothendieck group $K^0(\cata^\bla)$ is isomorphic to the
  lattice dual to $\bar V_\bla^\Z$, with the map induced by $\Yon$ given by the
  Shapovalov form. 
\end{theorem}
Since we have twisted our action by the Cartan involution, the
Shapovalov pairing defines a map of representations.
If $\fg$ is infinite dimensional, then we take the full dual; that is, as
abstract abelian groups, $\bar V_\bla^\Z$ is a direct {\it sum} of
copies of $\Z$, while $K^0(\cata^\bla)$ is a direct {\it product}.  We
let $\widehat V_\bla=K^0(\cata^\bla)\otimes_\Z\C$; this is a $\fg$
representation defined by taking direct product of the weight spaces
of a highest weight representation, rather than direct sum.

Even in finite type, the Shapovalov form is not always unimodular over
the integers, so
this will not usually coincide with $K^0(\hl^\bla)$; this will only happen if all entries in $\bla$ are
minuscule and $\fg$ is finite dimensional.  
In
particular, this shows that $\cata^\bla$ extremely rarely has finite
global dimension, since that is only possible when these lattices coincide.

For a tricolore triple
$\mathsf{I}=(\Bi,\bla,\kappa)$, the number of black strands in each
black block define a composition,
which we denote by $\nu_\sI$.   We also have a function
$\bal_\sI\colon [1,\ell]\to \rola$
given by the sum of the roots labeling each black block.  
For each such function $\bal$, we have a map $\wp\colon
R_{\al(1)}\otimes\cdots\otimes R_{\al(\ell)}\to H^\bla$ sending an $\ell$-tuple
of KL diagrams to their horizontal composition with red and blue lines
added, as in \cite[\eqref{m-wp-map}]{Webmerged}.
 \excise{Attached to every indecomposable object $P$ in $\hl^\bla$, there is a unique
stringy sequence, the smallest in lexicographic order in which it
appears as a summand; we let $\nu_{P}$ be the composition attached to
this stringy sequence.  If $\sI$ is a stringy sequence, then every
indecomposable summand $P$ of $\sI$ has $\nu_\sI$}

\begin{definition}
  The standard representation $S_\sI$ is the maximal quotient of
  $\Yon(\sI)$ such that $S_{\sI}(\sI')=0$ if $\nu_{\sI'}> \nu_\sI$ in the
  reverse dominance order on compositions.  

  More generally, we let the {\bf standardization} $S_M$ of an object
  $M$ in $\hl^{\la_1}_{\la_1-\al(1)}\times \cdots \times
  \hl^{\la_\ell}_{\la_\ell-\al(\ell)}$ for some fixed $\bal$ be the
  initial object in $\mathfrak{V}^\bla$ amongst those such that
  $S_M(\oplus_{\bal_{\sI}=\bal}\sI)\cong M$ as
  $R_{\al(1)}\otimes\cdots\otimes R_{\al(\ell)} $-modules (via $\wp$)
  and $S_M(\sI')=0$ if $\nu_{\sI'}> \nu_\sI$.
\end{definition}
We think of the relation induced from reverse dominance order on
compositions as a preorder of the set of sequences $\sI$. 
In terms of $H^\bla$-modules, this module has a presentation much like
that of standard modules of \cite{Webmerged}: one can define $S_{\sI}$
as a quotient of $\Yon(\sI)$ by the image of every map from
$\Yon(\sI')$ with $\sI'>\sI$.   In terms of Stendhal diagrams, this
means we quotient $H^\bla e_{\sI}$ by all diagrams which are
{\bf standardly violated}\footnote{This is the reflection of the definition of
  standardly violated from \cite{Webmerged}, since we are looking at
  left standard modules rather than right modules. }, that is, diagrams whose bottom is $\sI$,
and where some horizontal slice $y=a$ is $\sI'$ with
$\sI'>\sI$.

As in \cite{Webmerged}, we let ${}_k\eE_i$ and ${}_k\eF_i$ denote the
categorification functors on the $k$th factor of $\hl^{\la_1}\times \cdots \times
  \hl^{\la_\ell}$.
\begin{lemma}\label{prop:act-filter}
  The module $\eE_iS_M$ has a natural filtration $Q_1\supset
  Q_2\supset \cdots$ such that \[Q_k/Q_{k+1}\cong
  S_{{}_k\eE_iM}.\]

 The module $\eF_iS_M$ has a natural filtration $O_m\supset
  O_{m-1}\supset \cdots$ such that \[O_k/O_{k-1}\cong
  S_{{}_k\eF_iM}.\]
\end{lemma}
\begin{proof}
  Since this is quite close to the proof of
  \cite[\ref{m-prop:act-filter}]{Webmerged}, we will only give a short
  sketch covering the points to be changed (which themselves are
  almost the same as the changes made for the proof of Theorem
  \ref{th:shapovalov-1}).  The construction of the filtrations and the
  surjective maps from standardizations is
  exactly as shown in  \cite[Figures \ref{m-horiz-map} \&
  \ref{m-sta-filt}]{Webmerged}; thus we need only show that the
  successive quotients have the correct dimensions.   That is, we need
  to prove for any $\sI$ that 
  \begin{align}
    \dim \eE_iS_M(\sI)&=\sum \dim S_{{}_k\eE_iM}(\sI)\label{S-E}\\
    \dim \eF_iS_M(\sI)&=\sum \dim S_{{}_k\eF_iM}(\sI)\label{S-F}
  \end{align}
This must be
  shown via an induction over $\sI$ in reverse dominance order (rather
  than weight).  Note that both $\sI$ and $S_M$ have associated
  compositions, which don't coincide.  Our induction step is that we
  prove the equations \eqref{S-E} and \eqref{S-F} where either the
  composition for $\sI$ or $S_M$ coincides with $\nu$, assuming that
  they hold when either of the compositions is $>\nu$ in reverse
  dominance order.

As usual, we can assume that if $\la_\ell$ is
  (anti-)dominant, all strands after the last red or blue strand are
  downwards (upwards).  If the left-most strand in $\sI$ is red or
  blue, the equations \eqref{S-E} and \eqref{S-F} follow from the case
  where this strand is removed, so we can assume that the leftmost
  strand is black.

  If $\la_\ell$ is dominant, then we can apply the argument given in
  the proof of 
  \cite[\ref{m-prop:act-filter}]{Webmerged}: we note that 
  \begin{equation}
 \dim \eF_iS_M(\sI)=\dim \eE_j\eF_iS_M(\sI')\leq \sum\dim
S_{{}_p\eE_j{}_k\eF_iM}(\sI')\label{eq:3}
\end{equation}
which implies that
\begin{equation}
\dim \eF_i\eE_jS_M(\sI')\leq \sum\dim
S_{{}_k\eF_i{}_p\eE_jM}(\sI')\label{eq:2}
\end{equation}
with equality holding in either both
\eqref{eq:3} and \eqref{eq:2} or neither.  It holds in
\eqref{eq:2} by induction, since $\eE_i$ applied to a tricolore
triple with $\la_\ell$ dominant can always be rewritten as a summand
of triples higher in reverse dominance order. Thus, we also have
equality in \eqref{eq:3}, which is only
possible if \eqref{S-E} and \eqref{S-F} hold as well (since we already
know they are inequalities).  

If $\la_\ell$ is anti-dominant map, we just apply the same argument
with $\eF_i$ and $\eE_i$ reversed.  This is possible since now $\eF_i$
applied to this triple will be a summand of triples higher in reverse
dominance order.
\end{proof}

If we write $\sI$ as the concatenation of tricolore triples $\sI_i$ consisting
of one red or blue strand and then the black block, then we
let \[s_\sI=v_{\sI_1}\otimes \cdots \otimes v_{\sI_\ell}.\]
\excise{\begin{proposition}
  There is an isomorphism of $U_q(\fg)$-representations
  $K^0(\hl^\bla)\otimes_\Z\C\cong \bar V_\bla$,
  such that $[S_\sI]=s_\sI$.  
\end{proposition}
\begin{proof}
The proof is identical with \cite[\ref{m-standard-class}]{Webmerged},
so we leave it to the reader.
\end{proof}}
However, if we calculate the grading shifts of these filtrations, a
new wrinkle appears.  If the new line we add is blue, the grading
shifts will behave as though we have tensored the represenations in
the opposite order.  That is:
\begin{proposition}
  As representations of $U_q(\fg)$, we have an isomorphism 
  \[K^0_q(\hl^{(\la_1,\dots,\la_\ell)})\cong
  \begin{cases}
  K^0_q(\hl^{(\la_1,\dots,\la_{\ell-1})})\otimes
  K^0_q(\hl^{(\la_\ell)}) & \la_\ell\text{ dominant}\\
  K^0_q(\hl^{(\la_1,\dots,\la_{\ell-1})})\otimes^{\op}
  K^0_q(\hl^{(\la_\ell)}) & -\la_\ell\text{ dominant}\\
\end{cases}\]
induced by the functor \[\mathbb{S}\colon
\hl^{(\la_1,\dots,\la_{\ell-1})}\times \hl^{(\la_\ell)}\to \hl^\bla.\]
\end{proposition}
\begin{proof}
 We obtain a basis of $K^0(\hl^\bla)$ by taking the stringy
  sequences.  Each of these has a standard quotient, and the matrix
  giving the multiplicities of standards in the stringy basis is
  upper-triangular, with 1's on the diagonal.  Thus, it has an inverse
  with the same property, which we can use to define classes $[S_\sI]$
  in $K^0(\hl^\bla)$ which form a basis.  Since $\mathbb{S}$ sends
  pairs of these standard quotients to standard quotients, it sends a
  basis to a basis, and thus defines an isomorphism.  Thus, we need
  only check how $E_i$ and $F_i$ act.

The proof for $\la_\ell$ dominant is essentially the same as in
\cite{Webmerged}: one can consider the action on a standardization
$\mathbb{S} (M_1, M_2)$
from $\hl^{(\la_1,\dots,\la_{\ell-1})}\times \hl^{(\la_\ell)}$.  The
module $\eF_i \mathbb{S} (M_1, M_2)$ has a submodule $N$ generated by the element 
     \begin{equation}\label{crossing-element}
      \begin{tikzpicture}[very thick,yscale=1.5,baseline]
        \draw[wei] (1,-.5) -- +(0,1.5) node[at
        start,below]{$\la_1$} node[at end,above]{$\la_1$};

   \node at (2,.25){$\cdots$};
   \draw[wei] (3,-.5) -- (3,1) node[at start,below]{$\la_{\ell}$} 
  node[at   end,above]{$\la_{\ell}$}; 
   \draw (3.5,-.5) -- (3.5,1); 
   \node at (4,.25){$\cdots$};
   \draw (4.5,-.5) -- (4.5,1); 
\draw[rdir] (5,-.5) to[out=135,in=-90] node [below, at start]{$i$} (2.5,1);

      \end{tikzpicture}
    \end{equation}
    The submodule $N$ is isomorphic to the standardization $\mathbb{S}(\eF_i
    M_1, M_2)$ but with a shift by the degree of the element
    \eqref{crossing-element}, which is
    $\mu_\ell^i=\al_i^\vee(\mu_\ell)$ where $\mu_\ell$ is the weight
    of $M_2$.  The quotient $\eF_i \mathbb{S}(M_1, M_2)/N$ is isomorphic to
    $\mathbb{S}(M_1, \eF_i M_2)$.  That is, we have an
    equality
    \begin{multline*} [\eF_i \mathbb{S}(M_1,
      M_2)]=q^{\mu_\ell^i}[\mathbb{S}(\eF_i M_1, M_2)]+[\mathbb{S}(M_1, \eF_i
      M_2)]\\=q^{\mu_\ell^i}(F_i[M_1])\otimes [M_2]+[M_1]\otimes (F_i
      [M_2])=\Delta(F_i)( [M_1]\otimes [M_2]).
    \end{multline*}
On the other hand, the module $\eE_i \mathbb{S}(M_1, M_2)$ has a similar
submodule generated by      \begin{equation}\label{crossing-element2}
      \begin{tikzpicture}[very thick,yscale=1.5,baseline]
        \draw[wei] (1,-.5) -- +(0,1.5) node[at
        start,below]{$\la_1$} node[at end,above]{$\la_1$};

   \node at (2,.25){$\cdots$};
   \draw[wei] (3,-.5) -- (3,1) node[at start,below]{$\la_{\ell}$} 
  node[at   end,above]{$\la_{\ell}$}; 
   \draw (3.5,-.5) -- (3.5,1); 
   \node at (4,.25){$\cdots$};
   \draw (4.5,-.5) -- (4.5,1); 
\draw[dir] (5,-.5) to[out=135,in=-90] node [below, at start]{$i$} (2.5,1);

      \end{tikzpicture}
    \end{equation}
Since the element \eqref{crossing-element} has degree 0, the submodule
$N$ it generates is isomorphic to $ \mathbb{S}(\eE_i M_1, M_2)$ with no grading
shift, whereas the quotient is isomorphic to $\eE_i \mathbb{S}(M_1, M_2)/N\cong \mathbb{S}(M_1, \eE_i
M_2) (\mu^i-\mu^i_\ell)$, since every diagram in the quotient has a
cap to the right of the leftmost red strand, and the degree of this
cap decreases by $\mu^i-\mu^i_\ell$ when the label on the surrounding
region changes from $\mu$ to $\mu_\ell$.  Thus, we have that
 \begin{multline*} [\eE_i \mathbb{S}(M_1,
      M_2)]=[\mathbb{S}(\eE_i M_1, M_2)]+q^{\mu_\ell^i-\mu^i}[\mathbb{S}(M_1, \eE_i
      M_2)]\\=(E_i[M_1])\otimes [M_2]+q^{\mu_\ell^i-\mu^i}[M_1]\otimes (E_i
      [M_2])=\Delta(E_i)( [M_1]\otimes [M_2]).
    \end{multline*}
This shows the result for $\la$ dominant.  The result when $\la$ is
anti-dominant follows from the same argument with $\eE_i$ and $\eF_i$
switching places.  In fact, there's an
equivalence of categories $\hl^{\bla}\cong \hl^{-\bla}$ which reverses
the orientation on every black strand, and switches red and black
strands (one must multiply all black/black crossings by $-1$).  
This functor interchanges the action of $\eE_i$ and $\eF_i$, and thus is compatible
with the
Cartan involution $\omega$.  Since $(\omega\otimes \omega)\circ
\Delta \circ \omega=\Delta^{\op}$, this shows that the action when
$\la$ is anti-dominant is given by the opposite coproduct.
\end{proof}
\excise{
These filtrations categorify the equations
 \begin{multline*}
\Delta^{(\ell)}(E_i)=E_i\otimes 1\otimes \cdots \otimes 1+\tilde K_i\otimes E_i\otimes 1\otimes \cdots \otimes 1+ \cdots + \\ \tilde K_i\otimes\cdots \otimes \tilde K_i\otimes E_i \otimes 1+\tilde K_i\otimes \cdots\otimes \tilde K_i\otimes E_i.\end{multline*}
\begin{multline*}  
\Delta^{(\ell)}(F_i)=F_i\otimes \tilde K_{-i}  \otimes \cdots \otimes \tilde K_{-i}+1\otimes F_i\otimes \tilde K_{-i}  \otimes \cdots \otimes \tilde K_{-i}+\cdots +\\ 1\otimes \cdots \otimes 1\otimes F_i\otimes \tilde K_{-i}+ 1\otimes \cdots \otimes 1\otimes F_i.
\end{multline*}
since the grading shifts given in Proposition \ref{prop:act-filter}
are precisely the same as the power of $q$ introduced by the addition
of $K_{\pm i}$'s in the coproduct.}
Thus, applying this formula inductively, we obtain that:
\begin{proposition}\label{q-iso}
 There is a unique isomorphism of $U_q(\fg)$-modules
 $K^0(\hl^\bla)\cong V_\bla^\Z$ sending \([\sI]\mapsto v_{\sI} \)
 and $[S_\sI]\mapsto s_\sI$.
\end{proposition}

This result also shows that the category of objects filtered by
standards is invariant under the action of $\eF_i,\eE_i$ and the
addition of red and blue lines. In
particular, this shows that the objects $\Yon(\sI)$ all have standard
filtrations.

We refer to \cite[1.2.4]{DPS} for the definition of a {\bf strict
  stratifying system}.  While in that paper, they only consider the
case of a finite quasi-poset, their definition makes sense even for an
countably infinite quasi-poset such as compositions endowed with reverse
dominance order.  Similarly, for us a {\bf standard stratification}
will be allowed to be indexed by a countably infinite quasi-poset such that the
interval between any two elements is finite; in terms of 
\cite{CPS96}, this means allowing a stratification of an algebra to be
an infinite chain of ideals $\cdots \subset J_i \subset J_{i+1}\subset \cdots$ with
$i\in \Z$ such that $\bigcup J_i=A$ and $\bigcap J_i=\{0\}$, rather
than a finite chain of this type.

\begin{corollary}\label{standard-strat}
   The objects $S_\sI$ with the induced preorder define a strict stratifying system of
   $\cata_\bla$; thus, they define a standard stratification of the
   algebra $H^\bla$.  In particular, every indecomposable projective
   in $\hl^\bla$ has a filtration by standardizations of projective
   indecomposables in $\hl^{\la_1}_{\la_1-\al(1)}\times \cdots \times
  \hl^{\la_\ell}_{\la_\ell-\al(\ell)}$.
\end{corollary}
We note that outside finite type, the standards will
typically be infinite dimensional (assuming both red and blue strands
are used). However  there are only
finitely many compositions of any size larger than a fixed one in
reverse dominance order, and only finitely many sequences with a given
composition.  Thus, only finitely many standards will occur in the stratification of $\Yon(\sI)$.

\section{Orthodox bases}
\label{sec:orthodox-bases}

In this paper, we have developed the theory of categorifications of
$\dot{U}$ and its representations with a particular application in
mind: constructing bases of these representations. We remind the reader
that 
we have fixed a complete local ring $\K$ and polynomials $Q_{ij}$. We should note that the
bases we consider depend in an essential way on the ring and
polynomials chosen.  

\begin{definition}
Let $C$ denote the set of indecomposable $\tilde \psi$-invariant 1-morphisms (up to shift) in
$\dU$ and $C_\bla$ be the set of $\tilde \psi^\bla$-invariant objects
of $\hl^\bla$. 

Let the {\bf orthodox basis} $\{o_P=[P]\}_{P\in C}$ of $\dot{U}$ be defined by classes of
  $\tilde \psi$-invariant indecomposables under
  the isomorphism $K^0(\dU)\cong \dot{U}$.  Similarly, the orthodox
  basis $\{o_P=[P]\}_{P\in C_\bla}$ of $V_\bla^\Z$ is that defined by
  $\tilde{\psi}^\bla$-invariant indecomposable classes of $\hl^\bla$. 
\end{definition}
The orthodox bases of $\dot{U}$ and its representations carry over a
surprising amount of structure which occur for canonical bases.  

We've already shown in Theorem \ref{KS-duality} that $\dU$ and
$\hl^\bla$ are humorous categories.  Thus, by
Lemma \ref{category-pre-canonical}, we have a pre-canonical structure on the
Grothendieck groups of these categories:
\begin{definition} The {\bf orthodox pre-canonical structure}  on
  the vector spaces $\dot{U}$ and $V_\bla^\Z$ is defined as in Lemma \ref{category-pre-canonical} by:
  \begin{itemize}
  \item The bar-involution is given by $\psi$ ($\psi^\bla$).
  \item The inner product given by the Euler form $\langle
    [M],[N]\rangle=\sum_{i}q^{-i}\dim\Hom(M,N).$
  \item The standard basis $a_c$ is given the classes of tricolore triples attached to stringy sequences.
  \end{itemize}
\end{definition}

In several cases, the bar involutions and inner products of these
pre-canonical strctures include previously
defined structures.
Lusztig has defined bar-involutions and inner products on:
\begin{enumerate}
\item the modified quantized enveloping algebra $\dot{U}$. The bar involution
  $\bar{\phantom{x}}$ is defined in \cite[23.1.8]{Lusbook}.  In \cite[26.1.2]{Lusbook}, he defines a bilinear form
  $(-,-)$; we'll wish to consider the induced sesquilinear form
  $\langle u,v\rangle:=(\bar u,v)$.
\item the tensor product $V_{-\la,\mu}^\Z$ for $\la,\mu$ both
  dominant.  The bar involution $\Psi$ is defined in
  \cite[24.3.2]{Lusbook}, and the bilinear form $(-,-)_{\la,\mu}$ in
  \cite[26.2.1]{Lusbook}; as above, we take the induced sesquilinear
  form $\langle u,v\rangle=(\Psi(u),v)_{\la,\mu}$.
\item the tensor product $V_{\bla}^\Z$ if all $\la_i$ are dominant and $\fg$ is finite dimensional.
 We can use the bar-involution $\Psi$ defined in
  \cite[27.3.1]{Lusbook} for the tensor product of these as based
  modules.  We define a bilinear form on these varieties as the
  tensor product of the bilinear forms  $(-,-)_{0,\la_i}$, and then
  turn this into a sesquilinear form using $\Psi$ as above. 
\end{enumerate}

\begin{proposition}
  The orthodox bar-involution and inner product agree with Lusztig's in
  the cases (1-3).
\end{proposition}

\begin{proof}\hfill
  \begin{enumerate}
  \item The involution $\bar{\phantom{x}}$ on $\dot{U}$ is distinguished by
    fixing monomials in $E_i$ and $F_i$.  The same is true of that
    induced by $\tilde{\psi}$.  The agreement of forms follows from Theorem \ref{categorify-U}.
\item 
  The involution $\Psi$ on $V_{-\la}\otimes V_{\mu}$ is the unique
  involution which satisfies \[\Psi(u\cdot (v_{-\la}\otimes
  v_{\mu}))=\bar{u}\cdot  (v_{-\la}\otimes
  v_{\mu}).\] Thus we need only show that $\tilde\psi^{-\la,\mu}$
  satisfies this property. Flipping over a diagram commutes with acting on
  $(-\la,\mu)$ with it, so $\tilde\psi^{-\la,\mu}$ and $\tilde \psi$
  (on $\tU$) are compatible.

Similarly, the agreement of forms follows from the fact that for both
forms $\langle v_{-\la}\otimes
  v_{\mu},v_{-\la}\otimes
  v_{\mu}\rangle=1$ and  $\langle uv,w\rangle=\langle
  v,\tau(u)w\rangle$ where $\tau$ is the $q$-antilinear antiautomorphism defined by $$\tau(E_i)=q_i^{-1}\tilde{K}_{-i}F_i \qquad \tau(F_i)=q_i^{-1}\tilde{K}_{i}E_i \qquad \tau(K_\mu)=K_{-\mu}.$$
\item Lusztig's bar involution $\Psi$ on a tensor product is compatible with the
  action of $U_q(\fg)$ on the tensor product, and if $v\in M'$ is a
  highest weight vector, then it also commutes with the inclusion
  $M\to M\otimes M'$ sending $m\mapsto m\otimes v$.  Furthermore, if
  $M'$ is irreducible, $\Psi$ is uniquely characterized by these
  properties, since $M\otimes \{v\}$ generates $M\otimes M'$ over
  $U_q(\fg)$.   Note that in particular, the vectors $v_{\sI}$ defined
  above are all $\Psi$-invariant, and span the space $V_{\bla}$; thus
  any other anti-linear map that fixes these vectors must be $\Psi$.

  Thus, we need only check that $\tilde \psi^{\bla}$ categorifies a bar
  involution that fixes the same
  vectors $v_{\sI}$.  The tricolore triples $\sI$ are obviously
  $\tilde \psi^{\bla}$-invariant (essentially by definition) and so
  the result follows from Theorem \ref{th:shapovalov-1}.  The match of
  Euler form and Lusztig's form in this case is precisely \cite[2.30]{Webmerged}.\qedhere
  \end{enumerate}
\end{proof}

Our standard bases, however, are not the same as those typically used
by Lusztig; luckily, as we noted before, the dependence of canonical
bases on standard bases is very weak, so we can still show that
Lusztig's bases are canonical bases in our sense.

\begin{theorem}\label{thm:lusztig}
  The bases defined by Lusztig in cases (1-3) discussed
  above are canonical bases for the orthodox pre-canonical structure,
  in the sense of Definition \ref{def:canonical}.
\end{theorem}
It's worth noting: unlike Theorem \ref{main}, Theorem
\ref{thm:lusztig} does not require a symmetric Cartan matrix. It holds
for any symmetrizable Kac-Moody algebra, in particular for finite
dimensional Lie algebras of type BCDFG.

\begin{proof}
  In cases (1-3), the canonical basis of Lusztig satisfies the almost
  orthogonality conditions.  This is proven for cases (1-2) in
  \cite[26.3.1]{Lusbook}, and the same argument extends easily to (3).
  Thus, as argued in Lemma \ref{canonical-sign}, any canonical basis
  vector for the orthodox pre-canonical structure must either lie in
  Lusztig's basis, or its negative must.

As in \cite[14.4.2]{Lusbook}, the property that distinguishes
Lusztig's basis from its more easily found signed version is
compatibility with the action of $E_i$ and $F_i$.
First, as a base case, this basis contains (in the respective cases)
\begin{enumerate}
\item the vector $1_\nu$
\item the vector $v_{-\la}\otimes v_{\mu}$
\item all vectors of the form $B_{d}\otimes v_{\la_\ell}$ for the
  canonical basis of $V_{\la_1}^\Z\otimes \cdots \otimes V_{\la_{\ell-1}}^\Z$.
\end{enumerate}
There is also an inductive piece of the definition: the indexing set
of canonical basis for cases (1-3) can be identified with the
corresponding set of stringy sequences.  If this sequence does not end
with $\pm$ a simple root (that is, with a black strand at its far
right), then it belongs to one of the base cases above.  Otherwise,
its last term is of the form $ i^{(n)}$ for some $n\geq 0$, and $i\in
\pm \Gamma$.  If $c$ is an element of this indexing set, let $c'$ be
the object indexed by the stringy sequence with this last appearance
of $i^{(n)}$ deleted. 

Using the convention that $F_i=E_{-i}$, the positivity condition for
the corresponding Lusztig basis vector $B_c$ is that:
\[ B_c\in E_i^{(n)}B_{c'} +\sum_{d< c} \Z[q,q^{-1}]
B_{d}.\]  This is explicit in the
case for a basis vector in $U^+$ by \cite[14.3.2(c) \&
14.4.2]{Lusbook}, and easily carries over to the more general cases (1-3)
which use the basis of $U^+$ in their definition.

The base case is easily established for the canonical bases of the
orthodox pre-canonical structure.  The only case that is not
tautological is (3) where we must work by induction, and assume that
we have proven the theorem for tensor products with $\ell-1$ factors.
Once this is assumed, we need only note that $-\otimes v_\la$
commutes with $\Psi$, sends stringy sequences to the stringy sequences
and preserves the inner product, and thus sends canonical basis
vectors to canonical basis vectors.

Now, consider a minimal Lusztig canonical basis vector such that
$B_c\notin v_{\sI_c}+\sum_{d<c}\Z[q,q^{-1}]\cdot v_{\sI_d}$ where we let $\sI_d$
denote the stringy sequence for a crystal element $d$.  By minimality, we must
have $B_{c'}  \in v_{\sI_{c'}}+\sum_{d'<c'}\Z[q,q^{-1}]\cdot
v_{\sI_d'}$.  Then we have \[E_i^{(n)}B_{c'} \in  E_i^{(n)}v_{\sI_{c'}}+\sum_{d'<c'}\Z[q,q^{-1}]\cdot
E_i^{(n)}v_{\sI_d'}.\]  Now, by definition, $E_i^{(n)}v_{\sI_{c'}}=v_{\sI_c}$
and the use of lexicographic ordering implies that \[\sum_{d'<c'}\Z[q,q^{-1}]\cdot
E_i^{(n)}v_{\sI_d'}\subset \sum_{d<c}\Z[q,q^{-1}]\cdot v_{\sI_d}. \]  Thus, we must have
that \[B_c-v_{\sI_c}=(B_c-E_i^{(n)}B_{c'})+(E_i^{(n)}B_{c'}-v_{\sI_c})\in \sum_{d<c}\Z[q,q^{-1}]\cdot v_{\sI_d}\] which
is a contradiction.  This shows that Lusztig's basis is also canonical
in our sense.
\end{proof}

Orthodox bases and canonical bases sometimes coincide, and sometimes
do not.  For future comparison between them, we'll note that the only
issue is the condition {\it (III)}:
\begin{proposition}\label{I-II}
  The bases $o_P$ satisfy conditions {\it (I)} and {\it (II)} of
  Definition \ref{def:canonical}
  for the orthodox pre-canonical structure.
\end{proposition}
\begin{proof}
  Condition {\it (I)} is clear from the definition.  Condition {\it (II)} follows
  immediately from Propositions \ref{bla-unique} \& \ref{unique}.
\end{proof}

Furthermore, there is at least one property in which orthodox
bases are an improvement over canonical bases.

\begin{proposition}
  For any $\K$ and $Q_{ij}$, the structure coefficients of
  multiplication in $\dot{U}$ and matrix coefficients for the action
  on $V_\bla^\Z$ where the orthodox basis is used in $\Z_{\geq
    0}[q,q^{-1}]$.  
\end{proposition}
Similarly, by Corollary \ref{standard-strat}, we have that:
\begin{proposition}\label{standard-positive}
   For any $\K$ and $Q_{ij}$, the coefficients of any orthodox basis
   vector for $V_\bla^\Z$ in terms of pure tensor products of orthodox
   basis vectors in the factors has coefficients in $\Z_{\geq
    0}[q,q^{-1}]$.
\end{proposition}

Very loosely, the difference between orthodox and canonical bases is
to trade off positivity in coefficients for positivity in exponents of
$q$; Lusztig's basis is defined in a way that depends
strongly on the latter, at the cost of positivity of coefficients in
the non-symmetric case. 

In fact, the dependence of this basis on the base ring $\K$ is quite crude; the
corresponding basis only depends on the characteristic of the residue field.
\begin{theorem} \label{th:overfield}
For any overfield $\mathbb{R}\supset \mathbbm{r}\cong
\K/\mathfrak{m}$, the orthodox basis of $\hl^\bla$ or $\dU$ for $\K$ coincides with that for $\mathbb{R}$.
\end{theorem}
\begin{proof}
  First, we consider the reduction functor $R\colon \dU_\K\to
  \dU_{\mathbbm{r}}$; this sends indecomposables to indecomposables,
  since $\End(R(P))\cong \End(P)/\mathfrak{m}$, which sends graded
  local rings to graded local rings.  Thus, we can reduce to the case
  where $\K=\mathbbm{r}$ for $\dU$, and by the same proof for
  $\hl^\bla$.

  By Corollary \ref{absolutely-indecomposable} and Proposition \ref{unique}, each
  indecomposable object in $\dU$ or $\hl^\bla$ is absolutely
  indecomposable; thus it remains indecomposable on base extension to
  $\mathbb{R}$, so we have the same orthodox basis.
\end{proof}
Thus, we can assume that $\K$ is generated the coefficients of
$Q_{ij}$ over a prime field. If these coefficients are integers, then we need
only consider the prime fields.

  For $U_q^+(\widehat{\mathfrak{sl}}_n)$, the orthodox bases of
  the basic representation $V_{\omega_0}$ (for the
  choice of $Q_{ij}$ fixed in \cite[\S \ref{m-sec:type-A}]{Webmerged}) over
  $\K=\mathbb{F}_p$ were defined by Grojnowski as ``$p$-canonical
  bases'' \cite[\S 14.1]{Groslp}.  The equivalence of our approach and
  Grojnowski's is shown by the ``Main Theorem'' of Brundan and
  Kleshchev \cite{BKKL}.  This provides a wealth of examples where
  orthodox and canonical bases do not coincide.  

  \begin{example}\label{ex:affine}
    Perhaps the easiest example is when
    $\fg=\widehat{\mathfrak{sl}}_2$; in this case, we have chosen the
    polynomial $Q_{01}(u,v)=u^2-2uv+v^2$.  Consider the object in
    $\tU$ given by $\eF_1\eF_0\eF_1\eF_0$.  This has only a
    2-dimensional space of degree 0 endomorphisms, spanned by
\[ \tikz{  
\node[label=left:{$1=$}] at (-3,0){
\tikz[very thick]{\draw (0,0) -- node [above, at end]{$0$} node [below, at start]{$0$} (0,1); \draw (.5,0) -- node [above, at end]{$1$}  node [below, at start]{$1$} (.5,1); \draw (1,0)
  -- node [above, at end]{$0$}   node [below, at start]{$0$} (1,1); \draw (1.5,0) -- node [above, at end]{$1$}  node [below, at start]{$1$} (1.5,1);  
}
};
\node[label=left:{$\psi_2\psi_3\psi_1\psi_2=$}] at (3,0){
\tikz[very thick]{\draw (0,0) -- node [above, at end]{$0$}  node [below, at start]{$0$} (1,1); \draw (.5,0) -- node [above, at end]{$1$}  node [below, at start]{$1$} (1.5,1); \draw (1,0)
  -- node [above, at end]{$0$}  node [below, at start]{$0$} (0,1); \draw (1.5,0) -- node [above, at end]{$1$}  node [below, at start]{$1$} (.5,1);
}
};
}
\]
One can easily calculate that $(\psi_2\psi_3\psi_1\psi_2)^2=2
\psi_2\psi_3\psi_1\psi_2$.  

Thus, if $2$ is a unit in the ring $\K$,
we have that $\nicefrac{1}{2}( \psi_2\psi_3\psi_1\psi_2)$ is a
primitive idempotent, and $\End_0( \eF_1\eF_0\eF_1\eF_0)\cong \K
\oplus \K$ (so this object is the sum of two distinct summands).  On
the other hand, if $\K$ has characteristic $2$, then the same
calculation shows that $\psi_2\psi_3\psi_1\psi_2$ is nilpotent.  
We have an isomorphism $\End_0( \eF_1\eF_0\eF_1\eF_0)\cong
\K[t]/(t^2)$, so this object is indecomposable.

This example equally shows the dependence of the orthodox basis on the
choice of $Q_{ij}$: for any ring $\K$, if $Q_{01}(u,v)=u^2+v^2$, then $\eF_1\eF_0\eF_1\eF_0$ is
indecomposable, as in the characteristic 2 case.  Note that in this
case, the diagram $\psi_3\psi_1\psi_2$ gives a degree -2 map from
$\eF_1\eF_0\eF_1\eF_0$ to $\eF^{(2)}_1\eF_0^{(2)}$, which is
indecomposable for any choice of $Q_{01}$.  This shows that there is a
negative degree map between indecomposable $\tilde\psi$-invariant projectives.

The same example is treated by Tingley and the author in \cite[\S
3.5]{TW} and by Kashiwara
in \cite[Example 3.3]{Kashnote} from the dual
perspective (in terms of simples rather than projectives).  A
variant displaying similar behavior
was considered by Khovanov and Lauda in \cite[3.25]{KLI}.
  \end{example}

Using familiar methods from modular representation theory, we can relate orthodox bases for the same
algebra or module from characteristic 0 and 
characteristic $p$.  Fix polynomials  $Q_{ij}\in
\Z[u,v]$, and let $m=\operatorname{lcm}(t_{ij})$.
\begin{proposition}
Assume $p\nmid m$.
  Each $\mathbb{F}_p$-orthodox basis vector for the reduction of
  $Q_{ij}$ mod $p$ is a positive linear
  combination of $\mathbb Q$-orthodox basis vectors for $Q_{ij}$.
\end{proposition}
Of course, this theorem is easily generalized to values of $Q_{ij}$ lying in
the algebraic integers, replacing $\mathbb{F}_p$ by a larger finite field.
\begin{proof}
For simplicity, we only discuss $\dU$; the case of $\hl^\bla$ is
precisely the same.

We use the category $\dU$ over the ring $\K=\Z_p$, the $p$-adic
integers as a bridge between characteristics $p$ and $0$.

As usual, we have extension and reduction
functors \[\dU_{\Q_p}\overset{E}\longleftarrow
\dU_{\Z_p}\overset{R}\longrightarrow \dU_{\mathbb{F}_p}.\] As noted in
the proof of Lemma \ref{stringy-summand}, the functor $R$ sends indecomposables to indecomposables.  Thus, the $\mathbb{F}_p$-orthodox basis
coincides with the $\Z_p$-orthodox basis.  On the other hand, if $M$
is an indecomposable 1-morphism of $\dU_{\Z_p}$, then its extension of
scalars $E(M)$ is a sum of indecomposables in $\dU_{\Q_p}$.  The result
follows.
\end{proof}

This categorification framework provides us with a wealth of bases, which are actually quite
difficult to study in general.  As mentioned above, Brundan and
Kleshchev \cite{BKKL} have shown that the simple modules
over the symmetric groups over a field of characteristic $p$ give the
dual orthodox basis of the basic representation of $\mathfrak{sl}_p$
over the field $\mathbb{F}_p$;
the determination of these classes is one of the most important
questions in modular representation theory.

Similarly, in finite type, examples were recently described by
Williamson \cite{WilJames} where the canonical and orthodox bases do not coincide; in fact, for any prime $p$, the orthodox basis of $U(\mathfrak{sl}_{8p-1})$
from characteristic $p$ differs from the canonical basis.

Of course, if we are given any representation which seems to have a
natural choice of categorification, we can use this to define an
orthodox basis.  At the moment, the most obvious example is when
$\mathfrak{\widehat{sl}}_e$, and the associated representation is a
higher level Fock space.  As shown by Shan \cite{ShanCrystal}, the
category $\cO$'s of symplectic reflection algebras provide one such
categorification.  By recent independent work of the author
\cite{WebRou}, Rouquier, Shan, Varagnolo and Vasserot \cite{RSVV} and
Losev \cite{LoVV}, these have a graded lift where the classes of
projectives give a canonical basis; in \cite{WebRou}, we also give a
diagrammatic description of these categories more in the philosophy of
this paper.

\section{Canonical bases}
\label{sec:canonical-bases}

In this section, we consider the question of when the orthodox and
canonical basis coincide.

\begin{definition}
  We call an orthodox basis of $\dot{U}$ or $V_\bla^\Z$ {\bf
    canonical} if its elements are {\bf almost orthonormal}, that is
  $\langle o_P,o_{P'}\rangle\in
  \delta_{P,P'}+q^{-1}\Z_{\geq 0}[[q^{-1}]]$ for all $P,P'\in C$ or
  $C_\bla$.  That is, when the orthodox basis is canonial in the sense
  of Section \ref{sec:pre-canon-struct}.  
\end{definition}
Note that this is {\it a priori} stronger than {\it (III)}, but in
fact, for any orthodox basis, we have $\langle o_P,o_{P'}\rangle\in
  \Z_{\geq 0}((q))$, so in fact, it is equivalent. By Theorem
  \ref{thm:lusztig}, this is also the same as requiring the orthodox
  bases to match Lusztig's basis in the cases (1-3) we discussed
  there.
\begin{remark}
  We should note that the methods of Kashiwara \cite[3.1-2]{Kashnote}
  can be easily extended to show that if the orthodox basis of a
  $\dot{U}$ or $V^{\Z}_\bla$ is canonical for some value of $Q_{*,*}$
  over $\K$, then the
  same holds for generic $Q_{*,*}$, that is, when we take the coefficients of
  $Q_{*,*}$ to be formal variables on a space $\mathfrak{Q}$ of
  possible choices, and replace $\K$ by $\K(\mathfrak{Q})$.
\end{remark}
\excise{
\begin{proposition}
  The orthodox basis of $\dot{U}$, $V_{-\la}^\Z$, $V_\mu^\Z$ or
  $V_{-\la,\mu}^\Z$ is Lusztig's canonical basis if and only if it is
  canonical in the sense defined above.
\end{proposition}
\begin{proof}
In any of the cases above, the result \cite[26.3.1]{Lusbook} shows that Lusztig's
basis is canonical. On the other hand, Theorem
\ref{canonical-sign} shows that any element of a
canonical basis for a precanonical structure and inner product
agreeing with Lusztig's is $\pm 1$ times one of Lusztig's canonical basis vectors.

Thus, we need only rule out that $o=-b_c$ for some $c$. Assume $b_c$ is
the minimal element of the canonical basis which is the negative of an
orthodox basis vector; this is not $1_\nu$, since this is both a
canonical and an orthodox vector.  Then, there is some $c'<c$ such
that $E_ib_{c'}\in f_c(q)b_c+\sum_{d}f_d(q) b_{d}$ or $F_ib_{c'}\in
g_c(q)b_c+\sum_{d}g_d(q) b_{d}$ for some $i$, with $f_c$ or $g_c$
having positive coefficients.  Since the same is true of orthodox
vectors, these vectors must coincide, not differ by a sign.

On the other hand, if {\it (III)} fails for some pair of elements of the
orthodox basis of $ \dot{U}$ or $V_{-\la,\mu}^\Z$, then \cite[26.3.1(a)]{Lusbook}  shows
that they cannot both be in Lusztig's canonical basis.  
\end{proof}
}
\begin{proposition}
  The orthodox basis  of $V_{-\la}^\Z$, $V_\mu^\Z$ or
  $V_{-\la,\mu}^\Z$ is a crystal basis if and only if it is canonical.
\end{proposition}
\begin{proof}
  Any of these modules has at most
  one bar-invariant crystal basis, which in these cases coincides with
  Lusztig's basis.  Since the orthodox basis is bar-invariant by
  Proposition \ref{I-II}, the result follows.
\end{proof}

It follows immediately from 
Lemma \ref{mixed-canonical} that: 
\begin{proposition}\label{prop:mixed}
  The orthodox basis of $\dot{U}$ or $V_\bla^\Z$ is canonical if and
  only if the categorification $\dU$ or $\hl^\bla$ is mixed in the
  sense of Definition \ref{mixed}.
\end{proposition}

Thus, the question of when orthodox bases are canonical reduces to
computing when categorifications are mixed. As suggested by the name,
typically this is proven using relations to geometry; one shows that
there is a functor with nice properties sending $\tilde
\psi$-invariant objects in one's categorification to perverse sheaves
on some space, and deduces positivity of the grading from the fact
that perverse sheaves are the heart of a $t$-structure.  This connection
with geometry holds in only certain situations.   In Example
\ref{ex:affine}, we showed that if $\fg=\mathfrak{\widehat{sl}}_2$ and
$Q_{01}(u,v)=u^2+v^2$, then the category $\tU$ cannot be mixed.

For the rest of the
paper, we will assume $\fg$ has symmetric Cartan matrix (so that we
may use quiver varieties), $\K$ is a field of
characteristic 0 (so we may use the Decomposition Theorem), and we fix a particular choice of $Q_{*,*}$, which
coincides with the choice used in \cite[\S 3.3]{VV} and \cite[\S
3.2.4]{Rou2KM}.  This choice is forced on us by geometry and is of the
following nature: we choose an orientation $\Omega$ on our Dynkin
diagram, let $\epsilon_{ij}$ denote the number of edges
oriented from $i$ to $j$, and
fix
\begin{equation}
Q_{ij}(u,v)=(-1)^{\ep_{ij}}(u-v)^{c_{ij}}.\label{eq:Q}
\end{equation}

Note that these hypotheses include $\fg=\mathfrak{\widehat{sl}}_2$,
but with $\pm Q_{01}(u,v)=u^2-2uv+v^2$.

Assuming these hypotheses, the result \cite[4.5]{VV} says (in different language) that 
\begin{theorem}\label{dom}
  The orthodox basis of $V_\la^\Z$ is
  canonical.
\end{theorem}
The proof of this fact can be rephrased in terms of Lemma
\ref{lem:full-essential}.  We apply Lemma \ref{J-mixed} to the
category $\EuScript{J}$ of constructible sheaves on a
moduli space of quiver representations {\it \`a la} Lusztig generated
by certain pushforwards.  Vasserot and Varagnolo show that associated
mixed category $\mathcal{J}$ is equivalent to the projective modules over
the KLR algebra \cite[3.6]{VV}, so the functor of dividing by the
cyclotomic ideal defines a functor $\mathcal{J}\to \hl^\la$
intertwining Verdier duality with the duality $\tilde\psi$, to which
we can apply Lemma
\ref{lem:full-essential}.

Very similar arguments were also used by the author and Stroppel in
the study of quiver Schur algebras \cite{SWschur}, in order to show
that the indecomposable projectives over these algebras correspond to
canonical bases of higher-level Fock spaces; in fact, Proposition
\ref{prop:mixed} and Lemma \ref{lem:full-essential} seem to be applicable
in essentially any categorical $\fg$-module yet dreamed up.  The
difficult part is to understand the relevant pre-canonical
structure in terms of previously understood representation theory.  

We can extend this proof to the tensor product of highest weight
representations using an extension of Vasserot and Varagnolo's
geometric techniques.  This relies on a more general result on certain
generalizations of KLR algebras called {\bf weighted KLR algebras}.
\begin{proposition}[\mbox{\cite[4.9]{WebwKLR}}]\label{wKLR}
 Assume $\fg$ has symmetric Cartan matrix, $\K$ is  a field of
  characteristic 0, and $Q_{*,*}$ is as
  in \eqref{eq:Q}. Then, the algebra $\tilde T^\bla_\mu$ defined in \cite{Webmerged} is isomorphic
  to the Ext-algebra of an object $Y$ in the constructible derived
  category of a moduli space of quiver representations, denoted
  $E_{\la-\mu}/G_{\la-\mu}'$ which is a sum of shifts of semi-simple perverse
  sheaves.  This isomorphism intertwines the duality $\tilde\psi$ for
  $\tilde T^\bla_\mu$-modules and Verdier duality on the constructible
  derived category.
\end{proposition}

\begin{theorem}\label{tensor}
  Assume $\fg$ has symmetric Cartan matrix, all $\la_i$ are dominant, $\K$ is  a field of
  characteristic 0, and $Q_{*,*}$ is as
  in \eqref{eq:Q}. Then, the orthodox basis of $V_{\bla}^\Z$ is canonical. 
\end{theorem}
\begin{proof}
  We apply Lemma \ref{lem:full-essential} with $\mathcal{C}$
  being the sums of shifts of the summands of $Y$, with morphisms
  given by Ext's in the constructible derived category and the grading
  given by homological grading.  This category is a 
  mixed humorous category by Lemma \ref{J-mixed}.

  Since this category is equivalent to the graded projective modules
  over $\tilde{T}^\bla_\mu$, dividing by the violating ideal defines a
  full and essentially surjective functor. Pre-composed with
  the inverse of the equivalence of Proposition \ref{wKLR}, this gives
  a functor $\mathcal{C}\to \hl^\bla$ which satisfies the assumptions
  of Lemma \ref{lem:full-essential} and thus shows that $\hl^\bla$ is
  mixed.  Lemma \ref{mixed-canonical} shows that the orthodox basis
  for $\hl^\bla$ is canonical.
\end{proof}

The reader familiar with Lusztig's construction of this basis for
tensor products might wonder where his standard basis, the pure tensor
products of canonical basis vectors in the factors, has disappeared
to.  Of course, by Theorem \ref{dom}, these are the same as the pure
tensor products of orthodox basis vectors, and thus are given by the
classes of standardizations of indecomposable projectives by
Proposition \ref{q-iso}; in the context of standardly stratified
categories, these would usually just be called the ``standards.''  We
can also define a pre-canonical structure using the same bar and form
as the orthodox pre-canonical structure, but using these pure tensor
products as a standard basis.  This precanonical structure will be
almost orthonormal under the hypotheses of Theorem \ref{tensor}. 

Thus by Lemma \ref{two-prime}, 
for {\it this} standard basis we could use condition {\it (II')} from
Section \ref{sec:pre-canon-struct} as the definition of the canonical
basis, as Lusztig does.  In terms of representation theory, the
coefficients of the canonical basis in terms of the pure tensors are
multiplicities of the standard filtration on indecomposable tensor
products.  These are obviously positive integer Laurent polynomials,
and condition {\it (II')} corresponds to the fact that only positive
grading shifts of standards occur.

Combining this observation with Proposition \ref{standard-positive},
we see that:
\begin{corollary}
  The canonical basis of $V_{\bla}^\Z$ or $V_{-\la,\mu}^\Z$ is a
  linear combination of pure tensors of canonical basis elements with
  coefficients in $\Z_{\geq 0}[q^{-1}]$.
\end{corollary}

\begin{lemma}\label{lem:any-Q}
  If $\fg$ is finite dimensional and simply-laced and $\K$ contains
  all roots of unity, then all choices of $Q_{*,*}$ result in
  equivalent categories $\tU$ and $\hl^\bla$.
\end{lemma}
\begin{proof}
  The argument is precisely that given in \cite[pg. 17]{KLII} for KLR
  algebras.  Since $\fg$ is simply-laced, the polynomial $Q_{*,*}$ is
  determined uniquely by $t_{ij}$.  We simply note that if the products $t_{ij}t_{ji}^{-1}$
  coincide with those for another choice $t_{ij}'$, then these
  algebras are isomorphic by a rescaling of the crossings between
  differently colored strands.  Furthermore, if we multiply $t_{ij}t_{ji}^{-1}$
  by a coboundary in $\K^*$, then we get an algebra isomorphic by
  rescaling like colored crossings and dots.  Since all 1-cocycles on
  a tree are coboundary, we are done.
\end{proof}

\begin{corollary}\label{Udot}
  Assume $\fg$ is finite dimensional and simply-laced, $\K$ is a field
  of characteristic 0, and $Q_{*,*}$ is arbitrary.  Then, the orthodox
  basis of $\dot{U}$ or $V_{\bla}^\Z$ is canonical.
\end{corollary}
\begin{proof}
First, we can replace $\K$ by its algebraic closure $\bar \K$ by
Theorem \ref{th:overfield}.
By Lemma \ref{lem:any-Q}, we can reduce to the case where $Q_{*,*}$ is as in
\eqref{eq:Q}; Theorem \ref{tensor} thus establishes the case of $V_{\bla}^\Z$.

 Now, we consider $\dot{U}$.  By Proposition \ref{prop:hl-tensor}, the orthodox bases of
  $V_{\la,\mu}^\Z$ and $V_{w_0\la,\mu}^\Z$ coincide.  By Theorem
  \ref{tensor}, the former coincides with the canonical basis as well,
  so the same is true of $V_{w_0\la,\mu}^\Z$.  
Since the canonical basis of $\dot{U}$ is uniquely determined by
the fact that it lands on the canonical basis of $V_{w_0\la,\mu}^\Z$
under the map $u\mapsto
  u\cdot(v_{-\la}\otimes v_\mu)$, Lemma \ref{asymptotic} shows that
  the categorification $\dU$ is mixed and
  the orthodox basis agrees with the canonical basis. 
\end{proof}

\bibliography{../gen}
\bibliographystyle{amsalpha}
\end{document}